\documentclass[leqno,11pt]{amsart}

\usepackage[colorlinks,pagebackref,hypertexnames=false]{hyperref}

\usepackage{amsmath,amsthm,amssymb,mathrsfs} 
\usepackage[alphabetic,backrefs]{amsrefs}
\usepackage{ae,aecompl}
\usepackage[margin=1in]{geometry} 
\usepackage{tikz}
\usepackage{url}

\usepackage[english]{babel}

\usepackage{bookmark}

\numberwithin{equation}{section}

\newcommand{\Sp}{{\rm Sp}}

\newcommand{\SU}{{\rm SU}}

\newcommand{\U}{{\rm U}}

\renewcommand{\epsilon}{\varepsilon}

\newcommand{\del}{\partial}

\renewcommand{\Im}{\mathop{\mathrm{Im}}}

\newcommand{\ind}{\mathop{\mathrm{index}}}

\renewcommand{\Re}{\mathop{\mathrm{Re}}}
\newcommand{\Hess}{\mathrm{Hess}}

\newcommand{\vol}{\mathrm{vol}}

\newcommand{\sign}{\mathrm{sign}}

\def\<{\mathopen{}\left<}
\def\>{\right>\mathclose{}}
\def\({\mathopen{}\left(}
\def\){\right)\mathclose{}}

\newtheorem{theorem}{Theorem}

\newtheorem{conjecture}[theorem]{Conjecture}
\newtheorem{corollary}[theorem]{Corollary}

\newtheorem{proposition}[theorem]{Proposition}
\newtheorem{lemma}[theorem]{Lemma}

\theoremstyle{definition}

\newtheorem{example}{Example}

\newtheorem{remark}[example]{Remark}

\theoremstyle{definition}

\newtheorem{definition}[example]{Definition}

\numberwithin{theorem}{section}
\numberwithin{example}{section}
\numberwithin{equation}{section}
\numberwithin{figure}{section}

\author{Jason D. Lotay} 
\address[Jason D. Lotay]{University of Oxford, U.K.}
\urladdr{\href{http://people.maths.ox.ac.uk/lotay/}{http://people.maths.ox.ac.uk/lotay/}}
\email{jason.lotay@maths.ox.ac.uk}

\author{Goncalo Oliveira} 
\address[Gon\c{c}alo Oliveira]{Universidade Federal Fluminense IME--GMA, Niter\'oi, Brazil}
\urladdr{\href{https://sites.google.com/view/goncalo-oliveira-math-webpage/home}{https://sites.google.com/view/goncalo-oliveira-math-webpage/home}}
\email{{galato97@gmail.com}}

\title[Special Lagrangians, LMCF and the Gibbons--Hawking ansatz]{Special Lagrangians, Lagrangian mean curvature flow and the Gibbons--Hawking ansatz} 

\date{}

\begin{document}

\begin{abstract}
The Gibbons--Hawking ansatz provides a large family of circle-invariant hyperk\"ahler 4-manifolds, and thus Calabi--Yau 2-folds.  In this setting, we prove versions of the Thomas conjecture on existence of special Lagrangian representatives of Hamiltonian isotopy classes of Lagrangians, and the Thomas--Yau conjecture on long-time existence of the Lagrangian mean curvature flow.   
We also make observations concerning closed geodesics, curve shortening flow and minimal surfaces.
\end{abstract}

\maketitle
\tableofcontents

\section{Introduction}

\subsection*{Context}

In Calabi--Yau manifolds, which are K\"ahler and Ricci-flat, there is a distinguished class of submanifolds known as special Lagrangians, which are important in mathematics and theoretical physics, particularly in relation to Mirror Symmetry.  Special Lagrangians are Lagrangian and calibrated, thus volume-minimizing in their homology class. Therefore, special Lagrangians 
are of significant interest from both the symplectic and Riemannian viewpoints.  A central question is whether or not a Lagrangian admits a (unique) special Lagrangian in its Hamiltonian isotopy class. From the variational perspective, this is related to the existence and uniqueness of a volume-minimizer in the given class of Lagrangians.  Answering this question is also crucial for the construction of putative enumerative invariants for Calabi--Yau manifolds (see e.g.~\cite{JoyceCounting}).

Loosely speaking, under Mirror Symmetry, Lagrangians are conjectured to correspond to holomorphic connections on a complex bundle and special Lagrangians to Hermitian--Yang--Mills connections.  
Motivated by this, together with the Donaldson--Uhlenbeck--Yau theorem relating Hermitian--Yang--Mills connections and stable bundles, Thomas \cite{Thomas} conjectured that a unique special Lagrangian exists in a given Hamiltonian isotopy class of Lagrangians if and only if a certain stability condition holds.  This conjecture recasts the challenging nonlinear PDE problem for existence of special Lagrangians into an alternative topological question of stability of a Hamiltonian isotopy class of Lagrangians. 

Given that special Lagrangians are volume-minimizing, a natural approach to studying them is to use the gradient flow for the volume functional, namely the mean curvature flow.  In fact, mean curvature flow in Calabi--Yau manifolds preserves the Lagrangian condition \cite{Smoczyk}, leading to the notion of Lagrangian mean curvature flow, whose critical points are special Lagrangians.
  Given the success of Hermitian--Yang--Mills flow in studying the existence problem for Hermitian--Yang--Mills connections, together with the Mirror Symmetry considerations discussed above, 
Thomas--Yau \cite{ThomasYau} conjectured that certain stability conditions for a Hamiltonian isotopy class of a given Lagrangian should imply the long-time existence and convergence of Lagrangian mean curvature flow starting at the given Lagrangian.  Moreover, the flow should converge to the unique special Lagrangian from the original Thomas conjecture.  This motivates studying the relationship between the stability conditions in the Thomas--Yau conjecture and the behaviour of the Lagrangian mean curvature flow, particularly in light of the ground-breaking work of Neves \cite{NevesSingularities}, which shows that finite-time singularities form in the flow for any Hamiltonian isotopy class.  It is further conjectured that there is a deeper connection between Lagrangian mean curvature flow and stability, namely to Bridgeland stability conditions on Fukaya categories \cite{JoyceConjectures}.

\subsection*{Main results} 
A Riemannian 4-manifold $(X^{4},g)$ is hyperk\"ahler if it is equipped with three compatible symplectic structures $(\omega_1,\omega_2,\omega_3)$ such that the associated almost complex structures $(I_1,I_2,I_3)$ satisfy the quaternionic relations. These conditions force $(I_1,I_2,I_3)$ to be integrable, and the holonomy group of the Levi-Civita connection  to be contained in $\Sp(1)\cong\SU(2)$.  Thus, hyperk\"ahler 4-manifolds are essentially the same as Calabi--Yau 2-folds:   we see that $(g,I=I_1,\omega=\omega_1)$ is a K\"ahler structure, $g$ is Ricci-flat and $\Omega=\omega_2+i\omega_3$ is a holomorphic volume form of constant norm. 

The Gibbons--Hawking ansatz provides a large family of hyperk\"ahler 4-manifolds, including the well-known Eguchi--Hanson and Taub--NUT metrics, and describes all hyperk\"ahler $4$-manifolds admitting a tri-Hamiltonian circle action \cite{Bielawski}.  In particular, one obtains infinite families of ALE (asymptotically locally Euclidean) and ALF (asymptotically locally flat) gravitational instantons: complete hyperk\"ahler 4-manifolds whose Riemann curvature has finite $L^2$ norm, with maximal (i.e.~quartic) or cubic volume growth in the ALE or ALF cases respectively.   Hyperk\"ahler 4-manifolds arising from the Gibbon--Hawking ansatz, which we shall view as Calabi--Yau 2-folds $(X,\omega,\Omega)$, therefore provide a fertile testing ground for the Thomas and Thomas--Yau conjectures.
The key object in the stability conditions in these conjectures is the Lagrangian angle, which we now define.   

\begin{definition}\label{dfn:Lag.intro}
An oriented surface $L^2$ in $(X^4,\omega,\Omega)$ is Lagrangian if $\omega|_L\equiv 0$, and its Lagrangian angle $e^{i\beta}:L\to\mathbb{S}^1$ is defined by the condition
\begin{equation}\label{eq:Lag.angle}
e^{-i\beta}\Omega
|_L=\vol_L,
\end{equation}
where $\vol_L$ is the Riemannian volume form on $L$ with respect to the induced metric $g|_L$.  A special Lagrangian is then an oriented Lagrangian whose Lagrangian angle is constant, or, equivalently, calibrated by $\text{Re}(e^{-i\beta}\Omega
)
$ for a constant $\beta\in\mathbb{R}$.  

An oriented Lagrangian $L$ is therefore Hamiltonian isotopic to a special Lagrangian only if it is zero Maslov, i.e.~there is a function $\beta:L\to\mathbb{R}$ (called a grading of $L$) so that $e^{i\beta}$ is the Lagrangian angle.  We say that a zero Maslov Lagrangian $L$ is almost calibrated if there is a grading $\beta$ of $L$ whose variation is less than $\pi-\delta$ for some $\delta>0$.  If $L$ is compact, the almost calibrated condition is equivalent to saying   there is a constant $\beta_0$ so that $\text{Re}(e^{-i\beta_0}\Omega)|_L$ is a volume form on $L$.
\end{definition}

\noindent  Lagrangian mean curvature flow of a zero Maslov Lagrangian stays within its Hamiltonian isotopy class, and in this case the evolution of the Lagrangian angle is given by the heat equation, which means that the almost calibrated condition is  preserved by the flow.  The almost calibrated condition is geometrically natural and appears in several contexts, e.g.~\cites{Donaldson, LambertLotaySchulze, Solomon, Thomas, Wang}.

The stability conditions in the Thomas and Thomas--Yau conjectures can   be roughly phrased as follows. A compact zero Maslov Lagrangian $L$ is unstable if it is Hamiltonian isotopic to a graded Lagrangian connect sum $L_1\# L_2$ of two   zero Maslov Lagrangians satisfying a certain global condition on their Lagrangian angles; otherwise, $L$ is stable.  The global conditions on the Lagrangian angles of $L_1,L_2$ relate to their total variation or, roughly speaking, to their average values over $L,L_1,L_2$.  For the precise definitions we refer to Definitions \ref{dfn:stability} and \ref{dfn:flow_stability}, but we point out that the graded Lagrangian connect sums $L_1\# L_2$ and $L_2\# L_1$ are, in general, not Hamiltonian isotopic.

\subsubsection*{Thomas conjecture} 
We now state our first main result, which proves a version of the Thomas conjecture \cite[Conjecture 5.2]{Thomas}. 
The notion of stability used here is given in Definition \ref{dfn:stability}.  Recall that  hyperk\"ahler 4-manifolds given by the Gibbons--Hawking ansatz have a circle symmetry, and our results concern  Lagrangians invariant under this circle action.

\begin{theorem}\label{thm:Intro.SL.stable}
	Let $X^4$ be an ALE or ALF manifold arising from the Gibbons--Hawking ansatz and let $L^2\subseteq X^4$ be a compact, embedded, zero Maslov, circle-invariant  Lagrangian. 
	Then, $L$ can be isotoped via a circle-invariant Hamiltonian to a 
 special Lagrangian if and only if it is stable with respect to circle-invariant Hamiltonian isotopies. In this situation, the special Lagrangian in the circle-invariant Hamiltonian isotopy class of $L$ is unique.
\end{theorem}

\subsubsection*{Thomas--Yau conjecture} Our second main result proves a version of the Thomas--Yau conjecture \cite[Conjecture 7.3]{ThomasYau}.  Since the notion of stability used here, given in Definition \ref{dfn:flow_stability}, is different to that arising in Theorem \ref{thm:Intro.SL.stable}, we refer to it as flow stability for clarity.

\begin{theorem}\label{thm:Intro.LMCF.stable}
	Let $X^4$ be an ALE or ALF hyperk\"ahler 4-manifold arising from the Gibbons--Hawking ansatz and let $L^2\subseteq X^4$ be a compact, embedded, almost calibrated, circle-invariant Lagrangian. If $L$ is flow stable, the Lagrangian mean curvature flow starting at $L$ exists for all time and converges smoothly to the unique circle invariant special Lagrangian given by Theorem \ref{thm:Intro.SL.stable}.
\end{theorem}

\noindent For some of the manifolds considered in Theorem \ref{thm:Intro.SL.stable}, Thomas--Yau (c.f.~\cite[Theorem 7.6]{ThomasYau}) proved a version of the Thomas--Yau conjecture, but for a different, non-Ricci-flat, K\"ahler metric, and therefore not for the Lagrangian mean curvature flow,  but a modified flow instead (which is, in fact, the Maslov flow of \cite{LotayPacini}).  Whilst their arguments can probably be adapted to the genuine Calabi--Yau metric and Lagrangian mean curvature flow, they also require a  stronger assumption on the Lagrangian angle than almost calibrated for their result to hold.  

\begin{remark}
Open sets in the moduli space of Ricci-flat K\"ahler metrics on a K3 surface can be constructed via gluing methods, which include using ALE and ALF manifolds given by the Gibbons--Hawking ansatz as local models (see e.g.~\cites{DonaldsonGluing,FoscoloGluing}).  In particular, neighbourhoods of special Lagrangian 2-spheres in these K3 surfaces are well approximated by Gibbons--Hawking metrics.  We therefore expect that versions of Theorems \ref{thm:Intro.SL.stable} and \ref{thm:Intro.LMCF.stable} hold in these neighbourhoods via modification of the techniques presented here.  
\end{remark}

\subsubsection*{Eguchi--Hanson and multi-Taub--NUT on $T^*S^2$} As an immediate application of our results, suppose on $T^*S^2$ we take the Eguchi--Hanson metric in Example \ref{ex:Eguchi-Hanson} or the multi-Taub--NUT metric with $k=2$ as in Example \ref{ex:MultiTN}.  Then, for any compatible Calabi--Yau structure on $T^*S^2$ (i.e.~inducing these metrics) so that the zero section is special Lagrangian, the stability conditions in Theorems  \ref{thm:Intro.SL.stable} and \ref{thm:Intro.LMCF.stable} are vacuous, so we have the following result.

\begin{corollary}\label{cor:Intro.T*S2}
Let $T^*S^2$ be endowed with the Eguchi--Hanson metric or the multi-Taub--NUT metric and choose a compatible Calabi--Yau structure  on $T^*S^2$ so that the zero section $S^2$ is special Lagrangian.  Let $L\subseteq T^*S^2$ be a compact, embedded, zero Maslov, circle-invariant Lagrangian.  There is a circle-invariant Hamiltonian isotopy from $L$ to $S^2$, which is the unique special Lagrangian in $T^*S^2$, and if $L$ is almost calibrated then Lagrangian mean curvature flow starting at $L$ exists for all time and converges to $S^2$.
\end{corollary}

\noindent The long-time existence and convergence of Lagrangian mean curvature flow in Corollary \ref{cor:Intro.T*S2} is striking in that no assumption other than almost calibrated is required, so $L$ need not even be $C^0$-close to the zero section.  In particular, this is much stronger than the main flow results in \cites{LotaySchulze,TsaiWangFlows} for the special case of the Eguchi--Hanson metric and invariant Lagrangian initial condition.  Moreover, Corollary \ref{cor:Intro.T*S2} is sharp, since \cite{NevesSingularities} shows that without the almost calibrated assumption the flow fails, in that a finite-time singularity occurs,  even assuming that $L$ is circle-invariant.

\subsection*{Outline of the proofs of Theorems \ref{thm:Intro.SL.stable} and \ref{thm:Intro.LMCF.stable}}  We now briefly describe the main steps in the proofs of our main results.

\subsubsection*{Projection} The hyperk\"ahler 4-manifolds $X$ given by the Gibbons--Hawking ansatz have a projection 
\begin{equation}\label{eq:Moment_Map_Projection}
\mu=(\mu_1,\mu_2,\mu_3):X \to \mathbb{R}^3,
\end{equation}
which is, in fact, the hyperk\"ahler moment map of the circle action. Away from a discrete set of points $S \subset \mathbb{R}^3$ (which is finite for ALE and ALF gravitational instantons), \eqref{eq:Moment_Map_Projection} is  a circle bundle, while $\mu^{-1}(p)$ is a point for each $p\in S$.  The hyperk\"ahler metric on $X$ is determined by a harmonic function $\phi$ on $\mathbb{R}^3$ with singularities at the points in $S$. 

\subsubsection*{Curves} Any circle-invariant surface $L$ in $X$ is the pre-image of a curve $\gamma$ in $\mathbb{R}^3$. For instance, the pre-image of curve with endpoints in $S$ (and otherwise not containing points in $S$) is a  2-sphere, whilst the pre-image of a simple closed curve not intersecting $S$ is a 2-torus.  Moreover, the circle-invariant surface $L$ is Lagrangian with respect to some member of the hyperk\"ahler family of symplectic forms on $X$ if and only if $\gamma$ is planar (c.f.~Lemma \ref{lem:Lagrangian}). 

Up to an overall translation and rotation of $\mathbb{R}^3$, it suffices to consider circle-invariant Lagrangians 
which are lifts of curves $\gamma\subset\mathbb{R}^3$ lying in the plane $P$ where $\mu_3=0$. Since we restrict to oriented Lagrangians,  $\gamma$ comes equipped with an orientation, and we let $\gamma'$  denote the tangent velocity vector of any oriented parametrization of $\gamma$.

\subsubsection*{Stability} In this setting, all compact, zero Maslov, circle-invariant Lagrangians $L$ are of the form $\mu^{-1}(\gamma)$ for $\gamma$ a curve in the plane $P$ with endpoints in $S$.  Moreover, $L$ is embedded if and only if $\gamma$ meets no other points of $S$, and thus $L$ is a 2-sphere.  We show that a grading $\beta$ of $L$ can be identified with the angle that $\gamma'$ makes with a specific line in the plane $P$.  This enables us to relate the stability conditions arising in Theorems \ref{thm:Intro.SL.stable} and \ref{thm:Intro.LMCF.stable} to properties of the curve $\gamma$.

For Theorem \ref{thm:Intro.SL.stable}, we show that $L=\mu^{-1}(\gamma)$ is stable in the sense of the statement if and only if $\gamma$ can be isotoped through planar curves to the straight line $\ell$ connecting the endpoints $p_1,p_2\in S$ of $\gamma$.  Moreover, we show that $\mu^{-1}(\ell)$ is the unique circle-invariant special Lagrangian in the homology class $[L]$.  Theorem \ref{thm:Intro.SL.stable} then follows.

\subsubsection*{Flows} For Theorem \ref{thm:Intro.LMCF.stable}, one starts with a Lagrangian $L$ of the form $\mu^{-1}(\gamma)$ as before, which is now almost calibrated and flow stable.  
The key observation is that mean curvature flow $L_t$ starting at $L_0=L$ descends through $\mu$ to a modified curve shortening flow $\gamma_t$ starting at $\gamma_0=\gamma$: this modified flow depends on the harmonic function $\phi$ defining the hyperk\"ahler metric.  

 We show that $\gamma_t$ remains in the plane $P$ (which implies $L_t$ continues to be Lagrangian).  Given that the area of $L_t$ decreases, and the Lagrangian angle satisfies the heat equation, we show that flow stability of $L$ implies that  $\gamma_t$ must stay away from all other points of $S$ other than the endpoints $p_1,p_2$ of $\gamma$ (which are fixed by the flow).  Using blow-up analysis for both the flows $L_t$ and $\gamma_t$, together with a recent classification of ancient solutions to Lagrangian mean curvature flow in \cite{LambertLotaySchulze}, we show that no finite-time singularities can occur along the flow: it is here that one crucially needs the almost calibrated condition.  The convergence of the flow $\gamma_t$ to the straight line $\ell$ connecting $p_1$ and $p_2$, and thus $L_t$ to the special Lagrangian $\mu^{-1}(\ell)$ quickly follows, and hence Theorem \ref{thm:Intro.LMCF.stable} is proved.

\begin{remark}
 In proving Theorems \ref{thm:Intro.SL.stable} and \ref{thm:Intro.LMCF.stable}, we use that the harmonic function $\phi$ in the Gibbons--Hawking ansatz has finitely many singularities, which restricts us to the ALE and ALF gravitational instantons.  However, one may be able in certain cases to obtain similar, perhaps weaker, results when there are infinitely many singularities, e.g.~for the (incomplete) Ooguri--Vafa metric in Example \ref{ex:OV} and the (complete) Anderson--Kronheimer--LeBrun metrics in Example \ref{ex:AKL}.
\end{remark}

\subsection*{Summary}  We now briefly summarize the contents of this article.
\begin{itemize}
\item In Section \ref{sec:GH_Ansatz} we recall the Gibbons--Hawking ansatz and give several important examples of hyperk\"ahler 4-manifolds arising from it.  We also carry out some computations which will prove useful in the course of our work.  
\item Section \ref{sec:Geodesic_Orbits} investigates geodesic orbits of the circle action, which are therefore closed.  Closed geodesics are a classical object of study in Riemannian geometry, and all of the complete hyperk\"ahler 4-manifolds in $\S$\ref{sec:GH_Ansatz} are simply connected, so one is motivated to study the existence, number and properties of closed geodesics in this setting. 

We relate geodesic orbits to the  singularities of the  function $\phi$, and  prove  that if $\phi$ has $k \geq 2$ generically placed singularities, then there are at least $k-1$ geodesic orbits. We then give several examples: e.g.~for $k$ collinear singularities there are precisely $k-1$ such geodesics, the Ooguri--Vafa metric has a unique one, and the Anderson--Kronheimer--LeBrun metrics have infinitely many. Through examples, we also show that in the same   manifold there are hyperk\"ahler metrics with different numbers of   geodesic orbits, and study their index as critical points for length. We close  
 $\S$\ref{sec:Geodesic_Orbits} by studying  curve shortening flow of circle orbits   in several explicit examples, and relate our results to dynamics of monopoles on $\mathbb{R}^3$.

\item Section \ref{sec:Minimal_Surfaces} investigates circle-invariant minimal surfaces in Gibbons--Hawkings hyperk\"ahler 4-manifolds. We prove that every  degree 2 homology class is represented by a circle-invariant area-minimizing surface.   Since these surfaces generate the topology of gravitational instantons, one is motivated to study their uniqueness amongst minimal surfaces and the mean curvature flow of surfaces, for which the area-minimizing surfaces are natural critical points. 
  
Mean curvature flow in higher codimension is underdeveloped and poses a greater challenge than the hypersurface case.  In the circle-invariant setting, we show that mean curvature flow of surfaces reduces to a weighted curve shortening flow in $\mathbb{R}^3$. 
 We also show that, in many situations, the area-minimizing surfaces generating the second homology are locally isolated as minimal surfaces, and dynamically  stable under mean curvature flow.

\item In Section \ref{sec:Lagrangians} we investigate properties of circle-invariant Lagrangians in hyperk\"ahler 4-manifolds given by the Gibbons--Hawking ansatz.  We prove our first main result, Theorem \ref{thm:Intro.SL.stable}, which is a version of the Thomas Conjecture. We also discuss Jordan--H\"older filtrations and decompositions of graded Lagrangians in our context, and give a simple description of Seidel's symplectically knotted 2-spheres \cite{Seidel}, which we prove  are knotted in the circle-invariant setting by elementary observations, with no need for Lagrangian Floer homology.

\item In Section \ref{sec:Lagrangian_Mean_Curvature_Flow} we investigate the Lagrangian mean curvature flow in Gibbons--Hawking hyperk\"ahler 4-manifolds through the weighted curve shortening flow in $\mathbb{R}^2$ it induces.  We show that generic, embedded, circle-invariant, Lagrangian tori collapse to a circle orbit in finite time. We then prove our second main result, Theorem \ref{thm:Intro.LMCF.stable}, which is a version of the Thomas--Yau Conjecture.
\item In Appendix \ref{app:spheres}, we provide a sufficient condition so that the induced metric on a circle-invariant minimal sphere is positively curved. \item In Appendix \ref{app:hessian}, we compute the Hessian of any circle-invariant function, and deduce that there are no compact minimal submanifolds in Euclidean or Taub--NUT $\mathbb{R}^4$. 
\end{itemize}

\subsection*{Acknowledgements}

The first author would like to thank Ben Lambert for useful comments on some of the material in Section \ref{sec:Lagrangian_Mean_Curvature_Flow}. The second author thanks Grace Mwakyoma for help in making some of the figures in Section \ref{sec:Lagrangians}.  The first author was partially supported by Leverhulme Research Project Grant RPG-2016-174 during the course of this project.
The second author was supported by Funda\c{c}\~ao Serrapilheira 1812-27395, by CNPq grants 428959/2018-0 and 307475/2018-2, and FAPERJ through the program Jovem Cientista do Nosso Estado E-26/202.793/2019.

\section{The Gibbons--Hawking ansatz}\label{sec:GH_Ansatz}

\subsection{Definitions}\label{ssec:defns} We start this introductory section by recalling the Gibbons--Hawking ansatz.  We shall use the notation in this section throughout the article. Let $(X^4, g)$ be hyperk\"ahler and equipped with a circle action preserving the three symplectic (in fact, K\"ahler) forms $\omega_1,\omega_2,\omega_3$ associated with the three complex structures $I_1,I_2,I_3$ satisfying the quaternionic relations; i.e.~
$$I_1^2=I_2^2=I_3^2=I_1I_2I_3=-1
\quad\text{and}\quad\omega_i(\cdot,\cdot)=g(I_i\cdot,\cdot)$$
for $i=1,2,3$.  Denote by $\xi$ the infinitesimal generator of the $\U(1)$ action and let $\hat{X} \subset X $ be the open dense set where the action is free. Then, $\hat{X}$ can be regarded as a $\U(1)$-bundle 
over an open $3$-manifold $Y^{3}$, i.e.
$$\mu: \hat{X} \rightarrow Y^3.$$
We equip this bundle with a connection whose horizontal spaces are $g$-orthogonal to $\xi$, which we encode as $\eta \in \Omega^1(\hat{X}, \mathbb{R})$ such that $\ker(\eta)=\xi^{\perp}$ (identifying the Lie algebra of $\U(1)$ with $\mathbb{R}$). Then 
\begin{equation}\label{eq:xi}
 \eta(\xi)=1\quad\text{and}\quad\iota_{\xi} d \eta =0.
\end{equation}
Let $a:\hat{X}\to \mathbb{R}$ be a positive $\U(1)$-invariant function and define the 1-forms $\alpha_i:= I_i (a \eta)$, for $i=1,2,3$. The hyperk\"ahler metric $g$ may then be written on $\hat{X}$ as
\begin{equation}\label{eq:HypMetric}
g= a^2\eta^2+\alpha_1^2+\alpha_2^2+\alpha_3^2,
\end{equation}
and the associated symplectic forms are given by:
\begin{align}\label{eq:hypforms}
\omega_1 & =  a \eta \wedge \alpha_1 + \alpha_2 \wedge \alpha_3 ,\quad 
\omega_2 
=  a \eta \wedge \alpha_2 + \alpha_3 \wedge \alpha_1 ,\quad  
\omega_3 
=  a \eta \wedge \alpha_3 + \alpha_1 \wedge \alpha_2.
\end{align} 
Fixing the volume form $a\eta \wedge \alpha_{1}\wedge\alpha_{2}\wedge\alpha_{3}$, the forms $\omega_1,\omega_2,\omega_3$ give a trivialization of $\Lambda^2_+\hat{X}$, the bundle of self-dual 2-forms on $\hat{X}$. Conversely, if we define 2-forms $\omega_1,\omega_2,\omega_3$ as in \eqref{eq:hypforms} and fix the volume form $a\eta\wedge\alpha_1\wedge\alpha_2\wedge\alpha_3$, we can recover the metric $g$ as in \eqref{eq:HypMetric}, and it follows from \cite[Lemma 4.1]{Atiyah} that $g$ is hyperk\"ahler if and only if all the $\omega_i$ are closed. Using this characterization we shall now prove the following.

\begin{proposition}\label{prop:MainHK}
	The metric $g$ in \eqref{eq:HypMetric} equips $X$ with a hyperk\"ahler structure so that the $\U(1)$ action given by $\xi$ in \eqref{eq:xi} preserves $g$ and 
	$\omega_1,\omega_2,\omega_3$ in \eqref{eq:hypforms} if and only if the following hold.
	\begin{enumerate}
		\item[(a)] The symmetric $2$-tensor 
		$$g_E=a^2(\alpha_1^2+\alpha_2^2+\alpha_3^2)
				$$
		is the pullback of a flat metric on $Y^3$.
		\item[(b)] The pair $(\eta,\phi)$, 
		where $\phi=a^{-2}$, is a Dirac monopole on $Y^3$, i.e.
		\begin{equation}\label{eq:Bogomolnyi}
		\ast_E d\eta=- d\phi,
				\end{equation}
		where $\ast_E$ denotes the Hodge star operator associated with the metric $g_E$ 		on $Y^3$ from \emph{(a)}.
		
		\item[(c)] There are local coordinates $(\mu_1,\mu_2,\mu_3)$ on $Y^3$ such that $\alpha_i=\phi^{\frac{1}{2}}d\mu_i$, and the hyperk\"ahler metric can be written as
		\begin{equation}\label{eq:HypMetric2}
		g=\frac{1}{\phi}\eta^2 
		+\phi\left( d \mu_1^2  + d \mu_2^2 + d \mu_3^2 \right).
		\end{equation}
  	\end{enumerate}
\end{proposition}

Note that the metric on $Y^3$ so that $\mu:\hat{X}\to Y$ is a Riemannian submersion is $a^{-2}g_E$, which is conformally flat.  We also observe that $d(\phi \eta)$ is a self dual $2$-form on $(X^4 , g)$.

\begin{remark}
	If $Y^3$ is simply connected, then the $\U(1)$ action is hyperhamiltonian. In this case, the coordinates $(\mu_1,\mu_2,\mu_3)$ can be taken to be global and form the hyperk\"ahler moment map
	$$\mu: X \rightarrow \mathbb{R}^3.$$
	Up to a covering we can always assume $Y^3$ is simply connected and so an open set in $\mathbb{R}^3$. 
\end{remark}

\begin{proof}[Proof of Proposition \ref{prop:MainHK}]
	As we already remarked, the metric $g$ is hyperk\"ahler if and only if $d\omega_i=0$ for all $i=1,2,3$, and the $\U(1)$ action preserves $\omega_i$ if and only if $L_{\xi} \omega_i=0$, again for $i=1,2,3$. These   conditions imply that $d \iota_{\xi} \omega_i = d (a \alpha_i)=0$, for all $i$, which is the same as $g_E$ in (a) being flat. 
 
Using \eqref{eq:hypforms} and  the fact that the $d \omega_i$ vanish gives, for $(i,j,k)$ a cyclic permutation of $(1,2,3)$,
	\begin{equation*}
	a\alpha_i \wedge d \eta = 2 \frac{da}{a^3} \wedge a\alpha_j \wedge a\alpha_k.
	\end{equation*}
	This is easily seen to be the Bogomolnyi equation \eqref{eq:Bogomolnyi} with $ \phi=a^{-2}$, yielding (b). 
	The fact that $d ( a \alpha_i )=0$ implies, via the Poincar\'e lemma, the local existence of real valued functions $\mu_i$ on $X$, such that $a \alpha_i = d \mu_i$, for $i=1,2,3$.  By definition, the $\mu_i$ satisfy $\iota_{\xi} \omega_i= d \mu_i$ and so descend to local coordinates on $Y$.  Then (c) is proven by rewriting  \eqref{eq:HypMetric} in terms of $\phi$ and $\mu$.
\end{proof}

\subsection{Examples} We will now describe some key examples of hyperk\"ahler metrics which arise from the Gibbons--Hawking ansatz.

\begin{example}[Flat Metric]\label{ex:flat}
	Let $Y= \mathbb{R}^3 \backslash \lbrace 0 \rbrace$, let $r$ be the Euclidean distance from $0$ in $\mathbb{R}^3$, and 
	\begin{equation}\label{eq:phi.flat}
	 \phi= \frac{1}{2r}.
	\end{equation}
 Then $\hat{X}=\mathbb{R}^4\setminus\{0\}$ is the pullback to $\mathbb{R}^3 \backslash \lbrace 0 \rbrace $ of the Hopf bundle and $\eta$ is the unique $\SU(2)$-invariant connection on $\hat{X}$. In this case $g$ can be extended over $0$ where the circle action collapses, and we obtain the flat metric on $X=\mathbb{R}^4$.\footnote{Placing a charge $k \in \mathbb{N}$ other than $1$ at the origin, i.e.~taking $\phi=\frac{k}{2r}$, leads to  the same metric, but the connection form $\eta$ is $k$ times the connection form on the Hopf bundle.} Indeed, writing the metric in polar coordinates on $\mathbb{R}^3$ first with $r$ as the radial coordinate, and then changing to polar coordinates on $\mathbb{R}^4$ with radial coordinate $\rho=\sqrt{2r}$, we have 
	\begin{align*}
	g & =  2r \eta^2 + \frac{dr^2}{2r} + \frac{r}{2} g_{\mathbb{S}^2}  =  d\rho^2 + \rho^2 \left( \eta^2 + \frac{1}{4} g_{\mathbb{S}^2} \right) =  d\rho^2 + \rho^2 g_{\mathbb{S}^3}.
	\end{align*}
\end{example}

\begin{example}[Eguchi--Hanson]\label{ex:Eguchi-Hanson}
	Let $p \in \mathbb{R}^3\setminus\{0\}$, let $-p$ be its antipodal point and let $Y= \mathbb{R}^3 \backslash \lbrace p,-p \rbrace$. Then, setting
\begin{equation}\label{eq:phi.EH}
	\phi= \frac{1}{2 \vert x - p \vert_{g_E}} + \frac{1}{2 \vert x + p \vert_{g_E}}
	\end{equation}
	gives the Eguchi--Hanson metric. This metric will also extend smoothly by adding back two points. 
	
	This manifold contains a minimal $2$-sphere which we may think of in the following way. At the endpoints of the straight line joining $p$ to $-p$ the circle action collapses, but is free at any other point of the line. Hence, the preimage under $\mu$ of this line is a finite cylinder with the boundary circles collapsed to points, i.e.~a 2-sphere.

	 In fact, the Eguchi--Hanson space is $T^*S^2$ with the minimal $2$-sphere being the zero section.  It is straightforward to see that there are complex structures with respect to which it is either special Lagrangian or holomorphic.
	 
Moreover, since $\lim_{r\to \infty}\phi=0$ and $T^*S^2\setminus S^2\cong (0,\infty)\times\mathbb{RP}^3$, the Eguchi--Hanson metric is asymptotic to the flat metric on $\mathbb{R}^4/\{\pm1\}$, and so has maximal volume growth and is ALE (asymptotically locally Euclidean).  We  also see from \eqref{eq:phi.flat}--\eqref{eq:phi.EH} that $\lim_{p\to 0}\phi=\frac{1}{r}$, so the limit as $p$ tends to $0$ of the Eguchi--Hanson space is the flat orbifold $\mathbb{R}^4/\{\pm 1\}$. 
\end{example}

\begin{example}[Multi-Eguchi--Hanson]\label{ex:MultiEH}
	We can generalize Example \ref{ex:Eguchi-Hanson} by choosing $k\geq 2$ points $p_1, ..., p_k$ in $\mathbb{R}^3$, letting $Y= \mathbb{R}^3 \backslash \lbrace p_1, ... , p_k \rbrace$, and choosing
\begin{equation}\label{eq:phi.MultiEH}
	\phi= \sum_{i=1}^k\frac{1}{2 \vert x - p_i \vert_{g_E}} 
\end{equation}
	so as to obtain the multi-Eguchi--Hanson metric. 
	As in the Eguchi--Hanson case the metric extends smoothly over the points which were removed.
	
	  We can draw a collection of lines joining $p_1$ to $p_2$, then $p_2$ to $p_3$ and so on, and their preimages form a bouquet of $k-1$ minimal two spheres which generate $H_2(X, \mathbb{Z})$.  Again, for each 2-sphere, there are complex structures so that the 2-sphere is either complex or special Lagrangian.
	  
We still have $\lim_{r\to\infty}\phi=0$,  so   the multi-Eguchi--Hanson metric is asymptotic to the flat metric on $\mathbb{R}^4/\mathbb{Z}_k$, so is again ALE.  We also obtain a flat orbifold $\mathbb{R}^4/\mathbb{Z}_k$ in the limit as the $p_i$ tend to $0$.
\end{example}

\begin{example}[Taub--NUT]\label{ex:TN}
 Let $m>0$ be constant and set
\begin{equation}\label{eq:phi.TN}
 \phi= m +\frac{1}{2r}.
 \end{equation}
This can be completed by adding in a point at the origin, which topologically gives $X=\mathbb{R}^4$ again, and is called the Taub--NUT space.

Notice that taking $\phi=m$ in the Gibbons--Hawking ansatz gives $X=S^1\times\mathbb{R}^3$ with the product metric, where the size of the circle is governed by $m$ (tending to zero as $m\to\infty$).  Since $\lim_{r \rightarrow \infty} \phi=m$, 
we deduce that the Taub--NUT metric is asymptotic to the product metric on $S^1\times\mathbb{R}^3$, and hence has cubic volume growth and is ALF (asymptotically locally flat).

Comparing \eqref{eq:phi.flat} and  \eqref{eq:phi.TN}, we see that the Taub--NUT metric limits to the Euclidean metric on $\mathbb{R}^4$ as $m\to 0$.
\end{example}

\begin{example}[Multi-Taub--NUT]\label{ex:MultiTN}
	Let $p_1, ..., p_k$ be $k\geq 2$ points in $\mathbb{R}^3$ and $Y= \mathbb{R}^3 \backslash \lbrace p_1, ... , p_k \rbrace$. Then, letting $m>0$ be constant,  we obtain the multi-Taub--NUT metric by setting
\begin{equation}\label{eq:phi.MultiTN}
	\phi= m + \sum_{i=1}^k\frac{1}{2 \vert x - p_i \vert_{g_E}}.
\end{equation}
 This can also be smoothly extended by adding back the points on which the circle action degenerates, and it contains a bouquet of $k-1$ minimal 2-spheres as in Example \ref{ex:MultiEH}.

 As in Example \ref{ex:TN}, the multi-Taub--NUT metric is ALF.  Moreover, 
 in the limit as $m\to 0$ we have that $\phi$ in \eqref{eq:phi.MultiTN} becomes as in \eqref{eq:phi.MultiEH} in Example \ref{ex:MultiEH}, and so the multi-Eguchi--Hanson metric appears as a limit of multi-Taub--NUT metrics.  
\end{example}

\begin{example}[Ooguri--Vafa]\label{ex:OV}
There is a natural extension of the Taub--NUT/Multi-Taub--NUT metrics where $\phi$ has infinitely many collinear singularities so that the resulting metric becomes $2\pi$-periodic.  This metric is known as the Ooguri--Vafa metric and is of central importance in the study of degenerations of elliptically fibred K3 surfaces and its relation to the SYZ conjecture.  The description that we now make for this metric is adapted from \cite[$\S$3]{GrossWilson}.

 One can explicitly define the following function on $\mathbb{R}^3$:
\begin{equation}\label{eq:phi.OV}
\phi=m+\frac{1}{2 r}+\frac{1}{2}\sum_{k\in\mathbb{Z}\setminus\{0\}}^{\infty}\frac{1}{|x-(2\pi k,0,0)|}-\frac{1}{2\pi k},
\end{equation}
where $m>0$ is given by $\pi m=\log(4\pi)-\gamma$,
where $\gamma$ is Euler's constant, for convenience.  The function $\phi$ in \eqref{eq:phi.OV} is not positive on $\mathbb{R}^3$, but is positive on $\mathbb{R}\times B$ for some sufficiently small ball $B$ around $0$ in the $(x^2,x^3)$-plane.  Since it is $2\pi$-periodic in the $x^1$-direction by construction, $\phi$ is naturally defined on $\mathbb{S}^1\times B$, where the coordinate on the unit circle $\mathbb{S}^1\subseteq\mathbb{C}$ is $e^{ix^1}$.

Hence, one obtains an incomplete hyperk\"ahler metric on $X=\hat{X}\cup\{0\}$, where $\hat{X}$ is a $\U(1)$-bundle $\mu:\hat{X}\to (\mathbb{S}^1\times B)\setminus\{(1,0,0)\}$. 
\end{example}

\begin{example}[Anderson--Kronheimer--LeBrun]\label{ex:AKL}  There is also a natural extension of the Eguchi--Hanson/multi-Eguchi--Hanson metrics where $\phi$ has infinitely many singularities. Suppose we take infinitely many points $\lbrace p_i \rbrace_{i=0}^{+\infty}$ in $\mathbb{R}^3$ so that 
$$\sum_{i=1}^{\infty} \frac{1}{|p_i-p_0|} < \infty.$$
It is shown by Anderson--Kronheimer--LeBrun \cite{Anderson1989} that taking
$$\phi=\sum_{i=1}^\infty\frac{1}{2 \vert x - p_i \vert_{g_E}}$$
in the Gibbons--Hawking ansatz defines a complete hyperk\"ahler 4-manifold with infinite topology.
\end{example}

\subsection{Structure equations}\label{sec:Structure_Eqs}

To facilitate our later computations it will be useful to have the structure equations for the Levi-Civita connection of the hyperk\"ahler metric in the Gibbons--Hawking ansatz.  We can summarize our result as follows. 

\begin{lemma}\label{lem:Covariant_Derivatives}
	Let $\lbrace e_{a} \rbrace_{a=0}^3$ denote the orthonormal framing on $(\hat{X},g)$ whose dual coframing $\lbrace e^{a} \rbrace_{a=0}^3$ is given by $$e^0=\phi^{-1/2} \eta\quad\text{and}\quad e^i=\phi^{1/2} d\mu_i$$  for $i=1,2,3$.  Using Latin characters $i,j,k\in\{1,2,3\}$, the permutation symbol $\epsilon_{ijk}$ and the summation convention for repeated indices, we have that the covariant derivatives $\nabla_{e_a}e_b$ satisfy
	\begin{align}
	\nabla_{e_0} e_0 & = \frac{1}{2\phi^{3/2}}  \frac{\partial \phi}{\partial \mu_i}  e_i,\label{eq:nabla00} \\
	\nabla_{e_i} e_0 & = \frac{1}{2\phi^{3/2}} \epsilon_{ijk} \frac{\partial \phi}{\partial \mu_j}  e_k , \label{eq:nablai0}\\
	\nabla_{e_0} e_i & = \frac{1}{2\phi^{3/2}} \left( - \frac{\partial \phi}{\partial \mu_i}   e_0 +   \epsilon_{ijk} \frac{\partial \phi}{\partial \mu_j}  e_k   \right),  \label{eq:nable0i}\\
	\nabla_{e_i} e_j & = \frac{1}{2\phi^{3/2}} \left( \epsilon_{ijk} \frac{\partial \phi}{\partial \mu_k}  e_0 + \frac{\partial \phi}{\partial \mu_j} e_i - \delta_{ij} \frac{\partial \phi}{\partial \mu_k} e_k \right).\label{eq:nablaij}
	\end{align}
\end{lemma}

\begin{proof}
The covariant derivatives $\nabla_{e_a} e_b$ for the Levi-Civita connection of $g$ may be computed via
\begin{equation}\label{eq:connection}
\nabla_{e_a} e_b = 
\gamma^c_b (e_a) e_c ,
\end{equation}
where $\gamma^c_b$ denotes the connection 1-forms. 
The $\gamma^a_b$ satisfy the Cartan structure equations:
\begin{equation}\label{eq:Cartan.structure}
d e^a = - \gamma^a_b \wedge e^b , \ \ \  \gamma^a_b + \gamma^b_a =0.
\end{equation}
 Thus, using the monopole equation \eqref{eq:Bogomolnyi}, which we may write 
as
$$d\eta=-\ast_E d\phi = -\frac{1}{2}\epsilon_{ijk} \frac{\partial\phi}{\partial \mu_i} d\mu_j\wedge d\mu_k,$$ and the Cartan structure equations \eqref{eq:Cartan.structure}, we find
\begin{align}\label{eq:connection_form_1}
\gamma^k_j & = \frac{1}{2\phi^{3/2}} \left( \frac{\partial \phi}{\partial \mu_j} e^k - \frac{\partial \phi}{\partial \mu_k} e^j - \epsilon_{ijk} \frac{\partial \phi}{\partial \mu_i} e^0  \right), \\ \label{eq:connection_form_2}
\gamma^0_i & =  \frac{1}{2\phi^{3/2}} \left( - \frac{\partial \phi}{\partial \mu_i} e^0 + \epsilon_{ijk} \frac{\partial \phi}{\partial \mu^j} e^k  \right).
\end{align}
The formulae \eqref{eq:nabla00}--\eqref{eq:nablaij} quickly follow from substituting \eqref{eq:connection_form_1}--\eqref{eq:connection_form_2} into \eqref{eq:connection}.
\end{proof}

\section{Geodesic orbits}\label{sec:Geodesic_Orbits}

\subsection{Existence and location of the geodesic orbits}\label{ss:Existence_GO}

The length of an orbit is determined by the $\U(1)$-invariant function $l: X \rightarrow \mathbb{R}_{\geq 0}$ given by
\begin{equation}\label{eq:length}
l(x)= \frac{1}{\sqrt{\phi(x)}}.
\end{equation}
This function obviously vanishes at any points where the circle action collapses and descends to the quotient space $Y$.
Moreover, if we are in the ALE setting of Examples \ref{ex:flat}--\ref{ex:MultiEH} then $l\to\infty$ at infinity, whereas in the ALF settings of Examples \ref{ex:TN}--\ref{ex:MultiTN} then $l\to m^{-\frac{1}{2}}$.

Equation \ref{eq:length} provides the following simple observation.

\begin{lemma}\label{lem:orbits}
Geodesics orbits on $(X,g)$ are in one-to-one correspondence with critical points of $\phi$.
\end{lemma}

\begin{proof}
We first observe that the length functional restricted to the orbits is critical at an orbit if and only if the length functional is critical at the orbit as a functional on all curves.  This follows since the curvature of any $\U(1)$-invariant curve $\gamma$ will be $\U(1)$-invariant, and so, by the first variation formula for length, we need only consider the projection of variation vector fields onto their $\U(1)$-invariant component, which will then define a $\U(1)$-invariant variation of $\gamma$.

We see that the length functional restricted to critical orbits satisfies
$$\nabla l=-\frac{\nabla\phi}{2\phi^\frac{3}{2}}$$
and thus vanishes if and only if $\nabla\phi$ vanishes or $\phi$ has a singularity.  Since singularities of $\phi$ do not define orbits, we know that only the first possibility defines geodesics.
\end{proof}

\begin{example}[Flat and Taub--NUT metrics]
We see that for any $m\geq 0$ the function $\phi$ in \eqref{eq:phi.TN}
has no critical points and so there are no geodesic orbits in the flat metric in Example \ref{ex:flat} or the Taub--NUT metric in Example \ref{ex:TN}.
\end{example}

\begin{example}[Eguchi--Hanson]
\label{ex:EH.geodesics}
If we take $\phi$ as in \eqref{eq:phi.EH} in Example \ref{ex:Eguchi-Hanson}, we see that its (unique) critical point is given by
$$\frac{x-p}{|x-p|^3}+\frac{x+p}{|x+p|^3}=0\quad\Leftrightarrow\quad x=0.$$
Hence, there is a unique geodesic orbit in the Eguchi--Hanson space, which lies over $0\in\mathbb{R}^3$: it is an equator in the 2-sphere in $T^*S^2$.  Moreover, since the 2-sphere is totally geodesic, any equator will define a closed geodesic, and so there are infinitely many closed geodesics in the Eguchi--Hanson space, only one of which is a geodesic orbit.

Similarly, consider $\phi$ as in \eqref{eq:phi.MultiTN} with $p_1=p\neq 0$ and $p_2=-p$.  By the same calculation, there is a unique geodesic orbit in this multi-Taub--NUT space, which lies over $0\in\mathbb{R}^3$ and is again  an equator in a 2-sphere.
\end{example}

\begin{remark}\label{rmk:EH.geodesics}
For the Eguchi--Hanson space $X$ in Example \ref{ex:Eguchi-Hanson}, the minimal 2-sphere $\Sigma$ given by the straight line between the singularities of $\phi$ in \eqref{eq:phi.EH} is totally geodesic, and the squared distance function to $\Sigma$ is convex everywhere outside of $\Sigma$ (c.f.~\cite{TsaiWangFlows}).  This means that all closed geodesics lie in $\Sigma$ and  $\Sigma$ is the unique compact minimal submanifold in $X$ of dimension at least 2.
\end{remark}

\begin{proposition}[Existence of closed geodesics]\label{prop:existence.geodesics}
	Let $X$ be a connected hyperk\"ahler $4$-manifold obtained from the Gibbons--Hawking ansatz. Suppose that $\phi$ has a finite number $k\geq 2$ of isolated singularities. Then, there are closed geodesics in $X$. 
\end{proposition}

\begin{proof}
We can write $\phi$ as
	\begin{equation}\label{eq:phi.MultiTN_2}
	\phi= m + \sum_{i=1}^k\frac{1}{2 \vert x - p_i \vert_{g_E}} ,
	\end{equation}
	for $m \geq 0$, i.e.~it is given by Example \ref{ex:MultiEH} or \ref{ex:MultiTN} depending on whether $m=0$ or $m>0$.  We see that $\tilde{l}=(\phi-m)^{-1}$ has the same critical points as $\phi$ and, by \eqref{eq:phi.MultiTN_2}, satisfies
	$$\tilde{l}(x)\sim \frac{2}{k}|x|+o(|x|)$$
for $|x|$ sufficiently large, and so $\tilde{l}:\mathbb{R}^3\to\mathbb{R}_{\geq 0}$ is coercive.  Hence, every sequence of points $x_i\in\mathbb{R}^3$ for which $\tilde{l}(x_i)$ is bounded must lie in a bounded domain, and therefore has a convergent subsequence.  In particular, $\tilde{l}$ satisfies the Palais--Smale condition and so we may use the min-max principle to find critical points of $\tilde{l}$ (and hence $\phi$) as follows.

Suppose that $p$ and $q$ are distinct isolated singularities of $\phi$. Then $\tilde{l}$ vanishes at $p$ and $q$ and as these are isolated singularities of $\phi$, we can choose some arbitrarily small $\epsilon>0$ such that in small spheres around $p$ and $q$ we have that $\tilde{l}(x)>\epsilon$.  We may therefore apply the mountain pass lemma to deduce that there is a critical point of $\tilde{l}$  between $p$ and $q$.  Lemma \ref{lem:orbits} then yields the result.
\end{proof}

\begin{remark}
As noted in the proof, Proposition \ref{prop:existence.geodesics} applies to the multi-Eguchi--Hanson and multi-Taub--NUT spaces in Examples  \ref{ex:MultiEH} and \ref{ex:MultiTN}. It is well-known to be false (i.e.~there are no closed geodesics) in the flat case, since geodesics are straight lines, and we shall see in Example \ref{ex:no.min.TN}, it is false for Taub--NUT given in Example \ref{ex:TN}.
Other cases of interest to which one would like to be able to test a version of Proposition \ref{prop:existence.geodesics} are when the singularities of $\phi$ are not isolated or in infinite number.  We shall explicitly see the latter possibility in the case of the Ooguri--Vafa metric in Corollary \ref{cor:OV.geodesics}, and in the case of the Anderson--Kronheimer--LeBrun metrics in Example \ref{ex:AKlB}.
\end{remark}

Suppose that we are in the setting of Example \ref{ex:MultiEH} or \ref{ex:MultiTN}. Then for all the points $p_i$, $p_j$ with $i\neq j$, Proposition \ref{prop:existence.geodesics} gives a critical point $q_{ij}$ of $l$ with 
$$l(q_{ij})= \inf_{\gamma_{ij}} \sup_{x \in \gamma_{ij}} l(x),$$
where $\gamma_{ij}$ are smooth paths connecting $p_i$ to $p_j$.
However, the points $q_{ij}$ are not necessarily distinct. In fact, as we shall see in examples later, there are cases where some points coincide. From the examples it will also be clear that the number of geodesic orbits can change even for a fixed number of singularities. This suggests that is hard to obtain a general statement regarding the exact number and location of these geodesic orbits. Nevertheless, one can prove the following result. 

\begin{proposition}[Location of the geodesic orbits]\label{prop:Conex_Hull}
	Let $X$ be an ALE or ALF hyperk\"ahler 4-manifold 
	as in Example \ref{ex:MultiEH} or \ref{ex:MultiTN}   and let $\gamma$ be a geodesic orbit of the $\U(1)$-action. Then $\mu(\gamma)$ lies in the convex hull of the points $\lbrace p_i \rbrace_{i=1}^k$. 
\end{proposition}
\begin{proof}
By Lemma \ref{lem:orbits}, geodesic orbits correspond to critical points of $\phi$ in \eqref{eq:phi.MultiTN_2}. 
We may compute
	\begin{align}
\nabla\phi&=-\frac{1}{2}\sum_{i=1}^k\sum_{j=1}^3\frac{x^j-p_i^j}{|x-p_i|^3}\frac{\partial}{\partial x^j}=-\frac{1}{2}\sum_{j=1}^3\left(\sum_{i=1}^k\frac{x^j-p_i^j}{|x-p_i|^3}\right)\frac{\partial}{\partial x^j}.\label{eq:nabla.phi}
	\end{align}
	So any critical point $x$ of $\phi$ satisfies 
	$$x= \sum_{i=1}^k \beta_i(x) p_i\quad\text{
	where}\quad 
	\beta_i(x) = \left(\sum_{j=1}^k \frac{1}{|x-p_j|^3} \right)^{-1}\frac{1 }{|x-p_i|^3}\in (0,1].$$
	The result follows.
\end{proof}

We see that Proposition \ref{prop:Conex_Hull} agrees with the conclusion in Example \ref{ex:EH.geodesics}.

\subsection{Lower bounds on the number of geodesic orbits}\label{ss:Lower_Bound}

Now that we have existence of geodesic orbits, we study the question of 
lower bounds on the number of these orbits. Before we turn to the generic case, we shall start by analyzing the special configuration in which the points are collinear.

\begin{proposition}[Collinear points]\label{prop.collinear}
	Let $X$ be an ALE or ALF hyperk\"ahler 4-manifold 
	as in Example \ref{ex:MultiEH} or \ref{ex:MultiTN} such that the $k\geq 2$ singularities of $\phi$ are collinear.  There are  precisely $k-1$ geodesic orbits, whose projections lie on a line between any pair of adjacent singular points of $\phi$. 
\end{proposition}

\begin{proof}
	We may suppose, by changing coordinates, that the collinear singular points lie on the $x^1$-axis, and are ordered so that $p_1>\ldots>p_k$.  By Proposition \ref{prop:Conex_Hull} the critical points of $\phi$ lie on the line between $p_1$ and $p_k$.  Since $l$ in \eqref{eq:length} vanishes at each $p_i$ and is non-constant along the line between $p_i$ and $p_{i+1}$ for all $i$, $l$ must have a strict local maximum (and thus $\phi$ has a strict local minimum) at some point between $p_i$ and $p_{i+1}$.  Restricting $\phi$ to the $x^1$-axis and letting $x=x^1$ we see that
	$$\phi''(x)=\sum_{i=1}^k\frac{1}{|x-p_i|^3}>0,$$
	so the only critical points of $\phi$ are local minima.  If $\phi$ had two local minima between $p_i$ and $p_{i+1}$, it must have a local maximum between these local minima, which is a contradiction.  
\end{proof}

We now observe that the arguments for Proposition \ref{prop:Conex_Hull} and \ref{prop.collinear} clearly extend to the Ooguri--Vafa metric described in Example \ref{ex:OV}.  In fact, one can see explicitly for $\phi$ in \eqref{eq:phi.OV} that its only critical points are at $x=(\pi+2k\pi,0,0)$ for $k\in\mathbb{Z}$.  We deduce the following.

\begin{corollary}\label{cor:OV.geodesics}
	There is a unique geodesic orbit in the Ooguri--Vafa metric defined in Example \ref{ex:OV}, which is $\mu^{-1}(-1,0,0)$. 
\end{corollary}

Proposition \ref{prop.collinear} shows that when $\phi$ has $k$ singularities arranged in a line there are exactly $k-1$ geodesic orbits, and one may ask if this is a general phenomena independently of the location of the singularities. As we shall now show, by exploring the correspondence between geodesic orbits and critical points of $\phi$, Morse theory yields such a lower bound for the number of geodesic orbits.

\begin{theorem}\label{thm:Geodesic_Orbits_Lower_Bound}
	Let $X$ be an ALE or ALF hyperk\"ahler 4-manifold 
	as in Example \ref{ex:MultiEH} or \ref{ex:MultiTN} with $\phi$ having $k\geq 2$ singularities. For the generic arrangement of these singularities, there are at least $k-1$ geodesic orbits, each of index at least $1$.
\end{theorem}
\begin{proof}
	Let $p_1, \ldots , p_k$ be the singularities of $\phi$ and $B_r(p_i)$ the Euclidean ball in $\mathbb{R}^3$ of radius $r$ centred around $p_i$. As $\phi \geq m$ is harmonic, all its critical points are of index either $1$ or $2$, and we shall denote the number of such points by $m_1$ and $m_2$ respectively. The length of $\mu^{-1}(x)$ is given by $l(x)=\phi(x)^{-1/2}$ whose index at a critical point is $3$ minus the index of $\phi$, and as these variations give rise to variations of the corresponding geodesic we find that
	$$\ind ({\mu^{-1}(x)}) \geq 3 - \ind_x \phi .$$
	(Here, we used $\ind ({\mu^{-1}(x)})$ for the index of $\mu^{-1}(x)$ as a minimal submanifold.) 	Thus, 
	\begin{equation}\label{eq:Index_Bounds}
	\begin{aligned}
	\# \lbrace x \ | \ind(\mu^{-1}(x)) \geq 2 \rbrace & \geq \# \lbrace x \ | \ind_x \phi \leq 1 \rbrace = m_1, \\
	\# \lbrace x \ | \ind(\mu^{-1}(x)) \geq 1 \rbrace & \geq \# \lbrace x \ | \ind_x \phi \leq 2 \rbrace = m_1 + m_2.
	\end{aligned}
	\end{equation}
	To apply Morse theory and obtain bounds on $m_1$ and $m_2$ we must know that $\phi$ is Morse: this is true for the generic arrangement of singularities of $\phi$ \cite[Theorem 6.2]{Morse2014}.\footnote{In fact, \cite[Theorem 6.2]{Morse2014} proves that $\phi$ can be made Morse by generically perturbing only one  singularity.}
	
	Recall that $\phi \geq m$ and converges to $m$ at infinity so we may regard it as a (singular) function on the compactification $S^3$ of $\mathbb{R}^3$ with a unique local minimum at $\infty \in S^3$. Furthermore, as $\phi$ converges to $+\infty$ at the singularities we may modify $\phi$ in a neighbourhood $U_i$ of each singularity so that, in each such $U_i$, it has a unique critical point at $p_i$,  which is a maximum. Thus, we obtain a smooth Morse function $\tilde{\phi}:S^3 \to \mathbb{R}$ which agrees with $\phi$ away from  each $U_i$, has a unique global minimum at $\infty$, local maxima at each $p_i$ and the remaining critical points are those of $\phi$. The Euler characteristic of $S^3$ vanishes and so we find that $-k+m_2-m_1+1=0$, so that
	$$m_2-m_1=k-1.$$
The result follows from \eqref{eq:Index_Bounds}.
\end{proof}

 Proposition \ref{prop.collinear} shows that collinear singularities for $\phi$  lead to  the minimum number of geodesic orbits from Theorem \ref{thm:Geodesic_Orbits_Lower_Bound}, so one may ask if there are cases for which the number of geodesic orbits is greater than this minimum. The answer to this question is positive as we  now show.

\begin{proposition}\label{prop:triangle}
	Suppose that $X$ is an ALE or ALF hyperk\"ahler 4-manifold as in Example \ref{ex:MultiEH} or \ref{ex:MultiTN} with $k=3$, and suppose that the points $\lbrace p_1 , p_2 , p_3 \rbrace $ lie in an equilateral triangle. Then, $X$ admits $4$ geodesic orbits, one of which has index at least $2$, and the others have index at least $1$.
\end{proposition}
\begin{proof}
	With no loss of generality we may suppose  $p_1=(-\sqrt{3}a, 0, 0)$, $p_2=(\sqrt{3}a, 0, 0)$ and $p_3=(0,3a,0)$ for some $a>0$. We write $\phi=m + \sum_{i=1}^3 \frac{1}{r_i}$ with $r_i=|x-p_i|$ for $i=1,2,3$. Then, 
	\begin{align}
	-\frac{\del \phi}{\del \mu_1} & = \frac{\mu_1+\sqrt{3}a}{r_1^3} + \frac{\mu_1-\sqrt{3}a}{r_2^3} + \frac{\mu_1}{r_3^3}, \label{eq:triangle.1}\\
	-\frac{\del \phi}{\del \mu_2} & 
	= \frac{\mu_2}{r_1^3} + \frac{\mu_2}{r_2^3} + \frac{\mu_2-3a}{r_3^3},  \label{eq:triangle.2}\\
	-\frac{\del \phi}{\del \mu_3} & 
	= \frac{\mu_3}{r_1^3} + \frac{\mu_3}{r_2^3} + \frac{\mu_3}{r_3^3}.\label{eq:triangle.3}
	\end{align}
Equation \eqref{eq:triangle.3} shows that any critical point lies in the $\mu_3=0$ plane. Furthermore, when $\mu_1=0$ we have $\del_{\mu_1}\phi=0$ by \eqref{eq:triangle.1}, so we look for critical points with $\mu_1=0=\mu_3$, i.e.~on the $\mu_2$-axis. Substituting $\mu_1=\mu_3=0$ in \eqref{eq:triangle.2} shows that $(0,\mu_2,0)$ is a critical point if and only if
	$$f(\mu_2):=(\mu_2^2+3a^2)^{3/2}-2\mu_2 (\mu_2-3a)^2=0.$$
	Clearly $f(a)=0$, and given that $f'(a)>0$ while $f(0)=3^{3/2}a^3>0$ the intermediate value theorem yields that there is another zero of $f$ at some point $b \in (0,a)$. Furthermore, we  compute that
	$$\sign f''(\mu_2)= \sign \left( 2(2a - \mu_2)\sqrt{\mu_2^2+3a^2} + 2\mu_2^2 + 3a^2 \right) >0,$$
	for $\mu_2 \leq a$. Thus, we find that $f'$ is strictly increasing in $(0,a)$ and given that $f(b)=0=f(a)$, 
	we must have $f'(b)<0$. Furthermore, along the $\mu_2$-axis we find that all mixed second partial derivatives of $\phi$ vanish, and at the critical points  $p=(0,a,0)$ and $q=(0,b,0)$  we have
	$$\del^2_{\mu_2} \phi = \frac{f'(\mu_2)}{(\mu_2-3a)^2(\mu_2^2+3a^2)^{3/2}},$$
	where $\mu_2\in\{a,b\}$. Hence, $(\del^2_{\mu_2} \phi)_{q}<0$ while $(\del^2_{\mu_2} \phi)_p>0$. In the same way we compute
	$$\del_{\mu_3}^2\phi = - \frac{(3a^2+\mu_2^2)^{3/2} + 2(3a-\mu_2)^3}{(3a^2+\mu_2^2)^{3/2}(3a-\mu_2)^3}<0,$$
	for $\mu_2 \in (0,a)$. In particular, we deduce that $(\del_{\mu_2}^2\phi)_{q}<0$ and $(\del_{\mu_3}^2\phi)_{q}<0$ and given that $\phi$ is harmonic we have $(\del_{\mu_1}^2\phi)_q>0$. In the same way we have $(\del_{\mu_2}^2\phi)_p>0$, $(\del_{\mu_3}^2\phi)_p<0$ and by direct computation we find $(\del_{\mu_1}^2\phi)_p=3/16a^3>0$. From this we have
	$$\ind_p\phi=1, \ \ \ind_{q}\phi=2.$$
	Now, notice that the centre of the triangle with vertices $p_1,p_2,p_3$ is  $p$. Thus, $\phi$ is invariant under rotation by $\pi/3$ around $p$, and so it   has critical points at the orbits of $p$ and $q$ under this rotation. Given that $p$ is fixed and the orbit of $q$ is $3$ points, the result follows from the inequalities \eqref{eq:Index_Bounds}.
\end{proof}

\noindent In the setting of Proposition \ref{prop:triangle} we plot in Figure \ref{fig2} the level sets of the length function $l$ in \eqref{eq:length}, whose critical points correspond to geodesic orbits, so the position of the 4 orbits becomes clear. 

\begin{figure}[h]\label{fig:2}
	\centering
	\includegraphics[scale=0.5]{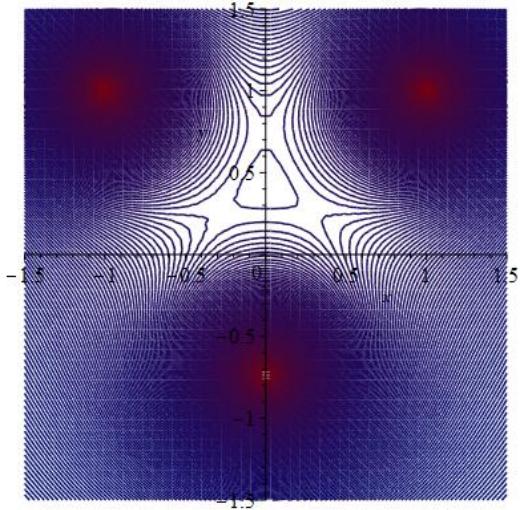}
	\caption{\label{fig2} Level sets of length $l$ for the configuration in Proposition \ref{prop:triangle}.}
\end{figure}

Notice that Proposition \ref{prop:triangle} and its proof yields the following immediate corollary. Recall that any straight line between singular points of $\phi$ defines a 2-sphere in $X$. We shall see below that such a 2-sphere is area-minimizing.

\begin{corollary}\label{cor:triangle}
	Let $X$ be an ALE or ALF hyperk\"ahler 4-manifold 
	as in Example \ref{ex:MultiEH} or \ref{ex:MultiTN} with 3 singularities  $p_1,p_2,p_3$ lying in an equilateral triangle. 
The geodesic orbits in $X$ 
 cannot lie on the area-minimizing 2-spheres whose projections are straight lines joining the singularities. 
\end{corollary}

\begin{remark}[Non-totally geodesic 2-spheres]\label{rmk:non.tot.geod} Suppose we are in the setting of Corollary \ref{cor:triangle}.  Let $\Sigma_1,\Sigma_2,\Sigma_3$ denote the three 2-spheres obtained via the edges of the triangle defined by $p_1,p_2,p_3$.    
	By Lemmas \ref{lem:min.surf} and \ref{lem:min.calibrated} below, each $\Sigma_i$ will be area-minimizing. 
	Since $\Sigma_i$ is a sphere with a $\U(1)$-invariant metric, it will admit a $\U(1)$-invariant closed geodesic.  This cannot be a geodesic in $X$ by Corollary \ref{cor:triangle} and so $\Sigma_i$ is not totally geodesic.
	
	Moreover, the squared distance function to the union $\Sigma=\cup_{i=1}^3\Sigma_i$ cannot be convex on $X\setminus\Sigma$ since otherwise no closed geodesics on $X\setminus\Sigma$ could exist, contradicting Corollary \ref{cor:triangle}.  We can also move the vertices of the triangle in $\mathbb{R}^3$ so that a geodesic orbit in $X$ is arbitrarily close to, say, $\Sigma_i$.  Thus, the neighbourhood of $\Sigma_i$  in which the squared distance function to $\Sigma_i$ could be convex can be made arbitrarily small by varying the hyperk\"ahler metric on the fixed space $X$.
	
	These behaviours are in marked contrast to the Eguchi--Hanson space in Example \ref{ex:Eguchi-Hanson} as discussed in Remark \ref{rmk:EH.geodesics}.
\end{remark}

\begin{example}[Infinitely many geodesic orbits]\label{ex:AKlB}
Proposition \ref{prop.collinear} shows that for any $n\in\mathbb{N}$ there is an irreducible complete hyperk\"ahler $4$-manifold $X_n$ 
which has exactly $n$ geodesic orbits. The same result shows that the lift of the Ooguri--Vafa metric in Example \ref{ex:OV} to its universal cover yields an example of an irreducible but incomplete hyperk\"ahler 4-manifold with an infinite number of geodesic orbits. 
Finally, the Anderson--Kronheimer--LeBrun metrics in Example \ref{ex:AKL} give complete examples admitting infinitely many geodesic orbits by Theorem \ref{thm:Geodesic_Orbits_Lower_Bound}.  
\end{example}

\subsection{Curve shortening flow}\label{ss:Curve_Shortening}

We now turn to the study of the curve shortening flow for $\U(1)$-invariant curves in the Gibbons--Hawking ansatz.  Recall that the curve shortening flow is
\begin{equation}\label{eq:curve.shortening.flow}
\frac{\partial\gamma}{\partial t}=\kappa=\gamma'',
\end{equation}
where $'$ denotes the derivative with respect to arclength along $\gamma$, and $\kappa$ is the curvature of $\gamma$.  We may compute the curvature of a $\U(1)$-orbit using Lemma \ref{lem:Covariant_Derivatives} to be
$$\kappa=\nabla_{e_0}e_0=\frac{1}{2\phi^2}\nabla\phi.$$
Hence, the $\U(1)$-invariant curve shortening flow is
$$\dot{\mu}_j=\frac{1}{2\phi^2}\frac{\partial\phi}{\partial\mu_j}\quad\text{for $j=1,2,3$.}$$

In particular, from \eqref{eq:nabla.phi}, in the setting of Examples \ref{ex:flat}--\ref{ex:MultiTN} we deduce that the $\U(1)$-invariant curve shortening flow is equivalent to
\begin{equation}\label{eq:Curve_Shortening_Flow}
\dot{x} =  
-\frac{1}{4\phi^2}\sum_{i=1}^k \frac{x-p_i }{|x-p_i|^3} .
\end{equation}
By definition, the curve shortening flow decreases length $l$ in \eqref{eq:length} and so increases $\phi$. However, for large enough $|x|$ we have
$$\phi = m + \frac{k}{2|x|} + o(|x|^{-1}),$$
and hence $\phi$ decreases with $|x|$.  Thus,   solutions of the flow \eqref{eq:Curve_Shortening_Flow}   stay within a bounded domain.

The only critical points of \eqref{eq:Curve_Shortening_Flow} are clearly the geodesic orbits (where $\nabla\phi=0$) and the singularities of $\phi$ (i.e.~the points $p_i$).

\begin{example}[Flat case]
In the flat case, where $\phi$ is given in \eqref{eq:phi.flat}, we see that the flow \eqref{eq:Curve_Shortening_Flow} is
$$\dot{x}=-\frac{rx}{r^2}=-\frac{x}{r}.$$
Hence, the flow is purely radial and we see that
$$\dot{r}=-1$$ 
whose solution is then 
$$r(t)=r(0)-t.$$
So all curves shrink to a point (at the origin) in a finite time determined by their initial distance from the origin.
\end{example}

\begin{example}[Taub--NUT]
In the Taub--NUT case, where $\phi$ is as in \eqref{eq:phi.TN}, the flow \eqref{eq:Curve_Shortening_Flow} becomes:
$$\dot{x}=-\frac{x}{(2m+\frac{1}{r})^2r^3}=-\frac{x}{r(2mr+1)^2}$$
Thus, the flow is again radial and is given by
$$\dot{r}=-\frac{1}{(2mr+1)^2}.$$
We therefore see that
$$\frac{d}{dt} \frac{1}{6m} \left( 1+ 2mr \right)^3=-1$$
and thus
$$\left(1 + 2mr(t) \right)^3 =\left(1 +2mr(0) \right)^3-6mt.$$
Therefore, since $m>0$, we have that the flow again exists for finite time,  shrinking to the origin, and the extinction time is determined by the initial distance from the origin.
\end{example}

\begin{example}[Eguchi--Hanson]\label{ex:EH.flow} Given $p\in\mathbb{R}^3\setminus\{0\}$ we can rotate coordinates so that $p=(a,0,0)$ for $a>0$.  We then write $(x^1,x^2,x^3)=(u,v\sin\theta,v\cos\theta)$.  Since $\phi$ in \eqref{eq:phi.EH} is independent of $\theta$, we see that $\theta$ is preserved along the flow \eqref{eq:Curve_Shortening_Flow}, which is then equivalent to:
\begin{align}
\dot{u}&=-\frac{1}{((u-1)^2+v^2)^{-1}+((u+1)^2+v^2)^{-1}}\left(\frac{u-1}{((u-1)^2+v^2)^{\frac{3}{2}}}+\frac{u+1}{((u+1)^2+v^2)^{\frac{3}{2}}}\right),\label{eq:curve.EH.1}\\
\dot{v}&=-\frac{v}{((u-1)^2+v^2)^{-1}+((u+1)^2+v^2)^{-1}}\left(\frac{1}{((u-1)^2+v^2)^{\frac{3}{2}}}+\frac{1}{((u+1)^2+v^2)^{\frac{3}{2}}}\right)\label{eq:curve.EH.2}
\end{align}
The right-hand side of \eqref{eq:curve.EH.2} has the opposite sign of $v$ and vanishes only when $v=0$. Thus the flow will take any point with $v\neq 0$ towards the axis $v=0$, which is preserved by the flow.

In contrast, we see that the right-hand side of \eqref{eq:curve.EH.1} is negative if $u>1$ and positive if $u<-1$, and so will take points with $|u|>1$ towards points with $|u|\leq 1$.

When $v=0$ we see that $\dot{u}>0$ if $u\in(0,1)$ and $\dot{u}<0$ when $u\in(-1,0)$.  We deduce that $(0,0)$ is an unstable critical point, which is verified by Figure \ref{fig3} of the integral curves of the right-hand side of \eqref{eq:curve.EH.1}--\eqref{eq:curve.EH.2}.
\begin{figure}[h]\label{fig:3}
	\centering
	\includegraphics[scale=0.5]{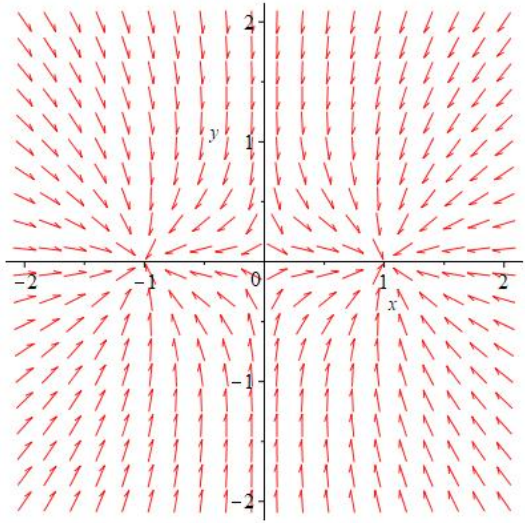}
	\caption{\label{fig3} $\U(1)$-invariant curve-shortening flow for Eguchi--Hanson.}
\end{figure}

Overall, any $\U(1)$-invariant curve with $x^1\neq 0$ will flow to one of the singular points of $\phi$ depending on the sign of $x^1$, whereas all $\U(1)$-invariant curves with $x^1=0$ will flow to the geodesic orbit at 0.
\end{example}

\begin{example}[Collinear points]
In the multi-Eguchi--Hanson and multi-Taub--NUT cases where the points  are collinear, we obtain a similar picture to Example \ref{ex:EH.flow}: the geodesic orbits are unstable critical points, and  generic $\U(1)$-invariant curves flow to the points $p_i$ (e.g.~see Figure \ref{fig4}).
\begin{figure}[h]\label{fig:4}
	\centering
	\includegraphics[scale=0.5]{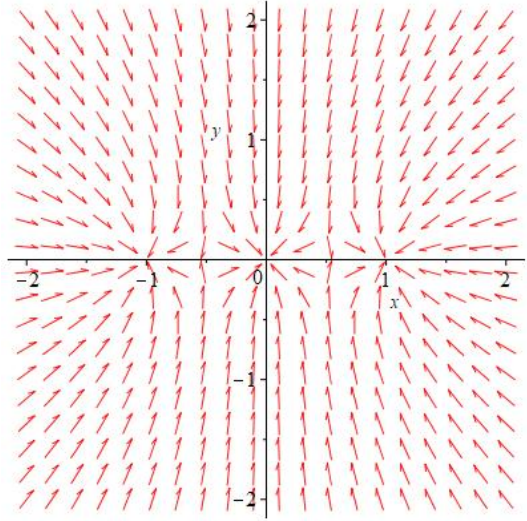}
	\caption{\label{fig4} $\U(1)$-invariant curve-shortening flow for multi-Eguchi--Hanson: three collinear points.}
\end{figure}
\end{example}

\begin{example}[Equilateral triangle]
For the case of $\phi$ with three singular points in an equilateral triangle, all the stationary points for the curve shortening flow are either the singularities of $\phi$ or the geodesic orbits.  
 All these geodesic orbits are unstable points as shown by the computations of Proposition \ref{prop:triangle}. The central point $p$ is a source with heteroclinic orbits connecting it to all the remaining rest points of the flow. The other $3$ geodesic orbits correspond to saddles. Their stable manifold are the three heteroclinic orbits connecting them to $p$, while their unstable manifold is the union of two heteroclinics which connect each of these saddle points to the two nearest singularities of $\phi$, which are attractors for the flow.  This is illustrated in Figure \ref{fig5}.
\begin{figure}[h]\label{fig:5}
	\centering
	\includegraphics[scale=0.5]{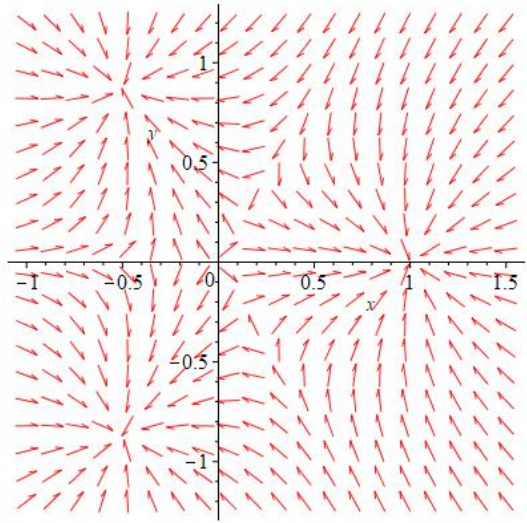}
	\caption{\label{fig5} $\U(1)$-invariant curve-shortening flow for multi-Eguchi--Hanson: three points in equilateral triangle.}
\end{figure}
\end{example}

In fact, the phenomena illustrated by these examples is a general one as we shall now show.

\begin{proposition}[Stability of critical points for the circle-invariant curve shortening flow]
	Let $X$ be an ALE or ALF hyperk\"ahler 4-manifold   
	as in Example \ref{ex:MultiEH} or \ref{ex:MultiTN}. Any critical point of the $\U(1)$-invariant curve shortening flow is one of the following.
	\begin{itemize}
		\item[(a)] A point corresponding to a singularity $p_1, \ldots , p_k$ of $\phi$: these are all stable critical points for the flow.
		\item[(b)] A geodesic circle orbit:   
		these are all unstable critical points for the flow. 
	\end{itemize} 
\end{proposition}
\begin{proof}
	Recall that the curve shortening flow equation is $\dot{x}=\frac{1}{2\phi^2}\nabla \phi$ and so the flow increases $\phi(x(t))$ with $t$. Suppose we are at a critical point $x_0$ of the flow, so either $\phi(x_0)^{-1}=0$ or $\nabla \phi(x_0)=0$. These correspond respectively to the case when $x_0=p_i$ for some $i \in \lbrace 1, \ldots , k \rbrace $ or the case when $\mu^{-1}(x_0)$ is a geodesic orbit. We shall now analyse the stability of these two kinds of critical points.
	
	 In the first case we have that $\phi(x_0)=\infty$ and so such critical points are stable. In the case when $x_0 \notin \lbrace p_1 , \ldots , p_k \rbrace$, then $x_0$ is a critical point of $\phi$, which is a non-constant harmonic function on $\mathbb{R}^3$ and so satisfies the mean value property.  Hence, there are directions at $x_0$ in which $\phi$ grows, and thus it is an unstable critical point of the flow. 
\end{proof}

\subsection{Monopole dynamics}

In great part, the physics interest on the hyperk\"ahler metric in monopole moduli spaces is motivated by using its geodesic flow to approximate the low energy dynamics of monopoles as first suggested in \cite{Manton}.
As a consequence of this, the closed geodesics we found have a physics interpretation in terms of monopole dynamics. 

In \cite{Gibbons} it is shown that the ALF metrics we consider, i.e. those arising from the Gibbons--Hawking ansatz,  appear when one considers the moduli space of $k+1$ monopoles on $\mathbb{R}^3$, $k$ of which are extremely massive. In the limit when the masses of these $k$ monopoles becomes infinitely large, they become static at a specific location and the metric on the moduli space of the remaining monopole is precisely the ALF metric we consider, with $\phi$ having $k$ singularities located at the position of the $k$ infinitely massive monopoles. The projection $\mu: X \to \mathbb{R}^3$ can be though of as giving the location of the monopole and points in the same fibre can be thought of differing from each other by an internal phase of the Higgs field. 

Thus, the closed geodesics in these moduli spaces (and hence in these ALF metrics) represent periodic motions of monopoles. In particular, the geodesic orbits we find are periodic motions on which the monopole stays in the same location but its phase is varying in a periodic manner. Our results show that if $k \geq 2$ these periodic motions always exist (Proposition \ref{prop:existence.geodesics}) and occur in the convex hull of the infinitely massive monopoles (Proposition \ref{prop:Conex_Hull}). Furthermore, generically there are at least $k-1$ such periodic orbits that are dynamically unstable (Theorem \ref{thm:Geodesic_Orbits_Lower_Bound}). We also find, perhaps somewhat surprisingly, that for the same number of very massive monopoles, depending on their location, there can be a different number of geodesic orbits. We saw such an example from putting together Propositions \ref{prop.collinear} and \ref{prop:triangle}. The dynamical stability of these periodic orbits is also analyzed in Proposition \ref{prop:triangle} and more generally in $\S$\ref{ss:Curve_Shortening}, which may be interesting from the point of view of monopole dynamics.

\section{Minimal surfaces}\label{sec:Minimal_Surfaces}

\subsection{Existence and classification}\label{ss:Min_Surf_Classification}

Any circle-invariant surface in $(X,g)$   projects to a curve $\gamma$ in $\mathbb{R}^3$, 
 so can be written as $\mu^{-1}(\gamma)$.  
Abusing notation, we identify $\gamma$ with an arclength parametrization $\gamma:I\subset \mathbb{R} \to \mathbb{R}^3$. The area of a circle-invariant surface in $X$ is then given by
\begin{equation}\label{eq:Area_Length}
\begin{aligned}
\mathrm{Area}(\mu^{-1}(\gamma)) & = \int_{\mu^{-1}(\gamma)} \phi^{-1/2} \eta \wedge \phi^{1/2} \sum_{i=1}^3 \dot{\gamma}_i dx^i  = 2 \pi \sum_{i=1}^3 \int_{\gamma}  \dot{\gamma}_i dx^i  = 2 \pi \ \mathrm{Length}(\gamma).
\end{aligned}
\end{equation} 
Since minimal surfaces are, by definition, critical for area and geodesics are critical for length, we  deduce the following, using a similar argument to Lemma \ref{lem:orbits}.

\begin{lemma}[Correspondence with straight lines]\label{lem:min.surf}
An embedded circle-invariant surface $\mu^{-1}(\gamma)$ in $X$ is minimal if and only if $\gamma$ is a geodesic (i.e.~a straight line segment) 
in $\mathbb{R}^3$.  
\end{lemma}
 
\begin{remark}[Topology of circle-invariant minimal surfaces]
 Topologically, the smooth, embedded, complete minimal surfaces $\mu^{-1}(\gamma)$ in $X$ can be of three types: spheres $S^2$, when $\gamma$ connects two fixed points of the $\U(1)$-action; planes $\mathbb{R}^2$, when $\gamma$ is an infinite ray starting at a fixed point; and cylinders $S^1\times\mathbb{R}$, when $\gamma$ is a straight line not passing through any of the fixed points.
\end{remark}

In fact, we can make a stronger statement concerning the Riemannian geometry of certain circle-invariant minimal surfaces in the case when the singular points of $\phi$ are collinear.

\begin{lemma}\label{lem:tot.geod}
Suppose that the singularities of $\phi$ lie on a line $\Lambda$ and let $\gamma$ be a line segment contained in $\Lambda$.    Then $\mu^{-1}(\gamma)$ is a totally geodesic surface in $X$.  In particular, any compact $\U(1)$-invariant minimal surface in $X$ is totally geodesic.
\end{lemma}

\begin{proof}
Rotations about $\Lambda$ preserve $\phi$ and are isometries on $\mathbb{R}^3$, so they lift to isometries on $X$. Their common fixed point set is  $\mu^{-1}(\Lambda)$.  Therefore, $\mu^{-1}(\gamma)$ is totally geodesic as claimed. The final statement follows from Lemma \ref{lem:min.surf}.
\end{proof}

\begin{example}[Collinear multi-Eguchi--Hanson/multi-Taub--NUT]
Lemma \ref{lem:tot.geod} shows that the 2-sphere in the Eguchi--Hanson metric, which is the zero section in $T^*S^2$, is totally geodesic, as stated earlier. More generally, if $X$ is given by Example \ref{ex:MultiEH} or \ref{ex:MultiTN} with collinear points $p_1,\ldots,p_k$, then any 2-sphere in $X$ defined by a straight line between adjacent singular points is totally geodesic. These observations fit well with Proposition \ref{prop.collinear}, which shows that the geodesic orbits lie in these totally geodesic spheres.
\end{example}

\begin{remark}
	There are interesting examples of Gibbons--Hawking type gravitational instantons, obtained from harmonic functions $\phi$  not having collinear singularities, which nevertheless admit totally geodesic $2$-spheres. Fror example, consider  $\phi$ with $4$ singularities lying in a square in a plane. Then, fix a diagonal $\gamma$ of the square and consider the group   generated by  reflection in the plane containing the singularities, and a reflection within the plane in the diagonal $\gamma$. The fixed point set of this group is $\gamma$, which then lifts to a totally geodesic $2$-sphere.
\end{remark}

\begin{example}[Ooguri--Vafa]
If we take the Ooguri--Vafa metric in Example \ref{ex:OV}, then $\Sigma=\mu^{-1}(\mathbb{S}^1\times\{0\})$ can be identified with the preimage of the $x^1$-axis in $\mathbb{R}^3$.  It is clear that $\Sigma$ is an immersed 2-sphere with a double point or, equivalently, it can be viewed as a pinched (or nodal) 2-torus.  Since the singularities of $\phi$ given in \eqref{eq:phi.OV} lie on a line, Lemma \ref{lem:tot.geod} shows that $\Sigma$ is totally geodesic.
\end{example}

\begin{remark}[Non-totally geodesic minimal 2-spheres]
It follows from Remark \ref{rmk:non.tot.geod} that circle-invariant minimal surfaces in $X$ need not be totally geodesic.  In fact, is clear that the arguments in Remark \ref{rmk:non.tot.geod} show that, for generic choices of $p_1,\ldots,p_k$ for $k\geq 3$ in Examples \ref{ex:MultiEH} and \ref{ex:MultiTN},  the 2-spheres given by $\mu^{-1}(\gamma_{ij})$, where $\gamma_{ij}$ is a straight line joining $p_i$ and $p_j$, are not totally geodesic in the multi-Eguchi--Hanson and multi-Taub--NUT metrics.
\end{remark}

If $\gamma$ is a straight line segment then (recalling that we parametrize $\gamma$ by arclength)   $\dot{\gamma}=(\dot{\gamma}_1 , \dot{\gamma}_2 , \dot{\gamma}_3)$ lies in the unit $2$-sphere $\mathbb{S}^2\subseteq\mathbb{R}^3$. 
Recall that $\mathbb{S}^2 \subset \mathbb{R}^3$ can be identified with the twistor sphere of $X$.  Thus, $v \in \mathbb{S}^2$ defines a complex structure $I_{v}$.  This observation leads us to the following result.

\begin{lemma}[Calibration]\label{lem:min.calibrated}
Any circle-invariant  minimal surface $\mu^{-1}(\gamma)$ is $I_{\dot{\gamma}}$-holomorphic, thus calibrated by 
$$\omega_{\dot{\gamma}}=  \sum_{i=1}^3 \dot{\gamma}_i \left( \eta \wedge  dx^i +  \phi \ dx^j \wedge dx^k \right) ,$$
with $(i,j,k)$ denoting a cyclic permutation of $(1,2,3)$, and hence $\mu^{-1}(\gamma)$ is area-minimizing.
\end{lemma}

\begin{example}[Ooguri--Vafa as an elliptic fibration]\label{ex:OV.elliptic}
We see that for the Ooguri--Vafa metric in Example \ref{ex:OV}, $\Sigma_b=\mu^{-1}(\mathbb{S}^1\times\{b\})$ for any $b\in B$ is an $I_1$-holomorphic curve, which is an elliptic curve that is embedded for $b\neq 0$ and nodal for $b=0$.  Hence, as is well-known, the Ooguri--Vafa metric is an elliptic fibration over $B$ with a single nodal fibre. 
\end{example}

A tantalizing question in Riemannian geometry concerns the representability of (some) homology classes by minimal (or area-minimizing) submanifolds. For the hyperk\"ahler $4$-manifolds obtained from the Gibbons--Hawking ansatz we have described, which only have non-trivial second homology, we have the following result.

\begin{proposition}[Representing homology classes by minimal surfaces]\label{prop:min.spheres}
	Let $X$ be an ALE or ALF hyperk\"ahler $4$-manifold 
	as in Examples \ref{ex:MultiEH} and \ref{ex:MultiTN} with points $p_1, \ldots , p_k$ in $\mathbb{R}^3$. Any cohomology class $\kappa \in H_2(X, \mathbb{Z}) $ can be represented by a circle-invariant minimal surface, which is a union of embedded 2-spheres that pairwise intersect transversely in at most one point. When $\kappa \cdot \kappa =-2$ there is a unique such circle-invariant surface which is area-minimizing (in fact, calibrated).  
\end{proposition}

\begin{proof}
Given any two points $p_i$ and $p_j$ there is a unique straight line $\gamma_{ij}$ between them. Then $\mu^{-1}(\gamma_{ij})$ is 
  area-minimizing in $[\mu^{-1}(\gamma_{ij})]$ by Lemmas \ref{lem:min.surf} and \ref{lem:min.calibrated}, and \eqref{eq:Area_Length} shows that it is the unique circle-invariant area-minimizer in the class. 
If $\gamma_{ij}$ does not pass through any of the other points $p_l$ then $\mu^{-1}(\gamma_{ij})$ is an embedded 2-sphere, and otherwise is as described in the statement.  

The second homology group $H_2(X, \mathbb{Z}) \cong A_{k-1}$ is generated by the classes $[\mu^{-1}(\gamma_{ij})]$ and so the first claim about the existence of minimal representatives follows. Finally, the classes $\kappa$ so that $\kappa \cdot \kappa =-2$ are precisely those of the form $[\mu^{-1}(\gamma_{ij})]$.
\end{proof}

\begin{remark}[Non-uniqueness of minimal representatives]
Suppose that we are in the setting of Proposition \ref{prop:min.spheres} and  $p_1, p_2,p_3$  are not collinear. Let $\kappa_i = [\mu^{-1}(\gamma_{i,i+1})]$. Then, $\kappa_1+\kappa_{2}$ has two minimal representatives.  The first is the union of the area-minimizing representatives of $\kappa_1$ and $\kappa_2$, which is a union of two minimal $2$-spheres which transversely intersect at a point.  The second is the minimal $2$-sphere which is given by the straight line from $p_1$ to $p_{3}$, i.e.~$\mu^{-1}(\gamma_{13})$. It follows immediately from \eqref{eq:Area_Length}
	and the proof of Proposition \ref{prop:min.spheres} that this second one is the minimizing one.
\end{remark}

\subsection{Mean curvature flow}\label{ss:Mean_Curvature_Flow}

The computation \eqref{eq:Area_Length} suggests a relation between mean curvature flow (the gradient flow for area) of $\mu^{-1}(\gamma)$ and curve shortening flow (the gradient flow for length) of $\gamma$. The precise relation is as follows. 

\begin{proposition}[Flow of curves]\label{prop:curve.flow}
	Parametrize curves $\gamma$ in $\mathbb{R}^3$ with respect to Euclidean arc-length and let $'$ denote the derivative with respect to this parameter. 	Then, the mean curvature flow of $\mu^{-1}(\gamma)$   in $X$ coincides with
	\begin{equation}\label{eq:Mean_Curvature_Flow_E}
	\frac{\partial}{\partial t}\gamma = \phi^{-1} \gamma''.
	\end{equation}
\end{proposition}
\noindent  Comparing with \eqref{eq:curve.shortening.flow}, we see that, up to the factor of $\phi^{-1}$, \eqref{eq:Mean_Curvature_Flow_E} coincides with the Euclidean curve shortening flow for $\gamma$.  However, this factor $\phi^{-1}$ is important since it has zeros, and therefore the flow fails to be parabolic (in fact, it is stationary) at these zeros of $\phi^{-1}$.

\begin{proof} We first assume that $\gamma=(\gamma_1,\gamma_2,\gamma_3)$ is parametrized by arc-length with respect to the induced metric on $\mathbb{R}^3$ and let $\dot{}$ denote differentiation with respect to this parameter.
The tangent space to $\mu^{-1}(\gamma)$   is spanned by $e_0$ and $\dot{\gamma}$.  We find that (using summation convention)
$$\dot{\gamma}= \phi^{1/2}\dot{\gamma}_i e_i.$$
Lemma \ref{lem:Covariant_Derivatives} then yields:
\begin{align}
\nabla_{\dot{\gamma}} \dot{\gamma} & =  
\nabla_{\dot{\gamma}} (\phi^{1/2}\dot{\gamma}_j e_j )\nonumber\\
& = \ddot{\gamma} +   
\dot{\gamma}_j\nabla_{ \dot{\gamma}}  (\phi^{1/2}  e_j) \nonumber\\
& = \ddot{\gamma} +   \phi^{-1/2} \nabla_{\dot{\gamma}}  (\phi^{1/2}) 
\phi^{1/2}\dot{\gamma}_j e_j + \phi^{1/2}
\dot{\gamma}_j\nabla_{ \dot{\gamma}} e_j \nonumber\\
& = \ddot{\gamma} +   \phi^{-1/2} \nabla_{\dot{\gamma}}  (\phi^{1/2}) \dot{\gamma} + \phi 
\dot{\gamma}_i \dot{\gamma}_j \nabla_{e_i} e_j \nonumber\\
& = \ddot{\gamma} +  \frac{1}{2\phi} (\nabla_{\dot{\gamma}} \phi) \dot{\gamma} +  
\dot{\gamma}_i \dot{\gamma}_j \frac{1}{2\phi^{1/2}} \left( \frac{\partial \phi}{\partial \mu_j} e_i - \delta_{ij} \frac{\partial \phi}{\partial \mu_k} e_k \right), \label{eq:ddot.gamma.1}
\end{align}
noting that the component of $\nabla_{e_i}e_j$ in the $e_0$ direction vanishes since the sum over $i,j$ is symmetric, whilst $\epsilon_{ijk}$ is skew-symmetric in $i,j$.  
We first observe that 
\begin{align}\label{eq:ddot.gamma.2}
\frac{1}{2\phi^{1/2}}\dot{\gamma}_i\dot{\gamma}_j\frac{\partial \phi}{\partial \mu_j} e_i&=\frac{1}{2\phi^{1/2}}\dot{\gamma}_i\nabla_{\dot{\gamma}}\phi e_i =\frac{1}{2\phi}(\nabla_{\dot{\gamma}}\phi)\dot{\gamma}.
\end{align}
Now, as our curve is parametrized by arc-length, meaning that 
$\sum_{i=1}^3 \dot{\gamma}_i^2 = \phi^{-1}$, we then see that
\begin{align}\label{eq:ddot.gamma.3}
\frac{1}{2\phi^{1/2}}\dot{\gamma}_i\dot{\gamma}_j\delta_{ij}\frac{\partial \phi}{\partial \mu_k} e_k&=\frac{1}{2\phi^{3/2}}\frac{\partial \phi}{\partial \mu_k} e_k=\frac{1}{2\phi}\nabla\phi.
\end{align}
Inserting \eqref{eq:ddot.gamma.2} and \eqref{eq:ddot.gamma.3} in \eqref{eq:ddot.gamma.1} 
we find that
\begin{equation}\label{eq:ddot.gamma.4}
\nabla_{ \dot{\gamma}} \dot{\gamma} = \ddot{\gamma} + \frac{1}{\phi}  ( \nabla_{\dot{\gamma}} \phi )\dot{\gamma}  -\frac{1}{2\phi}\nabla\phi
 .
\end{equation}
Observe from Lemma \ref{lem:Covariant_Derivatives} that
\begin{equation}\label{eq:nabla00.2}
\nabla_{e_0}e_0=\frac{1}{2\phi^{3/2}}\frac{\partial\phi}{\partial\mu_i}e_i=\frac{1}{2\phi}\nabla\phi.
\end{equation}
Putting \eqref{eq:ddot.gamma.4} together with \eqref{eq:nabla00}, we see that
\begin{equation}\label{eq:ddot.gamma.5}
I:=\nabla_{e_0} e_0 + \nabla_{\dot{\gamma}} \dot{\gamma}  = \ddot{\gamma} + \frac{1}{\phi}  ( \nabla_{\dot{\gamma}} \phi )\dot{\gamma}.
\end{equation}
Then, the mean curvature $H$ of $\mu^{-1}(\gamma)$, which is the normal projection of the quantity in \eqref{eq:ddot.gamma.5}, is 
\begin{align} 
H & = ( \nabla_{e_0} e_0 + \nabla_{\dot{\gamma}} \dot{\gamma})^{\perp} 
= I - \langle I, e_0 \rangle e_0 - \langle I, \dot{\gamma} \rangle \dot{\gamma}  
\label{eq:mean.curvature}
= \ddot{\gamma} + \frac{1}{2\phi}  ( \nabla_{\dot{\gamma}} \phi )\dot{\gamma},
\end{align}
where we have used the fact that $\sum_{i=1}^3 \dot{\gamma}_i^2 = \phi^{-1}$ implies
\begin{align*}
0 & = 2\dot{\gamma}_i\ddot{\gamma}_i+\phi^{-2}\frac{\partial \phi}{\partial \mu_i}\dot{\gamma}_i  = 2\phi^{-1}\langle \ddot{\gamma},\dot{\gamma}\rangle+\phi^{-2}\nabla_{\dot{\gamma}}\phi
\end{align*} 
and so 
$$2 \langle I, \dot{\gamma} \rangle =-\phi^{-1}\nabla_{\dot{\gamma}}\phi+2\phi^{-1}\nabla_{\dot{\gamma}}\phi= \phi^{-1}\nabla_{ \dot{\gamma}}\phi.$$

We seek solutions of mean curvature flow of the form $\mu^{-1}(\gamma(t))$, where $\gamma(t)$ denotes a $1$-parameter family of curves as above.  Note that mean curvature flow preserves circle-invariance, i.e.~mean curvature flow starting at $\mu^{-1}(\gamma(0))$ must be of the form $\mu^{-1}(\gamma(t))$. Now, $\mu^{-1}(\gamma(t))$ is a solution of the mean curvature flow if and only if $$\frac{\partial}{\partial t} \mu^{-1}(\gamma(t)) =H=  H\big(\mu^{-1}(\gamma(t))\big),$$ with $H$ as in \eqref{eq:mean.curvature}.  The computation of $H$ shows that it is orthogonal to the kernel of $\mu_\ast$ (where $\mu$ is the projection of $X$ to $\mathbb{R}^3$) so we may work on $\mathbb{R}^3$ and write $\mu_* (\partial_t \mu^{-1}(\gamma(t)))=  \mu_* H$. By \eqref{eq:mean.curvature}, this may be written as the following equation for $\gamma=\gamma(t)$:
\begin{equation}\label{eq:Mean_Curvature_Flow}
\frac{\partial}{\partial t}\gamma =  \ddot{\gamma} + \frac{1}{2\phi}  (\nabla_{\dot{\gamma}} \phi ) \dot{\gamma}.
\end{equation}
Now, recall we use $'$ to denote differentiation with respect to the Euclidean arc-length parameter.  Then,  $\sum_{i=1}^3 (\gamma_i')^2 =1$ while $\sum_{i=1}^3 \dot{\gamma_i}^2 =\phi^{-1}$ so
$$\dot{\gamma}=\phi^{-\frac{1}{2}}\gamma'.$$
We then compute
\begin{align}
\label{eq:ddot.gamma}\ddot{\gamma}
& = -\frac{1}{2 \phi} (\nabla_{ \dot{\gamma}} \phi ) \dot{\gamma} + \phi^{-1} \gamma'' .
\end{align}
Thus, inserting \eqref{eq:ddot.gamma} in \eqref{eq:Mean_Curvature_Flow} we find the claimed flow equation \eqref{eq:Mean_Curvature_Flow_E}.
\end{proof}

It is well-known that the mean curvature flow exists as long as the square norm of the second fundamental form of the flowing submanifold remains bounded.  From this observation we can deduce a criterion for blow-up related to the flow of curves \eqref{eq:Mean_Curvature_Flow_E}.

\begin{proposition}\label{prop:curve.flow.blow.up}
The mean curvature flow of $\mu^{-1}(\gamma)$ for curves $\gamma\subseteq\mathbb{R}^3$ exists as long as $\phi^{-1/2}\gamma''$ and $\nabla^{\perp}_{\mathbb{R}^3}\log\phi$ are bounded on $\gamma$, where $\nabla^{\perp}_{\mathbb{R}^3}$ is taken with respect to the flat metric on $\mathbb{R}^3$.
\end{proposition}

\begin{proof}  Our aim is to compute the square norm of the second fundamental form of $\mu^{-1}(\gamma)$.  

Recall the notation from the proof of Proposition \ref{prop:curve.flow}.  In particular, recall that $e_0$ and $\dot{\gamma}$ are orthogonal unit tangent vectors on $\mu^{-1}\gamma$.  We see from \eqref{eq:nabla00.2} that
\begin{equation*}
\nabla^{\perp}_{e_0}e_0=\frac{1}{2\phi}\nabla\phi-\frac{1}{2\phi}(\nabla_{\dot{\gamma}}\phi)\dot{\gamma}=\frac{1}{2\phi}\nabla^{\perp}\phi.
\end{equation*}
From this, \eqref{eq:mean.curvature} and \eqref{eq:ddot.gamma} we deduce that
\begin{equation*}
\nabla_{\dot{\gamma}}^{\perp}\dot{\gamma}=\ddot{\gamma}+\frac{1}{2\phi}(\nabla_{\dot{\gamma}}\phi)\dot{\gamma}-\frac{1}{2\phi}\nabla^{\perp}\phi
=\frac{1}{\phi}\gamma''-\frac{1}{2\phi}\nabla^{\perp}\phi.
\end{equation*}
We are then left with computing $\nabla_{\dot{\gamma}}^{\perp}e_0$.  Using Lemma \ref{lem:Covariant_Derivatives} we find that
\begin{align*}
\nabla_{\dot{\gamma}}e_0&=\phi^{\frac{1}{2}}\dot{\gamma}_i\nabla_{e_i}e_0 =\frac{1}{\phi}\epsilon_{ijk}\dot{\gamma}_i\frac{\partial\phi}{\partial\mu_j}e_k =\frac{1}{\phi^{\frac{3}{2}}}\epsilon_{ijk}\dot{\gamma}_i\frac{\partial\phi}{\partial\mu_j}\frac{\partial}{\partial\mu_k} =\frac{1}{\phi}\dot{\gamma}\times \nabla\phi,
\end{align*}
where $\times$ is the cross product on $\mathbb{R}^3$.  Since $u\times v$ is orthogonal to $u$ and $v$, and $u\times u=0$, we find
\begin{equation*}
\nabla_{\dot{\gamma}}^{\perp}e_0=\frac{1}{\phi}\dot{\gamma}\times\nabla^{\perp}\phi.
\end{equation*}
We may then compute (recalling that $|\dot{\gamma}|=1$)
\begin{align*}
|\nabla_{e_0}^{\perp}e_0|^2&=\frac{1}{4\phi^2}|\nabla^{\perp}\phi|^2=\frac{1}{4\phi^2}|\nabla_{\mathbb{R}^3}^{\perp}\phi|^2_{\mathbb{R}^3},\\
|\nabla_{\dot{\gamma}}^{\perp}e_0|^2&=\frac{1}{\phi^2}|\dot{\gamma}\times\nabla^{\perp}\phi|^2=\frac{1}{\phi^2}|\nabla^{\perp}_{\mathbb{R}^3}\phi|^2_{\mathbb{R}^3},\\
|\nabla_{\dot{\gamma}}\dot{\gamma}|^2&=\frac{1}{\phi^2}|\gamma''|^2-\frac{1}{\phi^2}\langle\gamma'',\nabla^{\perp}\phi\rangle+\frac{1}{4\phi^2}|\nabla^{\perp}\phi|^2,\\
&=\frac{1}{\phi}|\gamma''|_{\mathbb{R}^3}^2-\frac{1}{\phi^{3/2}}\langle\gamma'',\nabla_{\mathbb{R}^3}^{\perp}\phi\rangle_{\mathbb{R}^3}+\frac{1}{4\phi^2}|\nabla^{\perp}_{\mathbb{R}^3}\phi|_{\mathbb{R}^3}^2,
\end{align*}
where $\mathbb{R}^3$ is used to indicate when we are using Euclidean norms and derivatives.  We see   that the squared norm of the second fundamental form of $\mu^{-1}(\gamma)$ is bounded if and only if
$$\phi^{-1}|\gamma''|_{\mathbb{R}^3}^2\quad\text{and}\quad \phi^{-2}|\nabla_{\mathbb{R}^3}^{\perp}\phi|^2_{\mathbb{R}^3}$$
are bounded.
\end{proof}

We shall now prove the following local uniqueness and stability result for mean curvature flow.

\begin{theorem}\label{thm:Strong_Stability}  Let $X$ be a hyperk\"ahler manifold from Example \ref{ex:MultiEH} or \ref{ex:MultiTN} with collinear points $p_1,\ldots,p_k$.  
Let $\gamma \subseteq \mathbb{R}^3$ be the straight line between $p_1$ and $p_2$ with midpoint $q$, and let $2d=\text{\emph{Length}}(\gamma)$. Suppose for $i>2$ that the Euclidean distance from $p_i$ to $q$ is strictly greater than $s d$ for $s\geq \max \lbrace 4, \sqrt{(k-2)/2} \rbrace$.

Then $\Sigma=\mu^{-1}(\gamma)$ is an embedded minimal $2$-sphere in $X$, there exists a tubular neighbourhood $U$ of $\Sigma$ in $X$ such that $\Sigma$ is the unique compact minimal submanifold of $U$ of dimension at least $2$, and mean curvature flow starting at any surface $\Gamma$ which is sufficiently $C^1$-close to $\Sigma$ will exist for all time and converge smoothly to $\Sigma$. 
\end{theorem}
\begin{proof}  By Lemma \ref{lem:min.calibrated}, $\Sigma$ is a holomorphic 2-sphere.  Holomorphic curves in $X$ are special Lagrangian for a different K\"ahler structure on $X$ in the hyperk\"ahler family: we will see this explicitly in $\S$\ref{sec:Lagrangians}.  	It then follows from \cite[Proposition A]{Tsai}  
that, in this case, the strong stability condition of \cite{Tsai} holds for $\Sigma$ if and only if the Gauss curvature of $\Sigma$ is positive. We show that this is indeed the case, under the hypotheses of the statement, in Proposition \ref{prop:Positive_Curvature} of Appendix \ref{sec:Positive_Curvature}.  The conclusions then follow from \cite[Theorems A and B]{Tsai}.
\end{proof}

\begin{remark}  One can generalize Theorem \ref{thm:Strong_Stability} to the setting of weak notions of (minimal) submanifolds and mean curvature flows.  Specifically, by \cite{LotaySchulze}, under the conditions of Theorem \ref{thm:Strong_Stability}, $\Sigma$ is unique in $U$ amongst stationary integral varifolds with support of dimension at least 2, and for any integral 2-current $\Gamma$ in $U$ which is homologous to $\Sigma$ and of mass strictly less than twice the area of $\Sigma$, there is an enhanced Brakke flow which exists and is non-vanishing for all time and converges \emph{smoothly} to $\Sigma$.  One can also remove the assumption about the points being collinear by suitably modifying the lower bound on $s$ (see \cite[Proposition 1.7]{Trinca}).
\end{remark}

Whilst the curve shortening flow in Euclidean space is by now a classical subject, the flow is actually rather poorly understood for general curves in $\mathbb{R}^3$.  One of the causes for difficulty, in contrast to the situation in $\mathbb{R}^2$, is that a space curve may cross itself as it evolves along the flow.  We therefore leave the general mean curvature flow of circle-invariant surfaces for future study and instead focus on the case where the curve is contained in a plane in the next section.  As we shall see, this has a natural interpretation in terms of symplectic geometry: namely, that the corresponding circle-invariant surfaces are Lagrangian for some choice of symplectic form in the hyperk\"ahler triple.

\section{Lagrangian spheres}\label{sec:Lagrangians}

In this section we study Lagrangian submanifolds, particularly spheres, in $X$ with respect to symplectic forms compatible with the hyperk\"ahler structure.  Recall that the twistor space of $X$ can be identified with the unit sphere $\mathbb{S}^2 \subset \mathbb{R}^3$. So given $v \in \mathbb{S}^2$ we shall denote by $\omega_v$ the symplectic structure associated with $v$ and the hyperk\"ahler metric.  We shall restrict to the case where $X$ is an ALE or ALF 
hyperk\"ahler 4-manifold given by the Gibbons--Hawking ansatz, though much of our discussion holds without the ALE or ALF hypothesis.

\subsection{Classifying invariant Lagrangian classes}

\begin{definition}[Lagrangian class]
A homology class $\delta \in H_2(X, \mathbb{Z})$ is a Lagrangian class with respect to some symplectic form $\omega$ if it can be represented by a Lagrangian cycle.\footnote{Lagrangian cycles are supposed to be smooth almost everywhere.} Furthermore, $\delta$ is an invariant Lagrangian class if it admits a circle-invariant Lagrangian representative.
\end{definition}

\noindent Note that the definition of invariant Lagrangian class  requires the Lagrangian representative to be invariant, not just that the class be invariant and Lagrangian.  However, we shall see that all Lagrangian classes with respect to the hyperk\"ahler forms are invariant Lagrangian classes.

To start, we  classify  invariant Lagrangians using the following elementary but useful observation.  

\begin{lemma}\label{lem:Lagrangian}
	Let $v \in \mathbb{S}^2$ and let $\gamma$ be a curve in  $\mathbb{R}^3$.  Then, 
	$$\omega_{v} |_{\mu^{-1}(\gamma)} = \langle \gamma' , v \rangle \vol_{\mu^{-1}(\gamma)},$$
	where $\gamma'$ is the velocity of $\gamma$ with respect to Euclidean arclength and $\langle.,.\rangle$ is the Euclidean inner product. Hence, a circle-invariant surface $\mu^{-1}(\gamma)$ in $X$ is Lagrangian for some symplectic form in the sphere of hyperk\"ahler 2-forms on $X$ if and only if $\gamma$ is contained in a plane.
\end{lemma}

\begin{proof}	Without loss of generality we may focus on $v=(0,0,1)\in\mathbb{S}^2$  so that $\omega_{v}$ is
	$$\omega_3=\eta\wedge d \mu_3+\phi d\mu_1\wedge d\mu_2.$$
	We see that if we write $\gamma=(\gamma_1,\gamma_2,\gamma_3)$ then
	$$\omega_3|_{\mu^{-1}(\gamma)}=\gamma_3' \eta \wedge ds,$$
	where $s$ denotes the Euclidean arclength parameter. The result  follows.
\end{proof}

Given a symplectic structure $\omega_v$, it is natural to ask which classes are Lagrangian with respect to $\omega_v$. 
Lemma \ref{lem:Lagrangian} allows us to answer this question as follows.  

\begin{corollary}[Lagrangian classes with respect to a fixed $\omega$]\label{cor:Lagrangian_Classes_From_Omega}
	Let $v \in \mathbb{S}^2$ and $\delta \in H_2(X, \mathbb{Z})$. Then $\delta$ is a Lagrangian class with respect to $\omega_v$ if and only if $\delta$ is a finite sum of classes of the form $[\mu^{-1}(\gamma)]$, with $\gamma$ curves connecting two singularities of $\phi$ lying in the same plane parallel to $v^{\perp}$.  
\end{corollary}
\begin{proof}
	Recall from Proposition \ref{prop:min.spheres} that the second homology of $X$ is generated by classes of the form $[\mu^{-1}(\gamma)]$ for $\gamma$ a straight line connecting singularities $p_i$, $p_j$ of $\phi$. Lemma \ref{lem:Lagrangian} shows that
	\begin{equation*}
	\langle [ \omega_v ] , [\mu^{-1}(\gamma)] \rangle= 2 \pi \int_{\gamma} \langle \gamma' , v \rangle  = 2 \pi  \langle p_j - p_i , v \rangle.
	\end{equation*}
Hence $[\mu^{-1}(\gamma)]$ is Lagrangian if and only if $\langle p_j-p_i,v\rangle=0$, from which the result follows.
\end{proof}

\begin{remark} Proposition \ref{prop:min.spheres} and
	Corollary \ref{cor:Lagrangian_Classes_From_Omega} show that all Lagrangian classes with respect to a given $\omega_v$ are invariant Lagrangian classes, represented by a union of embedded, circle-invariant, minimal Lagrangian 2-spheres that pairwise intersect at most one point.
\end{remark}

Suppose instead we are given $\delta \in H_2(X, \mathbb{Z})$ of the form $\delta= [\mu^{-1}(\gamma)]$, where $\gamma$ is a straight line connecting two singularities of $\phi$, and we want to know when $\delta$ is a Lagrangian class with respect to $\omega_v$. Since $\gamma$ is a straight line, we see from Lemma \ref{lem:Lagrangian} that $\mu^{-1}(\gamma)$ is Lagrangian with respect to $\omega_v$ if and only if $v\in (\gamma')^{\perp}$. As $v \in \mathbb{S}^2$ and $(\gamma')^{\perp} \cap\mathbb{S}^2\cong \mathbb{S}^1$, we record this result as follows.

\begin{corollary}[How many $\omega$ have $\delta$ as a Lagrangian class?]\label{cor:Omega_from_Lagrangian_Classes}
	 Let $\delta=[\mu^{-1}(\gamma)]$, for $\gamma$ a straight line connected two singularities of $\phi$. Then, there is an $\mathbb{S}^1 \subseteq \mathbb{S}^2$, given by $\mathbb{S}^1=(\gamma')^{\perp} \cap\mathbb{S}^2$, so that $\delta$ is a Lagrangian class with respect to $\omega_v$ if and only if $v\in\mathbb{S}^1$.
\end{corollary}

\begin{remark}
	Despite the existence results for Lagrangian classes in Corollaries \ref{cor:Lagrangian_Classes_From_Omega} and \ref{cor:Omega_from_Lagrangian_Classes}, it is also immediate from Lemma \ref{lem:Lagrangian} that for the generic $\omega$ compatible with the hyperk\"ahler structure there are no Lagrangian classes.
\end{remark}

\subsection{Thomas conjecture}

There is a well-known conjecture due to Thomas \cite{Thomas} asserting that the existence of a special Lagrangian representative in the Hamiltonian isotopy class of a zero Maslov Lagrangian should be equivalent to a notion of \emph{stability}.  We wish to describe this conjecture here with a view to proving a version of it in our  setting. We begin with some preliminaries.

\subsubsection{Graded Lagrangians}\label{sss:graded}
Without loss of generality fix the K\"ahler structure on $X$ to be $(\omega,I)=(\omega_3, I_3)$ (i.e.~choose $v=(0,0,1)\in\mathbb{S}^2$). By Lemma \ref{lem:Lagrangian}, the circle-invariant surface $\mu^{-1}(\gamma)$ is Lagrangian if and only if the curve $\gamma$ lies in a plane parallel to $\mu_3=0$, so we restrict to such curves. 

Now, there is a natural choice of $I$-holomorphic volume form on $X$ given by
$$\Omega=\omega_1+i\omega_2.$$
(Any other choice of $I$-holomorphic volume form satisfying the normalization condition $2\omega^2=\Omega \wedge \overline{\Omega}$ differs from $\Omega$ by multiplication by a unit complex number.)
This choice of $\Omega$ enables us to determine the phase of the (oriented) Lagrangian $\mu^{-1}(\gamma)$ as in Definition \ref{dfn:Lag.intro}: it is the function $e^{i\beta}: \mu^{-1}(\gamma) \to \mathbb{S}^1$ so that 
$$\Re(e^{-i\beta}\Omega)|_{\mu^{-1}(\gamma)}=\vol_{\mu^{-1}(\gamma)}.$$
Notice that this equality implies $\Im(e^{-i\beta}\Omega)|_{\mu^{-1}(\gamma)}=0$ by Wirtinger's inequality. 
Recall that if we parameterize $\gamma=\gamma(s)$ by Euclidean arclength then  $\vol_{\mu^{-1}(\gamma)}=\eta\wedge ds$.  Hence,
\begin{align*}
\Im(e^{-i\beta}\Omega) |_{\mu^{-1}(\gamma)} & = ( \cos(\beta)\omega_2 - \sin(\beta)\omega_1 ) |_{\mu^{-1}(\gamma)} = (\gamma_2'\cos(\beta)-\gamma_1'\sin(\beta)) \vol_{\mu^{-1}(\gamma)}
\end{align*}
and so
\begin{equation}\label{eq:Beta_angle_of_curve_with_the_x_axis}
\tan(\beta)=\tfrac{\gamma_2'}{\gamma_1'}
\end{equation}
for $\gamma_1'\neq 0$.  Moreover, we see that 
\[
\Re(e^{-i\beta}\Omega)|_{\mu^{-1}(\gamma)}=(\gamma_1'\cos\beta+\gamma_2'\sin\beta)\vol_{\mu^{-1}(\gamma)},
\]
which means that, in fact, \eqref{eq:Beta_angle_of_curve_with_the_x_axis} can be improved to:
\begin{equation}\label{eq:Beta_angle_of_curve_with_the_x_axis.2}
\cos\beta=\gamma_1'\quad\text{and}\quad \sin\beta=\gamma_2'.
\end{equation}
Thus, up to an integer multiple of $2 \pi$, the Lagrangian angle $\beta$ is the angle between $\gamma'$ and the $\mu_1$-axis (in the plane parallel to $\mu_3=0$ in which $\gamma$ lies).

Recall that the curvature $\kappa$ of $\gamma$ with respect to the Euclidean metric on $\mathbb{R}_{\mu_1,\mu_2}^2$ is defined by $\gamma''=\kappa N$, where $N=-I\gamma'$ is the unit normal vector such that $\lbrace \gamma', N \rbrace$ is an oriented basis of $\mathbb{R}_{\mu_1,\mu_2}^2$ with the standard orientation. From this, we immediately deduce the following.

\begin{lemma}[$\beta'$ is the curvature of $\gamma$]\label{lem:Lagrangian_Angle_Curvature}
Let $\gamma$ be an oriented planar curve in $\mathbb{R}^3$ such that $\mu^{-1}(\gamma)$ is Lagrangian with respect to some K\"ahler structure $(\omega_v,I_v)$ on $X$. Let $\kappa$ be the curvature of $\gamma$, viewed as a function on $\gamma$, and let $e^{i\beta}$ be the phase of $\mu^{-1}(\gamma)$ with respect to some choice of $I_v$-holomorphic volume form.  If $'$ denotes differentiation with respect to Euclidean arclength on $\gamma$, then 
$$\beta'=\kappa.$$
\end{lemma}
\noindent Lemma \ref{lem:Lagrangian_Angle_Curvature} is a reflection of the well-known relation between the Lagrangian angle and the mean curvature (i.e.~$H=I_v\nabla\beta$) and the formula \eqref{eq:Area_Length}.

\begin{definition}[Maslov class and gradings]
Recall that given a Lagrangian $\mu^{-1}(\gamma)$, we may use the phase $e^{i\beta}:\mu^{-1}(\gamma)\to\mathbb{S}^1$ to pullback the fundamental class on $\mathbb{S}^1$ and hence obtain 
$$\frac{1}{2\pi}[d \beta] \in H^1( \mu^{-1}(\gamma), \mathbb{Z}).$$
This is the Maslov class of the Lagrangian $\mu^{-1}(\gamma)$.

When $[d\beta]=0$ we say $\mu^{-1}(\gamma)$ is zero Maslov, and we may choose a well-defined lift of the Lagrangian angle $\beta:\mu^{-1}(\gamma) \to \mathbb{R}$.  A choice of such a lift is called a grading for the Lagrangian and a zero Maslov Lagrangian equipped with such a lift is called graded.
\end{definition}
  
\begin{example}
	Let $\gamma$ be a planar curve in $\mathbb{R}^3$ connecting two singularities of $\phi$ lying in the plane $\mu_3=0$ and meeting no other singularities of $\phi$. Then, $\mu^{-1}(\gamma)\cong S^2$ and as $H^1(S^2)=0$ we find that $\mu^{-1}(\gamma)$ is a zero Maslov class Lagrangian.
\end{example}
  
  An immediate consequence of  Lemma \ref{lem:Lagrangian_Angle_Curvature} is the following simple observation.

\begin{corollary}
	For $\gamma$ a closed planar curve in $\mathbb{R}^3$ not meeting the singularities of $\phi$, the Lagrangian $\mu^{-1}(\gamma)$ is zero Maslov if and only if 
	$$\int_{\gamma} \kappa =0.$$
\end{corollary}

\begin{example}\label{ex:Lagrangian_Torus}
	If $\gamma$ is a simple closed curve in a plane in $\mathbb{R}^3$, then
	$$\int_{\gamma} \kappa = \pm 2 \pi.$$
	If $\gamma$ does not meet any singularities of $\phi$, we deduce that $\mu^{-1}(\gamma)$ is an embedded Lagrangian $2$-torus of non-zero Maslov class.  However, if $\gamma$ contains one singularity of $\phi$, then $\mu^{-1}(\gamma)$ is a topological 2-sphere (a pinched 2-torus) and so is zero Maslov as its first cohomology vanishes.
\end{example}

There is a particularly important class of zero Maslov Lagrangians, namely, the ones whose Lagrangian angle is constant.

\begin{definition}[Special Lagrangians]
A Lagrangian is special (with angle $\beta$) if its Lagrangian angle $\beta$ is constant.  A special Lagrangian is calibrated by $\Re(e^{-i\beta}\Omega)$ and so is area-minimizing.  
\end{definition}

\begin{remark}[Special Lagrangians and holomorphic curves]\label{rmk:SL.holo}
When $\beta$ is constant, $\Re(e^{-i\beta}\Omega)=\omega_{\tilde{v}}$ for some $\tilde{v}\in\mathbb{S}^2$.  Hence any special Lagrangian is in fact complex for some complex structure in the sphere of hyperk\"ahler structures.  
\end{remark}

\begin{example}[Special Lagrangians and straight lines]\label{ex:SL.straight} Lemmas \ref{lem:Lagrangian} and \ref{lem:Lagrangian_Angle_Curvature} show that $\mu^{-1}(\gamma)$ is special Lagrangian if and only if $\gamma$ is a straight line, with angle $\beta$ equal to the angle between $\gamma$ and the $\mu_1$-axis.  
Thus, a straight line $\gamma$ connecting singularities of $\phi$ defines a special Lagrangian $S^2$.
\end{example}

\begin{example}[Ooguri--Vafa as an SYZ fibration]
Combining Example \ref{ex:OV.elliptic}, Remark \ref{rmk:SL.holo} and Example \ref{ex:SL.straight} shows that one can view the Ooguri--Vafa metric as a special Lagrangian torus fibration with a single nodal fibre.  This shows the importance of this metric from the point of view of the SYZ Conjecture, as mentioned earlier (c.f.~\cite{GrossWilson}).
\end{example}

\subsubsection{Lagrangian and Hamiltonian isotopies}

Two Lagrangian submanifolds $L_0$ and $L_1$ of the symplectic manifold $(X, \omega)$  are said to be Lagrangian isotopic if there is a continuous family of smooth maps $\Phi_t:L \to X$ so that $L_t=\Phi_t(L)$ is Lagrangian for all $t \in [0,1]$ and connects $L_0$ to $L_1$. They are furthermore called Hamiltonian isotopic if there is $\Phi_t:L \to X$ so that the path $L_t=\Phi_t(L)$ is generated by a time-dependent Hamiltonian $H_t:L_t \to \mathbb{R}$, i.e. 
$$dH_t (\cdot) = \omega \left( \partial_t\Phi_t , \cdot \right)|_{L_t}.$$ 
When $L_0$ (and hence $L_1$) has vanishing first Betti number, the notions of Lagrangian isotopy and Hamiltonian isotopy coincide (since the form $\omega \left( \del_t \Phi_t  , \cdot \right)|_{L_t}$ will be closed for a Lagrangian isotopy).

\subsubsection{(Graded) Lagrangian connect sum}

Given two Lagrangians $L_1$ and $L_2$ in $(X,\omega)$, intersecting at one point, since $X$ is 4-dimensional we can always form the Lagrangian connect sums $L_1\# L_2$ and $L_2\# L_1$, which are Lagrangians diffeomorphic to the topological connect sum and in the homology class $[L_1]+[L_2]$ (when this is defined).    
In fact, this construction is defined for Hamiltonian isotopy classes of Lagrangians, so $L_1 \# L_2$ and $L_2 \# L_1$ actually denote Hamiltonian isotopy classes of Lagrangians, which are uniquely defined (c.f.~\cite[$\S$4]{Thomas}) and not equal in general. 

When $L_1$ and $L_2$ are zero Maslov Lagrangians intersecting at a point $p$, given a grading in $L_2$ there is a unique grading of $L_1$ such that $L_1 \# L_2$ can be graded. In \cite{Seidel2}, a local Floer index $\mathrm{Ind}_p(L_1,L_2)$ at $p$ is defined and when $\mathrm{Ind}_p(L_2,L_1)=1$ then $L_1 \# L_2$ can be graded. In fact, $$\mathrm{Ind}_p(L_1,L_2) + \mathrm{Ind}_p(L_2,L_1) = \dim_{\mathbb{R}}L_i=2.$$ Hence, if $L_1$ and $L_2$ are initially graded, the connect sum $L_1\#L_2$ can be graded if and only if $L_2 \# L_1$ can.  We then equip both $L_1 \# L_2$ and $L_2 \# L_1$ with such a grading and call them graded Lagrangian connect sums.

\begin{example}\label{ex:Lagrangian_Connect_Sum}
	Let $\gamma_1$ and $\gamma_2$ be two planar curves connecting singularities of $\phi$ so that the endpoint $p$ of $\gamma_1$ coincides with the initial point of $\gamma_2$. Then, the Lagrangians $L_1=\mu^{-1}(\gamma_1)$, $L_2=\mu^{-1}(\gamma_2)$ intersect at a point and $L_1\# L_2=\mu^{-1}(\gamma)$ for $\gamma:=\gamma_1 \# \gamma_2$ a planar curve constructed as follows: outside a small neighbourhood of $p$, $\gamma$ coincides with $\gamma_1 \cup \gamma_2$, whereas round $p$ the curve $\gamma$ departs from $\gamma_1$ to connect with $\gamma_2$ in a clockwise manner.   This is illustrated in Figure \ref{fig_ConnectSum}.
	
	\begin{figure}[h]\label{fig:ConnectSum}
		\centering
		\includegraphics[scale=0.5]{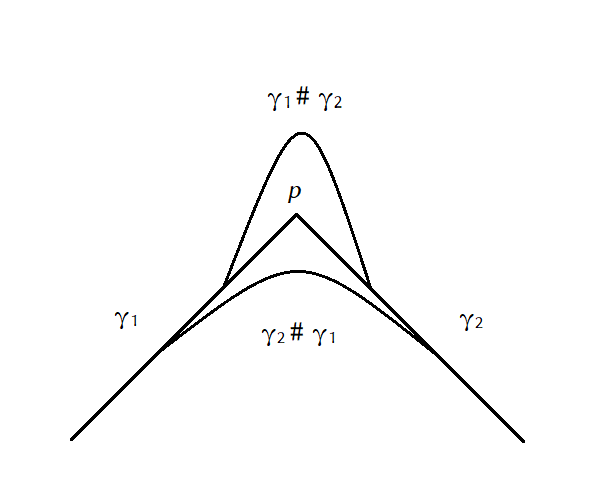}
		\caption{\label{fig_ConnectSum} Connect sums $\gamma_1 \# \gamma_2$ and $\gamma_2 \# \gamma_1$.}
	\end{figure}
\end{example}

\begin{definition}[Cohomological phase and slope]\label{dfn:Cohomological_Phase}
Let $L$ be a compact, oriented Lagrangian in $(X,\omega)$ and $\Omega$ the holomorphic volume form used to define the Lagrangian angle of $L$. The cohomological phase $e^{i\tau(L)}$ is the unit complex number such that  
	$$\int_L e^{-i\tau(L)} \Omega \in \mathbb{R}^+.$$
	(Here we can avoid the sign ambiguity in other definitions of the cohomological phase by assuming our initial orientation on $L$.)
	When the Lagrangian $L$ is graded and the variation of the Lagrangian angle is less than $2\pi$ we can canonically consider $\tau(L)$ to be a real number. 
	
The slope of $L$ is $\mu(L)=\tan(\tau(L)) \in \mathbb{R}$, which is well-defined independent of any grading.
\end{definition}

\begin{definition}[Stability]\label{dfn:stability}
	A compact zero Maslov Lagrangian $L$ is unstable if it is Hamiltonian isotopic to a graded Lagrangian connect sum $L_1 \# L_2$ of compact graded Lagrangians $L_1$ and $L_2$, with variations of their Lagrangian angles less than $2\pi$, so that $$\tau(L_1) \geq \tau(L_2).$$  The Lagrangian $L$ is called stable if it is not unstable.
\end{definition}

\begin{example}\label{ex:Cohomological_Phase}
	Let $v=(0,0,1)$ and $\gamma =(\gamma_1,\gamma_2,0) \subset \mathbb{R}^3$ be a curve connecting two singularities of $\phi$ lying in the plane $\mu_3=0$. Then, $e^{i\tau(L)}$ is a unit vector along the straight line connecting the endpoints of $\gamma$. Indeed, $\tau(L)$ is determined from:
\begin{align*}
0 & =\int_L \Im(e^{-i\tau} \Omega)  = \int_L( \cos \tau \ \omega_2 - \sin\tau \ \omega_1 )  = 2\pi \int_{\gamma} (\gamma_2'\cos\tau -\gamma_1'\sin\tau) \ ds, \\
0 & < \int_L \Re (e^{-i\tau} \Omega) = \int_L (\cos \tau \omega_1 + \sin \tau \omega_2) = 2 \pi \int_{\gamma} ( \gamma_1' \cos \tau + \gamma_2' \sin \tau ) \ ds.
\end{align*}
Thus, for $\int_{\gamma}\gamma_1'\neq 0$, we have
\begin{equation}
\sign(\cos \tau)=\sign(\int_\gamma \gamma_1' \ ds)
\quad\text{and}\quad
\label{eq:Cohomological_Lagrangian_Angle}
\tan(\tau)=\frac{\int_{\gamma}\gamma_2'ds}{\int_{\gamma}\gamma_1'ds}.
\end{equation} 
When $L$ is graded and the variation of the Lagrangian angle is less than $2\pi$ we can lift $\tau$ to be real valued. In this case, and bearing in mind that the grading $\beta$ takes values in $[0,2\pi)$, we have that $\tau(L) \in [0,2\pi)$ is the angle between the straight line connecting $\gamma$'s endpoints and the $\mu_1$-axis. 
\end{example}

\noindent With Example \ref{ex:Cohomological_Phase} in mind, the notion of stability in Definition \ref{dfn:stability} is illustrated by Figure \ref{fig_Stability}.

\begin{figure}[h]\label{fig:Stability}
	\centering
	\includegraphics[scale=0.5]{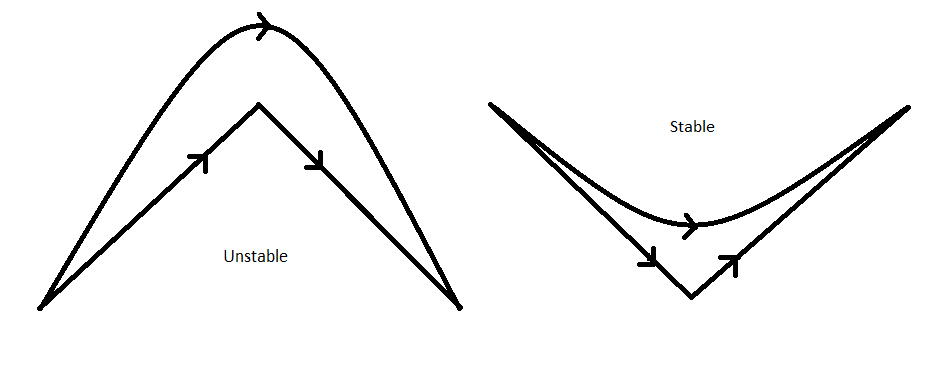}
	\caption{\label{fig_Stability} Projections of unstable and stable Lagrangians obtained from connect  sums.}
\end{figure}

\begin{remark}[Slope stability]
	For compact (oriented) Lagrangians $L$ we might consider an analogue of classical slope stability: i.e.~$L$ is slope unstable if it is Hamiltonian isotopic to a Lagrangian connect sum $L=L_1 \# L_2$ with $\mu(L_1) \geq \mu(L_2)$.  However, there are various subtleties here (including orientations of Lagrangians, for example) which mean that it does not seem prudent to pursue the analogy with stable bundles too closely here: see \cite[$\S$3]{Thomas} for further remarks.
\end{remark}

We may now state the Thomas conjecture (c.f.~\cite[Conjecture 5.1]{Thomas}).

\begin{conjecture}[Thomas conjecture]\label{conj:Thomas}
	Let $L$ be a compact zero Maslov Lagrangian in $(X,\omega,\Omega)$. Then, there is a 
	special Lagrangian in the Hamiltonian isotopy class of $L$ if and only if $L$ is stable, and the special Lagrangian is unique.
\end{conjecture}

\subsection{Special Lagrangians}

We first answer the question of when the minimal representatives of the classes in $H_2(X, \mathbb{Z})$ found in Proposition \ref{prop:min.spheres} are Lagrangian using Lemma \ref{lem:Lagrangian}.

\begin{corollary}
	Let $X$ be an ALE or ALF hyperk\"ahler 4-manifold as in Examples \ref{ex:MultiEH} and \ref{ex:MultiTN} with points $p_1,\ldots p_k\in\mathbb{R}^3$. There is a K\"ahler form $\omega$ compatible with the hyperk\"ahler structure on $X$ such that all the minimal surfaces given by Proposition \ref{prop:min.spheres} are Lagrangian with respect to $\omega$ if and only if all the points $p_1, \ldots , p_k$ lie in the same plane.
\end{corollary}

We now state and prove a version of Conjecture \ref{conj:Thomas} in the circle-invariant setting, for which we introduce some notation.  Let $L$ be a compact circle-invariant Lagrangian in $X$. First, we let $\mathrm{Ham}^{\U(1)}(L)$ denote its circle-invariant Hamiltonian isotopy class, i.e.~the set of Lagrangians  Hamiltonian isotopic to $L$ via circle-invariant Hamiltonians. Second, we say that $L$ is destabilized by a circle-invariant configuration if there is a representative of its class in $\mathrm{Ham}^{\U(1)}(L)$ of the form $L_1\# L_2$ as in Definition \ref{dfn:stability}, where $L_1$, $L_2$ are circle-invariant.

\begin{theorem}\label{thm:SL.stable}
	Let $X$ be an ALE or ALF hyperk\"ahler 4-manifold as in Example \ref{ex:MultiEH} and \ref{ex:MultiTN}, and let $L$ be a compact, embedded, circle-invariant, zero Maslov Lagrangian in $X$. There is a circle-invariant special Lagrangian in $\mathrm{Ham}^{\U(1)}(L)$ if and only if $L$ is not destabilized by a circle-invariant configuration.  Moreover, the special Lagrangian is unique in $\mathrm{Ham}^{\U(1)}(L)$.
\end{theorem}

\noindent Theorem \ref{thm:SL.stable} gives Theorem \ref{thm:Intro.SL.stable} in the introduction.

\begin{proof}
We start by better characterizing the Lagrangians $L$ in $X$ with respect to $\omega_v$, for some $v\in\mathbb{S}^2$, as in the statement of Theorem \ref{thm:SL.stable}. 

\begin{lemma}\label{lem:SL.stable.1} The Lagrangian $L=\mu^{-1}(\gamma)$ is a 2-sphere where $\gamma\subseteq\mathbb{R}^3$ is a curve connecting two singularities $p_1,p_2$ of $\phi$, and meeting no other singularities $p_i$ for $i>2$.  Moreover, $\gamma$ is contained in the plane $P_v$ orthogonal to $v$ and containing $p_1$ and $p_2$.
\end{lemma}

\begin{proof}
First, $L=\mu^{-1}(\gamma)$ for $\gamma\subseteq\mathbb{R}^3$ a curve lying in a plane $P_v$ perpendicular to $v$ by Lemma \ref{lem:Lagrangian}.  

If $\gamma$ has a singularity $p_i$ of $\phi$ as an interior point, then $L$ cannot be embedded at the point $\mu^{-1}(p_i)$.  If $\gamma$ has an immersed point $p$ which is not a singularity of $\phi$, then $L$ self-intersects along the circle $\mu^{-1}(p)$.  Therefore, for $L$ to be embedded, $\gamma$ is either a simple closed curve $\phi$ not meeting any singularities of $\phi$, or (the closure) of an open curve with two endpoints $p_1,p_2 \in \mathbb{R}^3$ which meets no other singularities of $\phi$. 

In the first case,  $L$ is a torus as in Example \ref{ex:Lagrangian_Torus} and so has nonzero Maslov class. We are therefore in the second case, so $P_v$ is uniquely determined by $v$, $p_1$ and $p_2$, and the result follows.
\end{proof}

We now want to understand $\mathrm{Ham}^{\U(1)}(L)$ in terms of curves.

\begin{lemma}\label{lem:SL.stable.2} We have that $\tilde{L}\in\mathrm{Ham}^{\U(1)}(L)$ if and only if $\tilde{L}=\mu^{-1}(\tilde{\gamma})$ for a curve $\tilde{\gamma}\subseteq P_v$ connecting $p_1, p_2$, which is isotopic to $\gamma$ through curves in $P_v\setminus\{p_i:i>2\}$ with endpoints $p_1$, $p_2$.
\end{lemma}

\begin{proof}
We observe that the notions of circle-invariant Hamiltonian and Lagrangian isotopy class agree for $L$. Indeed, if $\Phi_t:L_t \to X$ is a circle-invariant Lagrangian isotopy, 
since $b^1(L)=0$ we know there is $H_t :L_t \to \mathbb{R}$ such that $d H_t =\iota_{\partial_t \Phi_t }\omega$.  As $\iota_{\partial_t \Phi_t }\omega$ is circle-invariant and $\U(1)$ is compact, we can average $H_t$ over $\U(1)$ to obtain a circle-invariant Hamiltonian. Thus, a circle-invariant Lagrangian isotopy is actually a circle-invariant Hamiltonian isotopy. 

The result now follows from Lemmas \ref{lem:Lagrangian} and   \ref{lem:SL.stable.1}.
\end{proof}

Our next observation follows immediately from Proposition \ref{prop:min.spheres} and Example \ref{ex:SL.straight}.

\begin{lemma}\label{lem:SL.stable.3}  Let $\ell$ be the straight line from $p_1$ to $p_2$.
The only circle-invariant special Lagrangian representative of the homology class $[L] \in H_2(X)$ is $\mu^{-1}(\ell)$.
\end{lemma}

We now prove one direction of Theorem \ref{thm:SL.stable}.

\begin{proposition}\label{prop:SL.stable.1}
If $L$ is destabilized by a circle-invariant configuration, then there is no circle-invariant special Lagrangian in $\mathrm{Ham}^{\U(1)}(L)$.
\end{proposition}
\begin{proof} 
By assumption, there is a circle-invariant Hamiltonian isotopy from $L$ to a circle-invariant destabilizing configuration $L_1 \# L_2$ with $\tau(L_1) \geq \tau(L_2)$.  Therefore, $\mathrm{Ham}^{\U(1)}(L)=\mathrm{Ham}^{\U(1)}(L_1\# L_2)$.

 By Lemmas \ref{lem:SL.stable.1}--\ref{lem:SL.stable.2}, $L_1=\mu^{-1}(\gamma_1)$ and $L_2=\mu^{-1}(\gamma_2)$ where $\gamma_1,\gamma_2$ are curves in $P_v$ so that, up to relabelling the singularities of $\phi$ and reparametrising the curves, we have
 $$\gamma_1(0)=p_1,\quad \gamma_1(1)=p_3=\gamma_2(0), \quad \gamma_2(1)=p_2,$$ for some singularity $p_3$ of $\phi$. Then, $L_1\#L_2=\mu^{-1}(\gamma)$ with $\gamma=\gamma_1 \# \gamma_2$ as in Example \ref{ex:Lagrangian_Connect_Sum}. Since $\tau(L_1) \geq \tau(L_2)$ and $\gamma$ departs from $\gamma_1$ to meet $\gamma_2$ in a clockwise manner around $p_3$ as described in Example \ref{ex:Lagrangian_Connect_Sum}, we conclude that $p_3$ lies in the interior of the region bounded by $\gamma$ and the straight line $\ell$ connecting $p_1$ and $p_2$. Therefore, $\gamma$ and $\ell$ cannot be isotoped keeping the endpoints fixed without crossing the singularities of $\phi$. Lemma \ref{lem:SL.stable.2} then implies $\mu^{-1}(\ell) \notin \mathrm{Ham}^{\U(1)}(L_1\# L_2)=\mathrm{Ham}^{\U(1)}(L)$.  The result then follows from Lemma \ref{lem:SL.stable.3}.
\end{proof}

We now conclude by proving the other direction of Theorem \ref{thm:SL.stable}.

\begin{proposition}\label{prop:SL.stable.2}
If there is no circle-invariant special Lagrangian in $\mathrm{Ham}^{\U(1)}(L)$, then $L$ is destabilized by a circle-invariant configuration.
\end{proposition}

\begin{proof}
The hypothesis of the proposition, together with Lemmas \ref{lem:SL.stable.1}--\ref{lem:SL.stable.3}, state that $\gamma$ cannot be isotoped to the straight line $\ell$ connecting $p_1$ to $p_2$ through curves in $P_v$ with fixed endpoints without passing through another singularity of $\phi$. 

Consider an isotopy of curves $\Psi_t:[0,1] \to P_v$, for $t \in [0,1]$, with $\Psi_0=\gamma$ and $\Psi_1=\ell$. Let $p_3$ be the first singularity of $\phi$ which the isotopy crosses, which occurs at some $t_0\in (0,1)$.\footnote{By picking the isotopy in a generic manner we can assume that $\Psi_{t_0}(0,1)$ contains only one singularity of $\phi$.} Then, for all $t<t_0$, the open curves $\Psi_{t}(0,1)$ do not intersect any singularity of $\phi$ and $\Psi_{t_0}(0,1)$ only contains $p_3$. In particular, $\mu^{-1}(\gamma)$ and $\mu^{-1}(\Psi_t(0,1))$ are Hamiltonian isotopic for all $t<t_0$ but not for $t \geq t_0$.

By the Jordan curve theorem, we may assume that $p_3$ lies in the interior of the region bounded by $\gamma$ and $\ell$, otherwise we would just choose a different isotopy and rename the singularity we meet first in this interior as $p_3$. We know from construction that $p_3$ is also in the interior of the region bound by $\ell$ and $\Psi_{t}(0,1)$ for $t<t_0$.
	
	Now consider the curves $$\gamma_1=\Psi_{t_0}[0, \Psi_{t_0}^{-1}(p_3)]\quad\text{and}\quad\gamma_2=\Psi_{t_0}[\Psi_{t_0}^{-1}(p_3),1]$$ and the corresponding Lagrangians 
	$L_1=\mu^{-1} (\gamma_1)$ and $L_2=\mu^{-1}(\gamma_2)$. 
	Then, either $L_1 \# L_2$ or $L_2 \# L_1$ lies in $\mathrm{Ham}^{\U(1)}(L)$ by Lemma \ref{lem:SL.stable.2}. By possibly relabeling the points $p_1$ and $p_2$ we may assume that  $L_1 \# L_2 \in \mathrm{Ham}^{\U(1)}(L)$. 
	
	By \eqref{eq:Cohomological_Lagrangian_Angle}, the cohomological phases $\tau(L_1)$ and $\tau(L_2)$ coincide with the angles that the (oriented) straight lines $\ell_1$ connecting $p_1$ to $p_3$ and $\ell_2$ connecting $p_3$ to $p_2$ respectively make with $\ell$ (possibly after shifting both $\tau(L_1)$ and $\tau(L_2)$ by the same constant). 
Viewing $\ell$ as horizontal we have either: 
	\begin{itemize}
		\item[(a)] $\ell_1$ makes a non-negative angle with $\ell$, in which case $\ell_2$ makes a non-positive angle with $\ell$, and so $\tau(L_1)\geq \tau(L_2)$, 
		with the equality holding only when all lines are parallel; or
		\item[(b)] $\ell_1$ makes a negative angle with $\ell$ and $\ell_2$ makes a positive angle with $\ell$, so
		$\tau(L_1) < \tau(L_2)$.
	\end{itemize}
 Suppose that (b) holds and recall that $\gamma_1 \# \gamma_2$ goes round $p_3$ in a clockwise manner as described in Example \ref{ex:Lagrangian_Connect_Sum}. Then, $p_3$ would be outside the region bounded by $\gamma_1 \# \gamma_2$ and so $L_1 \# L_2$ would not be in $\mathrm{Ham}^{\U(1)}(L)$ which is a contradiction. Thus, (a) holds, so the result follows by definition of stability in Definition \ref{dfn:stability}.
\end{proof}

The main statement in Theorem \ref{thm:SL.stable} then follows from Propositions \ref{prop:SL.stable.1}--\ref{prop:SL.stable.2}, and the uniqueness statement is a consequence of Lemma \ref{lem:SL.stable.3}.
\end{proof}

\begin{remark}\label{rem:Area_Stability}
The proof of Theorem \ref{thm:SL.stable} actually holds under a replacement for the (in)stability condition: namely, that an unstable $L$ is one which can be decomposed as $L_1 \# L_2$ with 
	$$\mathrm{Area}(L) \geq \mathrm{Area}(L_1) + \mathrm{Area}(L_2),$$
for all $L$ in its Hamiltonian isotopy class. We will see that this is related to flow stability (Definition \ref{dfn:flow_stability}) in the Thomas--Yau conjecture.
\end{remark}

\subsection{Jordan--H\"older filtrations and decompositions}

We now discuss the Jordan--H\"older filtrations and decompositions for graded Lagrangians suggested in \cite{ThomasYau}. These are related to Bridgeland stability conditions and the Joyce conjectures on Lagrangian mean curvature flow (see \cite{JoyceConjectures}). We intend to return to some of those conjectures in the future. For now we shall simply explain how, in our setting, such Jordan--H\"older filtrations can be obtained and we start by recalling here the proposed definitions of \cite{ThomasYau}.

\begin{definition}[Subobjects]
	Given a graded Lagrangian $L$, we say  a graded Lagrangian $L_1$ is a subobject of $L$, written $L_1 \leq L$, if $L$ is Hamiltonian isotopic to a graded connect sum of the form $L_1 \# \tilde{L}$ which respects the initial grading on $L$. In this case, we denote $\tilde{L}$  as the ``quotient'' $L/L_1$.
\end{definition}

\begin{definition}[Jordan--H\"older filtration and decomposition]
	A Jordan--H\"older filtration of a graded Lagrangian $L$ is a sequence of graded Lagrangians
	$$L_1 \leq L_2 \leq \ldots \leq L_k=L ,$$
	such that the consecutive ``quotients'' $L_i/L_{i-1}$ are stable and
	$$\tau(L_1) \geq \tau(L_2) \geq \ldots \geq \tau(L_k).$$
	The union
	$$L_1 \cup L_2/L_1 \cup \ldots \cup L_k/L_{k-1}  $$
	is called a Jordan--H\"older decomposition of $L$.
\end{definition}

In our setting we can prove that there is a large class of graded Lagrangians which admit Jordan-H\"older filtrations and decompositions.

\begin{theorem}\label{thm:Jordan-Holder}
	Let $X$ be an ALE or ALF hyperk\"ahler 4-manifold as in Example \ref{ex:MultiEH} and \ref{ex:MultiTN}, and let $L$ be a compact, embedded, circle-invariant, graded Lagrangian in $X$. If the grading $\beta:L \to \mathbb{R}$ is a perfect Morse function, then $L$ admits a  Jordan--H\"older decomposition.
\end{theorem}

\begin{proof}
	Given that $(L, \beta)$ is circle-invariant, we can write $L$ as $\mu^{-1}(\gamma)$ with $\beta$ descending to a function on $\gamma$. Consider the straight-line $\ell$ connecting the endpoints of $\gamma$.     Then, $\mu^{-1}(\ell)$ is the unique special Lagrangian in the homology class of $L$ and, up to changing the grading of all Lagrangians by the same constant, we may suppose that for any circle-invariant graded Lagrangian $\tilde{L}=\mu^{-1}(\tilde{\gamma})$, the grading $\beta(\tilde{L})$ coincides with the angle that $\tilde{\gamma}'$ makes with $\ell$. Furthermore, denoting by $\tilde{\ell}$ the straight-line connecting the endpoints of $\tilde{\gamma}$, $\tau(\tilde{\gamma})$ is the angle that $\tilde{\ell}$ makes with $\ell$.

Since $L=\mu^{-1}(\gamma)$ is an embedded sphere, the condition that $\beta$ is a perfect Morse function on $L$ implies that it has a unique maximum and minimum, which must therefore correspond to the endpoints of $\gamma$.  As a consequence, $\beta'$ never vanishes in the interior of $\gamma$.   Recall that by Lemma \ref{lem:Lagrangian_Angle_Curvature} we have $\beta'=\kappa$, so the region $B$ in the plane bounded by $\gamma \cup \ell$ must be convex. 

Let $\bigtriangleup$ be the closed convex hull of the singularities of $\phi$ enclosed in $B$. Given that the endpoints of $\gamma$ are singularities of $\phi$, $\bigtriangleup$ has $\ell$ as one of its facets and the remaining ones give a sequence of straight-lines $\ell_1, \ldots , \ell_k$ with consecutive initial and endpoints, i.e.
	$$\partial \bigtriangleup = \ell \cup \ell_1 \cup \ldots \cup \ell_k .$$
	Up to rotation we may suppose that $\ell$ is horizontal and $\ell_1 \cup \ldots \cup \ell_k$ is above $\ell$. The construction is such that $\gamma$ can be isotoped through planar curves to $\ell_1 \# \ldots \# \ell_k$ without crossing any singularity of $\phi$: see Figure \ref{Jordan-Holder}. 
	\begin{figure}[h]\label{fig:Jordan-Holder}
		\centering
		\includegraphics[scale=0.5]{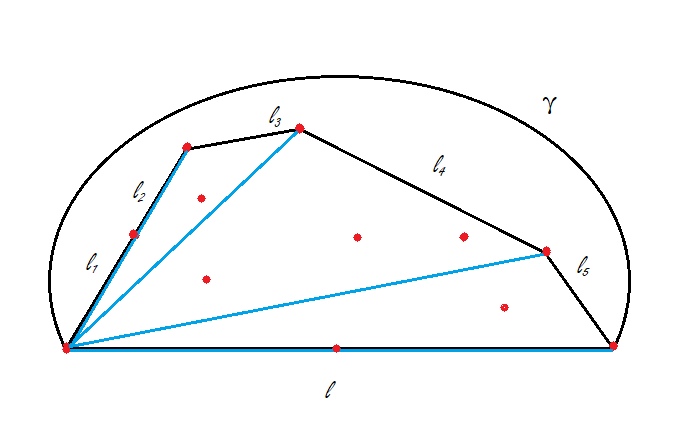}
		\caption{\label{Jordan-Holder} Red points represent singularities of $\phi$: there are none in the region bounded by $\gamma$ and $\ell_1 \# \ldots \# \ell_k$.}
	\end{figure}
	Thus,  $L=\mu^{-1}(\gamma)$ is Hamiltonian isotopic to $\mu^{-1}(\ell_1 \# \ldots \# \ell_k)$ by Lemma \ref{lem:SL.stable.2}. We now define
	\begin{align*}
	\gamma_1 & := \ell_1,\quad
	\gamma_2  := \ell_1 \# \ell_2,\quad \ldots \quad
	\gamma_k := \ell_1 \# \ldots \# \ell_k , 
	\end{align*}
	and $L_i :=\mu^{-1}(\gamma_i)$ for $i=1, \ldots , k$. By construction, the sequence $\tau(L_i)$ is decreasing: see Figure \ref{Jordan-Holder}, where the $\tau(L_i)$ are the angles between the blue lines and the horizontal line $\ell$. Furthermore, the consecutive ``quotients'' $L_{i}/L_{i-1}$ are Hamiltonian isotopic to $\mu^{-1}(\ell_i)$, which is a special Lagrangian and thus stable. This proves the statement.
\end{proof}

\begin{remark}
	Alternatively, the construction in the proof of Theorem 
	\ref{thm:Jordan-Holder} can be made in a similar way to \cite[Section 5.3]{ThomasYau}.  Specifically, one first finds a subobject of $L$ of maximal phase and, among those, one of minimal area. Repeating this for the quotient of $L$ by such a subobject and proceeding inductively gives the Jordan--H\"older filtration in Theorem \ref{thm:Jordan-Holder}.
	
However, we point out that in our circle-invariant setting, these strategies do not seem to work without the assumption that $\beta$ is a perfect Morse function. Indeed,  Figure \ref{Bad_Curve} shows a curve defining a graded Lagrangian which does not appear to have a circle-invariant Jordan--Holder filtration. 

The question of what happens for the Lagrangian mean curvature flow starting at such a Lagrangian is a tantalizing one which we shall address in future work by relating this issue with the predictions of Joyce's conjectures \cite{JoyceConjectures}.

		\begin{figure}[h]\label{fig:Bad_Curve}
		\centering
		\includegraphics[scale=0.6]{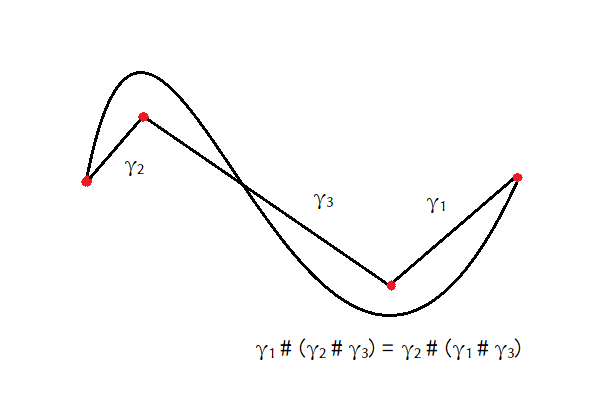}
		\caption{\label{Bad_Curve} Circle-invariant filtrations of this curve do not satisfy the monotonicity condition on the cohomological angles, or lead to some unstable quotients.}
	\end{figure}
\end{remark}

\subsection{Seidel's symplectically knotted 2-spheres}\label{ss:Seidel}

In this section we will prove an invariant version of a deep result of Seidel \cite{Seidel}. Namely, we shall explicitly construct an infinite family of embedded, circle-invariant, Lagrangian $2$-spheres which are symplectically knotted, i.e.~no two of them are Lagrangian isotopic, even though they are isotopic through embedded (non-Lagrangian) $2$-spheres. In our case, we restrict to the circle-invariant setting, with the fact they cannot be unknotted through possibly non-invariant Lagrangian isotopies due to Seidel's original work. Nevertheless, we emphasize that, even though yielding a weaker result, our work only uses elementary methods with no need of Floer homology.

Let $X$ be a Gibbons--Hawking hyperk\"ahler $4$-manifold where $\phi$ has at least three singularities $\lbrace p_0 , p_1, p_2 \rbrace$, so that there is an open set $U \ni  p_0 , p_1, p_2$, diffeomorphic to a ball in $\mathbb{R}^3$, which contains no other singularity of $\phi$. Suppose further that the straight lines $\ell_1$ and $\ell_2$, connecting $p_0$ to $p_1$ and $p_2$ respectively, only intersect at $p_0$ and are contained in $U$. We also let $\ell_{\infty}$ be the infinite ray starting at $p_1$ so that $\ell_1\cup\ell_{\infty}$ is an infinite ray starting at $p_0$ (so $p_0$ does not lie on $\ell_{\infty}$).  For $r \in \mathbb{N}$ we define $\gamma^{r} \subset U$ to be a simple curve in the plane defined by $p_0,p_1,p_2$ which starts at $p_1$, never intersects $\ell_1$, and intersects both $\ell_2$ and $\ell_{\infty}$ exactly $r$ times before it ends at $p_2$, see Figure \ref{fig:Curve_Gamma} below.\footnote{One can choose $\gamma^r$ with positive geodesic curvature and intersecting $\ell_2$ and $\ell_{\infty}$ with the same angle for each $r$.} Then, for $i=1,2$ let $L_i=\mu^{-1}(\ell_i)$, and for $r \in \mathbb{N}$ let 
\begin{equation}\label{eq:Lr}
 L^r:=\mu^{-1}(\gamma^r).
\end{equation}

\begin{figure}[h]
	\centering
	\includegraphics[width=0.33\textwidth,height=0.25\textheight]{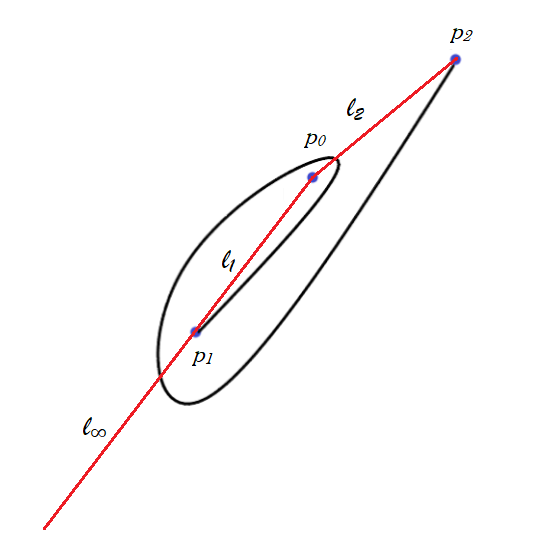}
	\caption{\label{fig:Curve_Gamma} The curve $\gamma^1$.}
\end{figure}

\begin{remark}
	The Lagrangian $2$-spheres $L^r$ are related to those constructed in \cite{Seidel}. Indeed, using Arnold's generalized Dehn twist $\tau_{L_1}$ around $L_1$ we have
	$$L^r=\tau_{L_1}^{2r}(L_2).$$
\end{remark}

It is immediate from their definition that for $r \neq s$ one cannot isotope $\gamma^r$ to $\gamma^s$ through planar curves fixing the end-points and without passing through $p_0$. Indeed, suppose $r>s$, then we may consider the curve $\gamma^r \cdot \overline{\gamma^s}$ obtained by concatenating $\gamma^r$ with  $\gamma^s$ with its opposite orientation. The result is an immersed closed curve which can be isotoped through planar curves, without crossing the singularities of $\phi$, to a closed curve containing $p_0$ in its interior and index $r-l$. In terms of the Lagrangian $2$-spheres $L^r$, this can be stated as follows.

\begin{proposition}[Knotted Lagrangian 2-spheres]
	For all $r,s \in \mathbb{N}$ with $r \neq s$, the Lagrangian 2-spheres in 
	\eqref{eq:Lr} satisfy $L^r \notin \mathrm{Ham}^{\U(1)}(L^s)$.
\end{proposition}

\begin{remark}
	Seidel \cite{Seidel}  showed further  that, under the additional hypothesis that there is $p_3$ so that there is a path $\gamma$ connecting it to $p_1$ with $\gamma \cap \ell_2=\emptyset$ and $\gamma \cap \ell_1 = \lbrace p_1 \rbrace$, $L^r$ and $L^s$ are not Lagrangian isotopic for $r \neq s$.
\end{remark}

\section{Lagrangian mean curvature flow}\label{sec:Lagrangian_Mean_Curvature_Flow}

The Thomas--Yau conjecture \cite{ThomasYau} builds upon the Thomas Conjecture (Conjecture \ref{conj:Thomas}),  
stating that a different notion of stability, which we call flow stability, implies long-time existence and convergence of Lagrangian mean curvature  flow in a Calabi--Yau manifold. 

Throughout we shall restrict to the flow in $X$, where $X$ is an ALE or ALF gravitational instanton given by the Gibbons--Hawking ansatz, though some analysis will go through without the ALE or ALF hypothesis.  We shall typically view $X$ as a Calabi--Yau manifold by taking, without loss of generality, $\omega=\omega_3$ as the K\"ahler form and $\Omega=\omega_1+i\omega_2$ as the holomorphic volume form.  

To define flow stability, we recall the notion of almost calibrated Lagrangians.

\begin{definition}[Almost calibrated]\label{dfn:almost.calibrated}
Let $L^2$ be an oriented, zero Maslov Lagrangian in $(X^4,\omega,\Omega)$ with Lagrangian angle $e^{i\beta}:L\to\mathbb{S}^1$.  Then $L$ is almost calibrated if there is a grading $\beta$ of $L$ such that, for some $\delta>0$, 
$$\sup_L\beta-\inf_L\beta< \pi-\delta.$$
If $L$ is compact, $L$ is almost calibrated if and only if there is a constant $\beta_0$ so that $\Re(e^{-i\beta_0}\Omega)|_L>0$.
\end{definition}

Recall the notion of cohomological phase $e^{i\tau(L)}$ of a compact, oriented Lagrangian in Definition \ref{dfn:Cohomological_Phase}.  If $L$ is almost calibrated, we can view $\tau(L)$ as a real number and, by multiplying $\Omega$ by a unit complex number, we can always ensure that $\tau(L)=0$.  We may now define flow stability.

\begin{definition}[Flow stability]\label{dfn:flow_stability}  
Let $L$ be a compact, almost calibrated Lagrangian   in $(X,\omega,\Omega)$ such that $\tau(L)=0$. Then $L$ is flow stable if for any decomposition of $L$, in its Hamiltonian isotopy class, as a graded Lagrangian connect sum $L_1 \# L_2$ of almost calibrated Lagrangians $L_1,L_2$ we have
	\begin{itemize}
	\item[(a)] $[\tau(L_1),\tau(L_2)]\nsubseteq (\inf_L\beta,\sup_L\beta)$, or
	\item[(b)] $\mathrm{Area}(L) < \int_{L_1} e^{-i\tau(L_1)}\Omega + \int_{L_2} e^{-i\tau(L_2)}\Omega $.\footnote{Thomas--Yau \cite{ThomasYau} write a weak inequality but we believe it was their intention to write a strict one.}\end{itemize}
	Note that the condition (b) implies that $\mathrm{Area}(L)< \mathrm{Area}(L_1)+\mathrm{Area}(L_2)$, c.f.~Remark \ref{rem:Area_Stability}. 
\end{definition}  

\noindent Figure \ref{fig_Stability} shows curves defining invariant Lagrangians which are flow unstable and flow stable.

\begin{remark}[The almost calibrated condition] Whilst the almost calibrated condition does not explicitly appear in \cite{ThomasYau}, as stated in that paper, the Thomas--Yau conjecture should really be made only for Lagrangians with some bound on the variation of the Lagrangian angle.  In \cite{ThomasYau}, they prove the conjecture in a special case assuming a stronger bound on the Lagrangian angle than the almost calibrated condition which, from the purely geometric viewpoint, is not a natural assumption, though necessary to make their analysis go through.  Here, we prove an invariant version of the Thomas--Yau conjecture only assuming that the initial Lagrangian is almost calibrated. 

The class of almost calibrated Lagrangians has received significant study, both in terms of Lagrangian mean curvature flow \cites{LambertLotaySchulze,Wang} and in other contexts \cites{Donaldson,Solomon,Thomas}. Moreover, by work of Neves \cite{NevesSingularities} we know that, given any compact Lagrangian $L$ in a Calabi--Yau manifold, we can always find another compact Lagrangian $L'$ arbitrarily $C^0$-close and Hamiltonian isotopic to the initial one  (which is embedded if $L$ is, but is \emph{not} almost calibrated even if $L$ is) such that Lagrangian mean curvature flow starting at $L'$ develops a finite time singularity.  The fact that $L'$ is not almost calibrated plays a crucial role in the proof that a singularity forms. Therefore, without the almost calibrated assumption, we are guaranteed to have Lagrangians in the Hamiltonian isotopy class for which Lagrangian mean curvature flow becomes singular in finite time.
\end{remark}

We will prove the following version of the Thomas--Yau conjecture in the setting of hyperk\"ahler 4-manifolds given by the Gibbons--Hawking ansatz.

\begin{theorem}\label{thm:LMCF.stable}
	Let $X$ be an ALE or ALF hyperk\"ahler 4-manifold given by the Gibbons--Hawking ansatz as in Examples \ref{ex:MultiEH} and \ref{ex:MultiTN} and let $L$ be an embedded, compact, almost calibrated, circle-invariant Lagrangian   in $X$.  
	If $L$ is flow stable in the sense of Definition \ref{dfn:flow_stability}, then Lagrangian mean curvature flow starting at $L$ exists for all time.  Moreover, the flow converges smoothly to the unique circle-invariant special Lagrangian in the Hamiltonian isotopy class of $L$, given by Theorem \ref{thm:SL.stable}.
\end{theorem}

\noindent The setting of Theorem \ref{thm:LMCF.stable} is similar to \cite[\S 7]{ThomasYau}, but the main differences are we use the actual hyperk\"ahler metric and Lagrangian mean curvature flow in the proof (rather than a modified, non Ricci-flat metric, and thus a modified flow as in \cite{ThomasYau}), and we have the optimal result that the initial Lagrangian is almost calibrated (rather than the stronger bound on the Lagrangian angle assumed in \cite{ThomasYau}).  However, given the similarities in the two situations, one could use similar techniques to those employed in \cite{ThomasYau} but with some key modifications, though we mainly take our own approach here.

\subsection{Flow of planar curves} In Proposition \ref{prop:curve.flow}  we re-cast the mean curvature flow of circle-invariant surfaces in $X$ as a flow \eqref{eq:Mean_Curvature_Flow_E} of curves in $\mathbb{R}^3$, which we rewrite here for convenience:
\begin{equation}\label{eq:Mean_Curvature_Flow_E_2}
\frac{\partial \gamma}{\partial t} = \phi^{-1} \gamma'',
\end{equation}
where $'$ denotes the derivative with respect to Euclidean arclength.  Since Lemma \ref{lem:Lagrangian} shows that curves in $\mathbb{R}^3$ correspond to Lagrangians if and only if they are planar, we want to study the evolution of planar curves along \eqref{eq:Mean_Curvature_Flow_E_2}.

We consider two types of solutions $\gamma_t$ of \eqref{eq:Mean_Curvature_Flow_E_2}: simple closed curves and open curves with fixed ends given by singularities of $\phi$.
More precisely, we write these closed and open problems as either
\begin{itemize}
	\item[(C)] $\gamma_t$ is a simple closed curve in $\mathbb{R}^3$ not meeting any singularities of $\phi$,	or
	\item[(O)] $\gamma_t$ is an embedded arc in $\mathbb{R}^3$ with endpoints $p_1$ and $p_2$, where $p_1,p_2$ are singularities of $\phi$.
	\end{itemize}
Notice that if $p_1,p_2$ are singular points of $\phi$, then $\partial_t \gamma_t (p_i)=0$ for $i=1,2$ by \eqref{eq:Mean_Curvature_Flow_E_2}, which suggests that condition (O) is preserved.  Moreover, if $\gamma_t$ satisfies (C) then $\mu^{-1}(\gamma_t)$ is an embedded 2-torus, and if $\gamma_t$ satisfies (O) then $\mu^{-1}(\gamma_t)$ is an embedded 2-sphere.
We shall now prove the following.

\begin{lemma}\label{lem:planar.curves}
	Let $\gamma_0 \subseteq \mathbb{R}^3$ be a curve satisfying condition (C) or (O) above for $t=0$.  Then there is $T>0$ and a unique smooth solution $\gamma_t$ to \eqref{eq:Mean_Curvature_Flow_E_2} for $t\in[0,T)$ starting at $\gamma_0$,  which satisfies (C) or (O) respectively. 
	
Suppose further that $\gamma_0$ lies in a plane $P_v$ orthogonal to $v\in\mathbb{S}^2$.	Then $\gamma_t$ lies in the plane $P_v$ for all $t\in[0,T)$.
\end{lemma}

\begin{proof}
Proposition \ref{prop:curve.flow} shows that \eqref{eq:Mean_Curvature_Flow_E_2} is equivalent to mean curvature flow of $\mu^{-1}(\gamma)$.  If $\gamma_0$ satisfies (C) or (O) for $t=0$ then $L_0=\mu^{-1}(\gamma_0)$ is a smooth embedded surface in $X$. Hence, there is a unique solution $L_t$ to  mean curvature flow in $X$ starting at $L_0$ for $t\in[0,T)$, for some $T>0$.  Moreover, $L_t$ remains circle-invariant, and will stay embedded for  $T$ sufficiently small. Since $L_t$ is of the form $\mu^{-1}(\gamma_t)$ and its topology does not change, $\gamma_t$ satisfies conditions (C) or (O) respectively.

Furthermore, if $\gamma_0$ lies in $P_v$ then $L_0$ is Lagrangian by Lemma \ref{lem:Lagrangian}.  Since $X$ is Calabi--Yau and $L_0$ is compact, $L_t$ is then Lagrangian for all $t\in [0,T)$ \cite{Smoczyk}.  Hence, $\gamma_t$ lies in a plane orthogonal to $v$ for all $t$ by Lemma \ref{lem:Lagrangian}, i.e.~$\langle\gamma_t',v\rangle=0$ for all $t$.  Thus, if $s$ denotes arclength along $\gamma_t$,
\begin{align*}
\frac{\partial}{\partial t}\langle \gamma_t,v\rangle&=\langle \phi^{-1}\gamma_t'',v\rangle=\phi^{-1}\frac{\partial}{\partial s}\langle\gamma_t',v\rangle=0
\end{align*}
for all $t$, where we used \eqref{eq:Mean_Curvature_Flow_E_2}.  We conclude that $\gamma_t$ lies in $P_v$ for all $t$ as required.
\end{proof}

\noindent One can prove Lemma \ref{lem:planar.curves} directly just using the flow \eqref{eq:Mean_Curvature_Flow_E_2}, but this is more involved.

Now that we know that planar curves remain planar under \eqref{eq:Mean_Curvature_Flow_E_2}, we want to understand when they become singular along the flow.   If $\gamma$ denotes a planar curve, we let $\kappa$ denote the curvature of $\gamma$ in the plane, viewed as a function on $\gamma$.

\begin{lemma}\label{lem:flow.exist}
Let $\gamma_0$ be a curve in a 2-plane $P$ in $\mathbb{R}^3$ satisfying conditions (C) or (O) above for $t=0$.
The solution $\gamma$ of \eqref{eq:Mean_Curvature_Flow_E_2} in $P$ starting at $\gamma_0$ exists and remains embedded as long as $\phi^{-1}\kappa^2$ and $|\nabla^{\perp}_{\mathbb{R}^3}\log\phi|_{\mathbb{R}^3}^2$ remain bounded and the flow does not reach any other singularities of $\phi$. Furthermore, the flow remains in a bounded region in $P$.
\end{lemma}

\begin{proof}
Since the flow \eqref{eq:Mean_Curvature_Flow_E_2} is parabolic away from the singularities of $\phi$, an  application of the maximum principle shows that $\gamma$ stays embedded as long as the flow remains smooth and does not meet any other singularities of $\phi$. The stated sufficient conditions for the flow to remain smooth follow directly from Proposition \ref{prop:curve.flow.blow.up}.

Finally, we prove that the flow stays within a bounded region in $P$.  By translating coordinates, we may assume that $P$ contains the $0\in\mathbb{R}^3$ and that $0$ lies in the interior of $\gamma_0$. Take a large circle $\gamma_{R}$ of radius $R$ centred at $0$ in $P$ and, given $c>0$, consider the family $\gamma_{\sqrt{R^2-2ct}}$ of circles depending on $t$. For $c=1$ this curve is a solution to the usual curve shortening flow \eqref{eq:curve.shortening.flow} in the plane. In a fixed large annular region sufficiently far from the origin we know that $\phi^{-1}$ is uniformly bounded from both below and above. Thus, by choosing $c$ appropriately we can use these circles $\gamma_{\sqrt{R^2-2ct}}$ as barriers in this region to ensure (again by the maximum principle) that given an embedded curve satisfying \eqref{eq:Mean_Curvature_Flow_E_2} there exists a compact set in $P$ so that the flow remains in this set.  
\end{proof}

Lemmas \ref{lem:planar.curves} and \ref{lem:flow.exist} have the following important corollary.

\begin{corollary}\label{cor:tori.nosings}  Let $L_0$ be an embedded circle-invariant Lagrangian torus in an ALE or ALF hyperk\"ahler 4-manifold $X$ given by the Gibbons--Hawking ansatz as in Examples \ref{ex:MultiEH} or \ref{ex:MultiTN}.  Suppose further that the planar curve $\gamma_0=\mu(L_0)$ lies in a plane containing no singularities of $\phi$.  Then, the solution $\gamma_t$ to \eqref{eq:Mean_Curvature_Flow_E_2} starting at $\gamma_0$ shrinks to a point in finite time, so the corresponding solution $L_t=\mu^{-1}(\gamma_t)$ remains an embedded torus until it collapses to a circle orbit in finite time.  
\end{corollary}

\noindent Notice that there is no assumption on the initial Lagrangian torus other than it is embedded and circle-invariant.

\begin{proof}
Let $P$ be the plane containing $\gamma_0$ and hence the flow $\gamma_t$.  Since $P$ contains no singularities of $\phi$, $\phi^{-1}$ is uniformly bounded below on $P$ by $a>0$.  Moreover, by Lemma \ref{lem:flow.exist}, the flow stays in a bounded region in $P$ and hence $\phi^{-1}$ is bounded above by $b>0$ in this region.

Therefore, the flow $\gamma$ is uniformly controlled by the curve shortening flow \eqref{eq:Curve_Shortening_Flow} in the Euclidean plane $P$.  Hence, by Grayson's Theorem \cite{Grayson}, the flow $\gamma$ will shrink to a point in finite time.  
\end{proof}

\noindent In the setting of Corollary \ref{cor:tori.nosings}, the flow of planar curves starting at $\gamma_0$ will become convex and then shrink to a round point in finite time, i.e.~it has a Type I singularity whose Type I blow-up is the shrinking circle.  Therefore, the Lagrangian mean curvature flow starting at $L_0$ has a finite-time Type I singularity whose Type I blow-up is a shrinking cylinder $S^1\times\mathbb{R}$.  We now discuss flows where other singularities can occur.

\begin{example}[Clifford torus]
The Clifford torus in $\mathbb{R}^4$ is an embedded circle-invariant Lagrangian torus which can be defined as $\mu^{-1}(\gamma_0)$, where $\gamma_0$ is a planar circle of radius $1$ centred at the origin in $\mathbb{R}^3$.  If we assume that $$\gamma_t=r(t)\gamma_0$$ then 
$$\gamma_t''=-\frac{1}{r(t)}\gamma_0.$$ 

 Therefore, if we consider the Lagrangian mean curvature flow in Euclidean $\mathbb{R}^4$ starting at the Clifford torus, we see by Example \ref{ex:flat} and \eqref{eq:Mean_Curvature_Flow_E_2} that the flow becomes the ODE:
$$\dot{r}(t)=-\frac{2r(t)}{r(t)}=-2.$$
Hence, 
$$r(t)=r(0)-2t.$$
We deduce the well-known fact that the Clifford torus is a self-shrinker for Lagrangian mean curvature flow, so it has a Type I singularity at the origin in finite time and is its own Type I blow-up.

If instead we consider the Lagrangian mean curvature flow in Taub--NUT $\mathbb{R}^4$ starting  at the Clifford torus, we see by Example \ref{ex:TN} and \eqref{eq:Mean_Curvature_Flow_E_2} that the flow becomes:
$$\dot{r}(t)=-\frac{1}{r(t)(m+\frac{1}{2r(t)})}=-\frac{2}{2mr(t)+1}.$$
Hence, $r(t)$ is the positive solution to the quadratic equation
$$mr(t)^2+r(t)=mr(0)^2+r(0)-2t.$$
Notice that this reduces to the Euclidean case when $m=0$.  One observes that again the flow starting at the Clifford torus in Taub--NUT shrinks to the origin in finite-time.  Moreover, it is clear that the flow has a Type I singularity whose Type I blow-up is the usual Clifford torus in $\mathbb{R}^4$.
\end{example}

\subsection{Long-time existence and stability (Proof of Theorem \ref{thm:LMCF.stable})}
As we have seen, circle-invariant Lagrangian mean curvature flow  is equivalent to the flow \eqref{eq:Mean_Curvature_Flow_E_2} for curves $\gamma$ in a plane, which we may identify with $\mathbb{R}^2$ (and we restrict $\phi$ to this plane).  By Lemma \ref{lem:SL.stable.1}, any Lagrangian in Theorem \ref{thm:LMCF.stable} is of the form $L=\mu^{-1}(\gamma)$, where $\gamma$ is an embedded planar arc which has singularities $p_1,p_2$ of $\phi$ as endpoints and meet no other singularities $p_i$ for $i>2$ of $\phi$. In proving Theorem \ref{thm:LMCF.stable}, we may therefore restrict ourselves to the Dirichlet problem (O) above for \eqref{eq:Mean_Curvature_Flow_E_2} 
for embedded arcs $\gamma$ in $\mathbb{R}^2$ which have endpoints at singularities $p_1$ and $p_2$ of $\phi$ in $\mathbb{R}^2$, with initial curve $\gamma_0=\mu(L_0)$.    

We start with an elementary, but important observation.  

\begin{lemma}\label{lem:almost.calibrated}
If the curve $\gamma_0$ in $\mathbb{R}^2$ is almost calibrated (i.e.~the variation of the angle it makes with a fixed axis is less than $\pi$), then the variation of the angle for the solution of \eqref{eq:Mean_Curvature_Flow_E_2} is non-increasing along the flow, and $\gamma_t$ remains almost calibrated as long as the flow exists.  
\end{lemma}

\begin{proof}
Recall that the Lagrangian angle  $\beta$ of a zero Maslov circle-invariant Lagrangian $L$ agrees with the angle the curve $\gamma=\mu(L)$ in $\mathbb{R}^2$ makes with a fixed axis (see $\S$\ref{sss:graded}).    Since $\tau(L)=0$, by translating $\beta$ by a constant if necessary, we can assume that the almost calibrated condition is given by $\cos\beta>0$ (i.e.~$\Re\Omega|_L>0$).  As $L_0$   is compact there exists some $\epsilon>0$ such that $\cos\beta_0 \geq \epsilon$.

Since the Lagrangian angle satisfies the heat equation along the flow (see e.g.~\cite[Lemma 2.3]{ThomasYau}), the maximum principle 
applied to the compact manifold $L_t$ without boundary shows that the almost calibrated condition is preserved as the minimum of $\cos\beta_t$ is non-decreasing. 
\end{proof}

\begin{remark}
To obtain the same result as Lemma \ref{lem:almost.calibrated} in \cite[Lemma 7.8]{ThomasYau}, the authors needed to use the stability assumption and a stronger initial bound on the Lagrangian angle, together with a much more involved argument.  However, one can see that their proof can be greatly simplified, as here, in the case where the Lagrangians are surfaces.
\end{remark}

Lemma \ref{lem:almost.calibrated} gives two important consequences. The first is the following.

\begin{lemma}\label{lem:log.bounded}
If the curve $\gamma_0$ in $\mathbb{R}^2$ is almost calibrated, then $|\nabla^{\perp}_{\mathbb{R}^3} \log\phi|_{\mathbb{R}^3}^2$ remains bounded as long as the curve does not reach any other singularities of $\phi$.
\end{lemma}

\begin{proof}
Recall that the flow $\gamma$ stays in a compact region in $\mathbb{R}^2$ by Lemma \ref{lem:flow.exist}.  If we take a compact subset $K$ of this region not containing any singularities of $\phi$, then $\phi$ is uniformly bounded above and below away from zero, and so $|\nabla_{\mathbb{R}^3}\log \phi|^2_{\mathbb{R}^3}$ is bounded in $K$.  Therefore, if we assume that $\gamma$ does not reach any singularities $p_i$ of $\phi$ for $i>2$, we need only show that $|\nabla^{\perp}_{\mathbb{R}^3} \log\phi|_{\mathbb{R}^3}^2$ remains bounded at the endpoints $p_1$, $p_2$ of $\gamma$.
Notice by \eqref{eq:phi.MultiTN_2} that the dominant term in $\phi$ near each $p_i$ is rotationally symmetric around $p_i$.  Therefore for $|\nabla^{\perp}_{\mathbb{R}^3}\log\phi|_{\mathbb{R}^3}^2$ to remain bounded we must ensure that $\gamma$ does not wind around $p_1$ or $p_2$.  

Since $\gamma_0$ is almost calibrated, we can choose two rays $\ell_1$ and $\ell_2$ in $\mathbb{R}^2$ emanating from $p_1$ so that $\gamma_0$ does not meet $\ell_1$ and $\ell_2$ except at $p_1$, and the angle between $\ell_1$ and $\ell_2$ is strictly less than $\pi$.  As $\ell_1$ and $\ell_2$ are fixed by the flow \eqref{eq:Mean_Curvature_Flow_E_2} (they define minimal Lagrangians in $X$), they act as barriers, and so $\gamma_t$ must not meet $\ell_1$ and $\ell_2$ except at $p_1$ for all $t$.  Therefore, $\gamma_t$ cannot wind around $p_1$ and $|\nabla^{\perp}_{\mathbb{R}^3}\log\phi|^2_{\mathbb{R}^3}$ must remain bounded near $p_1$ along the flow.  A similar argument works at $p_2$.    
\end{proof}

\begin{remark}[Flow stability of curves] \label{rem:Curve_Stability}
Let $L$ be an embedded, compact, almost calibrated, circle-invariant Lagrangian in $(X,\omega,\Omega)$. 
	Without loss of generality, suppose $\mu(L)=\gamma$ is an embedded arc in the $\mu_3=0$ plane and $\Omega=\omega_1 + i \omega_2$. Suppose that the initial point of $\gamma$ is $(0,0,0)$ and its final point is $(x,y,0)$. We find that
	$$\int_L \Omega= 2 \pi \int_{\gamma} d\mu_1 + i d \mu_2= 2 \pi (x + i y).$$
	Thus, $$e^{i \tau(L)} = \frac{x + i y}{\sqrt{x^2+y^2}}$$ and  hence $\tau(L)$ is the angle  the straight line $\overline{\gamma}$ with the same endpoints as $\gamma$ makes with the $\mu_1$-axis.  Moreover, 
	$$\int_L e^{-i\tau(L)}\Omega = 2\pi \sqrt{x^2+y^2}=2\pi \mathrm{Length(\overline{\gamma})}.$$
Note that, by Proposition \ref{prop:min.spheres}, $\mu^{-1}(\overline{\gamma})$ is the area-minimizer in $[L] \in H_2(M, \mathbb{Z})$ and $2\pi \mathrm{Length(\overline{\gamma})}= \mathrm{Area}(\mu^{-1}(\overline{\gamma}))$.
We also see that $\tau(L)=0$ if and only if the endpoints of $\gamma$ lie on the $\mu_1$-axis.
 
We deduce that, in the circle-invariant setting, the notion of flow stability as in Definition \ref{dfn:flow_stability} can be re-cast as follows. For any decomposition of $\gamma$ as $\gamma_1 \# \gamma_2$, where $\gamma_1,\gamma_2$ are almost calibrated, we let $\overline{\gamma_1}, \overline{\gamma_2}$ denote the straight-lines with the same endpoints as $\gamma_1$ and $\gamma_2$ respectively.  We assume that the endpoints of $\gamma$ are on the $\mu_1$-axis and we let $\beta$, $\overline{\beta}_1$, $\overline{\beta}_2$ denote the angles that $\gamma$, $\overline{\gamma}_1$, $\overline{\gamma}_2$ make with the $\mu_1$-axis respectively.  Then $\gamma$, equivalently $L=\mu^{-1}(\gamma)$, is flow stable if  for all decompositions   $\gamma=\gamma_1\#\gamma_2$ we have  
\begin{itemize}
\item[(a)] $[ \min \lbrace \overline{\beta_1},\overline{\beta}_2 \rbrace , \max \lbrace \overline{\beta_1},\overline{\beta}_2 \rbrace ]\nsubseteq(\inf_{\gamma}\beta,\sup_{\gamma}\beta)$, or
\item[(b)] $\mathrm{Length}(\gamma) < \mathrm{Length}(\overline{\gamma_1}) + \mathrm{Length}(\overline{\gamma_2})$
\end{itemize} 
One can picture condition (a) in Figure \ref{fig_Stability}. Condition (b) should be compared with Remark \ref{rem:Area_Stability}.
\end{remark}

For flow stable initial curves for the flow, we have a very important consequence of the observations in Remark \ref{rem:Curve_Stability}.

\begin{lemma}\label{lem:avoid.sings}  Let $p_i$ denote the singularities of $\phi$ as in Examples \ref{ex:MultiEH} and \ref{ex:MultiTN}, and recall that $p_1,p_2$ are the fixed endpoints of the curve along flow. 
There exists $\delta>0$ so that the flow \eqref{eq:Mean_Curvature_Flow_E_2} of curves $\gamma$ in $\mathbb{R}^2$ starting at a flow stable curve remains outside of $\cup_{i>2}\overline{B_{\delta}(p_i)}$ for all time for which the flow is defined.  Hence, the flow \eqref{eq:Mean_Curvature_Flow_E_2} exists as long as $\phi^{-1}\kappa^2$ remains bounded.
\end{lemma}

\begin{proof}  Let $\gamma_0$ be the initial flow stable curve and let $\overline{\gamma}$ be the straight line between the endpoints $p_1$ and $p_2$ of $\gamma_0$. By the reformulation of flow stability in Remark \ref{rem:Curve_Stability}, we see that the singularities $p_i$ for $i>2$ cannot lie in the interior of the region bounded by $\gamma_0$ and $\overline{\gamma}$. 
	
To see this we argue by contradiction and suppose the existence of one such singularity, say $p_3$, in the interior of the region bounded by $\gamma_0$ and $\overline{\gamma}$. Then $\gamma_0$ could be decomposed as $\ell_1 \# \ell_2$ for some straight lines $\ell_1$ and $\ell_2$ with endpoints $p_1,p_3$ and $p_3,p_2$ respectively. Then $\mathrm{Length}(\ell_1)+ \mathrm{Length}(\ell_2) < \mathrm{Length}(\gamma_0)$ which would contradict the condition (b) in Remark \ref{rem:Curve_Stability}. Furthermore, by the intermediate value theorem, if $\beta_0, \beta_1 , \beta_2$ respectively denotes the gradings associated to $\gamma_0, \ell_1,\ell_2$ we must either have $\sup_{\gamma_0} \beta_0 > \beta_1 > \beta_2 > \inf_{\gamma_0} \beta_0$, or $\sup_{\gamma_0} \beta_0 > \beta_2 > \beta_1 > \inf_{\gamma_0} \beta_0$, and so   condition (a) in  Remark \ref{rem:Curve_Stability} is also violated.  Overall, we obtain our required contradiction to the flow stability of $\gamma_0$. 

Suppose that $\gamma_0$ is flow stable in the sense of condition (a) in Remark \ref{rem:Curve_Stability}.  Since the variation of the angle of $\gamma$ is non-increasing along the flow \eqref{eq:Mean_Curvature_Flow_E_2} by Lemma \ref{lem:almost.calibrated}, $\gamma$ satisfies (a) for all time for which the flow exists.  Then $p_i$ for $i>2$ can never lie in the interior of the region bounded by $\gamma$ and $\overline{\gamma}$, and we deduce that $\gamma$ must remain outside of a compact region containing the $p_i$ for $i>2$.

Suppose instead that $\gamma_0$ is flow stable in the sense of condition (b) in Remark \ref{rem:Curve_Stability}.  Lagrangian mean curvature flow is the gradient flow for the area of the Lagrangian $\mu^{-1}(\gamma)$, so its area is decreasing along the flow and hence, by \eqref{eq:Area_Length}, the length of $\gamma$ is decreasing along \eqref{eq:Mean_Curvature_Flow_E_2}.  Therefore, $\gamma$ satisfies (b) for all time for which the flow exists, and we again conclude as before that $\gamma$ must lie outside of $\cup_{i>2}\overline{B_{\delta}(p_i)}$ for some $\delta>0$.

Since we have established that the flow $\gamma$ never reaches any other singularities of $\phi$, the final statement then follows from Lemmas \ref{lem:flow.exist} and \ref{lem:log.bounded}.
\end{proof}

\begin{remark}
Notice that if we bound the curvature $\kappa$ then we bound $\phi^{-1}\kappa^2$ since $\phi^{-1}$ is bounded in the compact region in which the flow is taking place.  However, it is important to know that we require $\phi^{-1}\kappa^2$ to blow up, not just $\kappa$ to blow up, for a singularity to form.  This is a key point where we differ from \cite{ThomasYau}, who only assume $\kappa$ blows up at the singularity.
\end{remark}

We now need to understand the possible blow-up behaviour of the flow.  Suppose we parametrise the flowing curves $\gamma=\gamma_t(s)$ by arclength so that $s=0$ at $p_1$.  Let $\kappa=\kappa_t(s)$ denote the curvature of $\gamma$.  Suppose   the flow has a singularity at time $T$.  Then there must exist a sequence of spacetime points $(s_i,t_i)$ with $t_i\to T$ as $i\to\infty$ such that if $x_i=\gamma(s_i)$ and
$\lambda_i=\phi(x_i)^{-1/2}|\kappa_{t_i}(s_i)|$ defined by
$$\lambda_i=\max\{\phi(\gamma_t(s))^{-1/2}|\kappa_t(s)|:s\geq 0, t\leq t_i\}$$
satisfies $\lambda_i\to\infty$ as $i\to\infty$.  By a diagonal argument or taking a subsequence we can assume further that $(s_i,t_i)$ also represents a maximum for $|\kappa_t(s)|$ for $t\leq t_i$.

We now perform the standard blow-up analysis and define curves $\hat{\gamma}^i_t$ by
$$\hat{\gamma}^i_t=\lambda_i(\gamma_{t_i+\lambda_i^{-2}t}-x_i)$$
for $t\in [-\lambda_i^2t_i,\lambda_i^2(T-t_i))$.  The flow cannot have a Type I singularity, since any such singularity in our setting would be modelled by a smooth self-shrinker, but by \cite{Wang} any non-trivial smooth self-shrinker would have angle variation more than $\pi$ (in fact, at least $2\pi$), and so is excluded by Lemma \ref{lem:almost.calibrated}.  Thus, we have must have a Type II singularity, i.e.~$\lambda_i^2(T-t_i)\to\infty$.

The curves $\hat{\gamma}^i$ satisfy the equation
\begin{equation}\label{eq:limit.flow}
\frac{\partial}{\partial t}\hat{\gamma}^i=\left(m+\sum_{j=1}^k\frac{1}{|\lambda_i^{-1}\hat{\gamma}^i+x_i-p_j|}\right)^{-1}(\hat{\gamma}^i)'',
\end{equation}
using the notation of Examples \ref{ex:MultiEH} and \ref{ex:MultiTN}.

Since the flowing curves are compact and remain so, we have that the points $x_i=\gamma_{t_i}(s_i)$ admit a convergent subsequence converging to a point $x_{\infty}$, say.
There are two cases: when the blow-up happens in the interior of the curve (i.e.~$x_\infty\notin\{p_1,p_2\}$) or at $p_1$ or $p_2$.  We begin with the first.

\begin{lemma}\label{lem:p1p2}
In the setting above, if there is a singularity at time $T$, then $x_{\infty}$ must be $p_1$ or $p_2$. 
\end{lemma}

\begin{proof}
Suppose that $x_\infty$ is not $p_1$ or $p_2$.  Then, also using Lemma \ref{lem:avoid.sings}, there exists some $\delta>0$ so that for $i$ sufficiently large, we have that $B_{1/\sqrt{\lambda_i}}(x_i)$ is disjoint from $\cup_{j=1}^k\overline{B_{\delta}(p_j)}$.   Restricting to the balls $B_{1/\sqrt{\lambda_i}}(x_i)$, our curves $\hat{\gamma}^i_t$ are defined on the ball of radius $B_{\sqrt{\lambda_i}}(0)$.  Therefore, $|\lambda_i^{-1}\hat{\gamma}^i_t|\to 0$ on this sequence of balls, and so (passing to a subsequence as necessary), from \eqref{eq:limit.flow} we will obtain an eternal solution $\hat{\gamma}^{\infty}_t$ to
$$
\frac{\partial}{\partial t}\hat{\gamma}^{\infty}_t=\left(m+\sum_{j=1}^k\frac{1}{|x_{\infty}-p_j|}\right)^{-1}(\hat{\gamma}^{\infty})''=c(\hat{\gamma}^{\infty}_t)''$$
on $\mathbb{R}^2$, where $c>0$ is constant.
Since $\hat{\gamma}^{\infty}$ satisfies the curve shortening flow on $\mathbb{R}^2$ up to a constant reparametrization of time,  the blow-up analysis  for the curve shortening flow is valid (see e.g.~\cite{GageHamilton}) and shows that a singularity can only develop if the variation of the angle of $\hat{\gamma}^{\infty}_t$ is at least $\pi$.  
  This possibility is excluded by Lemma \ref{lem:almost.calibrated} and so the result follows.
\end{proof}

Now we know, by Lemma \ref{lem:p1p2}, that any singularity of the flow must occur, without loss of generality, at $p_1$.  By changing coordinates we can assume  $p_1=0$. Therefore, our points $x_i$ must satisfy $|x_i|\to 0$ as $i\to\infty$, and  thus $\phi(x_i)^{-1}\approx |x_i|$ for $i$ large.  However, we also have $\phi(x_i)^{-1}\kappa_{t_i}(s_i)^2\to\infty$.  There are therefore two possibilities: there is a subsequence so that $\phi(x_i)^{-1}|\kappa_{t_i}(s_i)|\to\infty$ (i.e.~the points move ``slowly'' towards the singularity) or $\phi(x_i)^{-1}|\kappa_{t_i}(s_i)|$ is bounded (so the points $x_i$ move ``quickly'' towards the singularity).  We now exclude each case in turn.

\begin{lemma}\label{lem:blow.up.case.1} 
In the setting above, there does not exist a subsequence so that $\phi(x_i)^{-1}|\kappa_{t_i}(s_i)|\to\infty$.  
\end{lemma}

\begin{proof}
Suppose not.  Then consider the balls $B(x_i,\frac{1}{2}|x_i|)$ for $i$ sufficiently large so that they do not contain $p_1=0$ (which is clear by construction) or $p_2$.   Translating $x_i$ to the origin and rescaling by $|\kappa_{t_i}(s_i)|$ (not by $\lambda_i$), we obtain balls $B(0,\frac{1}{2}|x_i\kappa_{t_i}(s_i)|)$ which will cover $\mathbb{R}^2$ as $i\to\infty$ since $|x_i|\approx\phi(x_i)^{-1}$.  Moreover, using $|\kappa_{t_i}(s_i)|$ in our blow-up sequence instead of $\phi(x_i)^{-1/2}|\kappa_{t_i}(s_i)|$ we obtain a sequence of curves $\hat{\gamma}^i$ with curvature $1$ at the origin for all $i$.  Taking the limit as $i\to\infty$ then gives a smooth curve $\hat{\gamma}^{\infty}$ in $\mathbb{R}^2$ with curvature $1$ at $0$.  However, the corresponding blow-up sequence $\hat{L}^i$ of Lagrangian mean curvature flows (using $|\kappa_{t_i}(s_i)|$) will have a limit $\hat{L}^{\infty}$ which is a plane, since we are blowing up at a faster rate (since $\phi(x_i)^{-1/2}\to 0$ as $i\to\infty$) than the norm of the second fundamental form, recalling the proof of Lemma \ref{prop:curve.flow.blow.up}.  This contradicts the fact that the projection of $\hat{L}^{\infty}$ to $\mathbb{R}^2$, which is $\hat{\gamma}^{\infty}$, is not a straight line.
\end{proof}

\begin{lemma}\label{lem:blow.up.case.2}
In the notation above, we cannot have that $\phi(x_i)^{-1}|\kappa_{t_i}(s_i)|$ is bounded.
\end{lemma}

\begin{proof}
If $\phi(x_i)^{-1}|\kappa_{t_i}(s_i)|$ is bounded then for $i$ sufficiently large, there is some constant $A$ so that 
\begin{equation}\label{eq:xi.bound}
|x_i|<\max\left\{A|\kappa_{t_i}(s_i)|^{-1}
,1\right\},
\end{equation}
since $|x_i|\approx \phi(x_i)^{-1}$.  
We may then consider the balls $B(x_i,A\phi(x_i)|\kappa_{t_i}(s_i)|^{-1})$ which will contain the origin since, by \eqref{eq:xi.bound}, $|x_i|<1$ whereas $A\phi(x_i)|\kappa_{t_i}(s_i)|^{-1}>1$ for $i$ sufficiently large.  Rescaling by $\lambda_i=\phi(x_i)^{-1/2}|\kappa_{t_i}(s_i)|$ and translating by $x_i$, we obtain balls $B(0,A\phi(x_i)^{1/2})$ which will cover $\mathbb{R}^2$ as $i\to\infty$. Moreover, the image $y_i$ of $0$ under this transformation will lie at distance $\phi(x_i)^{-1/2}|x_i||\kappa_{t_i}(s_i)|$ from $0$, which tends to $0$ as $i\to\infty$ by \eqref{eq:xi.bound}. 

By construction, above the origin the norm of the second fundamental form of the blow-up sequence $\hat{L}^i=\mu^{-1}(\hat{\gamma}^i)$ of Lagrangian mean curvature flows is $1$.  Sending $i\to\infty$ we obtain, as a limit, a non-trivial, smooth, embedded, almost calibrated, ancient solution $\hat{L}^{\infty}$ to Lagrangian mean curvature flow in $\mathbb{R}^4$, which projects to a curve $\hat{\gamma}^{\infty}$ through $0$.  Moreover,   $\mu^{-1}(y_i)$ is a point in $\hat{L}^i$ for all $i$ and $|y_i|\to 0$ as $i\to\infty$. Hence, after rescaling, $\mu$ is the radially extended Hopf fibration, as in Example \ref{ex:flat}, and $\mu^{-1}(0)\in \hat{L}^{\infty}$ is a point. 
Thus, $\hat{L}^{\infty}$ cannot topologically be a cylinder $S^1\times\mathbb{R}$ as its projection passes through the origin. We also know that $\hat{L}^{\infty}$ must be an exact Lagrangian by \cite{LambertLotaySchulze}.

We now observe that, since $\hat{\gamma}^{\infty}$ is almost calibrated, it must be graphical over a straight line, and its blow-down must be at most two (multiplicity one) rays emanating from the origin.  In particular, the blow-down cannot be a multiplicity two ray by the almost calibrated condition. This means that the blow-down of $\hat{L}^{\infty}$ is either a single plane or a pair of planes intersecting transversely at a point.  The exact, almost calibrated, ancient solutions to Lagrangian mean curvature flow in $\mathbb{R}^4$ whose blow-downs are planes were classified in \cite{LambertLotaySchulze}.  The only such ancient solutions whose blow-down is a single plane is the plane itself, which is ruled out because $\hat{L}^{\infty}$ is non-trivial.  If instead the blow-down is a pair of transverse planes, the only possibilities are the Lawlor neck (a special Lagrangian described, for example, in \cite{LambertLotaySchulze}) or the pair of planes themselves.  The pair of planes is again excluded because $\hat{L}^{\infty}$ is non-trivial, and the Lawlor neck is excluded because it is topologically $S^1\times\mathbb{R}$.

We have therefore reached our desired contradiction, and the result follows.
\end{proof}

Combining our lemmas gives our long-time existence and convergence result.

\begin{corollary}\label{cor:flow.exists.all.time}
The flow \eqref{eq:Mean_Curvature_Flow_E_2} starting at a flow stable curve with endpoints $p_1$ and $p_2$ exists for all time and converges smoothly to the straight line connecting $p_1$ and $p_2$.
\end{corollary}

\begin{proof}
The long-time existence of the flow \eqref{eq:Mean_Curvature_Flow_E_2} of curves $\gamma$ follows directly from Lemmas \ref{lem:avoid.sings}--\ref{lem:blow.up.case.2}.

It is then a standard argument, as one may see in \cite{ThomasYau}, using the fact that the Lagrangian angle satisfies the heat equation and that circle-invariance is preserved along Lagrangian mean curvature flow, to deduce that the Lagrangian mean curvature flow $L=\mu^{-1}(\gamma)$ converges smoothly to a circle-invariant special Lagrangian $L_{\infty}=\mu^{-1}(\gamma_{\infty})$.  Theorem \ref{thm:SL.stable} then gives that $\gamma_{\infty}$ is the straight line between $p_1$ and $p_2$.
\end{proof}

Corollary \ref{cor:flow.exists.all.time} completes the proof of Theorem \ref{thm:LMCF.stable}.

\appendix

\section{Curvature of special Lagrangian spheres}\label{sec:Positive_Curvature}\label{app:spheres}

In this appendix, we prove the following sufficient condition for positive curvature of the induced metric of a special Lagrangian sphere in a hyperk\"ahler 4-manifold $X$ given by the Gibbons--Hawking ansatz, with harmonic function $\phi$ as in Examples \ref{ex:MultiEH} or \ref{ex:MultiTN} with collinear singularities $p_1,\ldots,p_k$.

\begin{proposition}\label{prop:Positive_Curvature}
Let $\gamma \subseteq \mathbb{R}^3$ be a straight line between $p_1$ and $p_2$, let $2d=\text{\emph{Length}}(\gamma)$, and let $q$ by the midpoint of $\gamma$. Suppose that the Euclidean distance from $p_i$ to $q$ is strictly greater than $s d$ for $s\geq \max \lbrace 4, \sqrt{(k-2)/2} \rbrace$ and all $i>2$. 
	Then $\mu^{-1}(\gamma)$ is a smooth minimal $2$-sphere and its induced metric  is positively curved.
\end{proposition}

Up to a rigid motion of $\mathbb{R}^3$ we may assume that 
$\gamma$ is a straight line connecting $p_{\pm}=(0,0,\pm a)$ and we re-label the remaining singularities of $\phi$ by $q_1, \ldots , q_{k-2}$. Then, using $r_l$ to denote the Euclidean distance to the point $q_l$, at the points of $\gamma$, where $\mu_1=0=\mu_2$ and $\mu_3 \in (-a,a)$, we find that
\begin{equation}\label{eq:tilde.phi}
\phi = m + \sum_{l=1}^{k-2} \frac{1}{2r_l} + \frac{1}{2(\mu_3+a)} - \frac{1}{2(\mu_3-a)} = \tilde{\phi} + \frac{a}{a^2 - \mu_3^2}. 
\end{equation}

\begin{lemma}\label{lem:phi_Slag_Convex}
	On $\gamma$, the function $\tilde{\phi}= m + \sum_{l=1}^{k-2} \frac{1}{2r_l}$ satisfies
	$$ | \del_{\mu_3} \tilde{\phi} | \leq \sum_{l=1}^{k-2} \frac{1}{2r_l^2} , 
	\quad\text{and}\quad
	\del^2_{\mu_3} \tilde{\phi} = \sum_{l=1}^{k-2}  \frac{1}{r_l^3}>0 .$$
\end{lemma}
\begin{proof}
	Using $\del_{\mu_3} r_l=\tfrac{\mu_3-q_l^3}{r_l}$, where $q_l^3$ denotes the $\mu_3$-component of $q_l$, we find from \eqref{eq:tilde.phi} that
\begin{equation*}
\del_{\mu_3} \tilde{\phi} = - \sum_{l=1}^{k-2} \frac{\del_{\mu_3} r_l}{2r_l^2}  =  - \sum_{l=1}^{k-2} \frac{\mu_3-q_l^3}{2r_l^3}.
\end{equation*}
Given that $r_l = \sqrt{(q_l^1)^1 + (q_l^2)^2 + (\mu_3 - q_l^3)^2} \geq |\mu_3 - q^3_l|$ on $\gamma$, we deduce that
\begin{equation*}
| \del_{\mu_3} \tilde{\phi} | \leq  \sum_{l=1}^{k-2} \frac{|\mu_3-q_l^3|}{2r_l^3} \leq \sum_{l=1}^{k-2} \frac{1}{2r_l^2}
\end{equation*}	
as claimed. As for the second derivatives of $\tilde{\phi}$, we find 
	\begin{equation*}
	\del^2_{\mu_3} \tilde{\phi}  = - \sum_{l=1}^{k-2} \left( \frac{1}{2r_l^3}  - 3 \frac{\mu_3-q^3_l}{2r_l^4} \del_{\mu_3} r_l\right) 
	 = - \sum_{l=1}^{k-2}  \frac{1}{2r_l^3}  +  \frac{3}{2}  \sum_{l=1}^{k-2} \frac{(\mu_3-q^3_l)^2}{r_l^5}   = \sum_{l=1}^{k-2}  \frac{1}{r_l^3},
	\end{equation*}
	which gives the result.  
\end{proof}

An orthonormal coframing of $\mu^{-1}(\gamma)$ for the induced metric is given by
$$f^0=\phi^{-1/2} \eta , \ \ f^1=\phi^{1/2}d\mu_3.$$
The Cartan structure equations yield that the connection $1$-form $\alpha$ is determined by
$$df^0=-\alpha \wedge f^1 , \ \ df^1 = \alpha \wedge f^0,$$
and the curvature $2$-form by
$$d\alpha = K f^0 \wedge f^1,$$
where $K$ is the Gauss curvature of $\mu^{-1}(\gamma)$. We compute $$df^0=-\tfrac{1}{2}\phi^{-3/2} (\del_{\mu_3}\phi) d\mu_3 \wedge \eta = \tfrac{1}{2}\phi^{-2} (\del_{\mu_3}\phi ) \eta \wedge f^1\quad\text{and}\quad df^1=0.$$ We deduce that
$$\alpha = -\frac{1}{2\phi^{2}} \frac{\del \phi}{\del \mu_3} \eta.$$
Thus
$$d \alpha =- \frac{\del}{\del \mu_3} \left( \frac{1}{2\phi^{2}} \frac{\del \phi}{\del \mu_3} \right) d\mu_3 \wedge \eta = - \frac{\del^2}{\del \mu_3^2} \left( \frac{1}{2\phi} \right) f^0 \wedge f^1,$$
and so the Gaussian curvature of $\mu^{-1}(\gamma)$ is given by
\begin{equation}\label{eq:Keq}
K=- \frac{\del^2}{\del \mu_3^2} \left( \frac{1}{2\phi} \right) .
\end{equation}
Using our formula \eqref{eq:tilde.phi} for $\phi$ we find that
$$\phi^{-1}= \frac{a^2-\mu_3^2}{a + \tilde{\phi} (a^2-\mu_3^2)}.$$
From this we compute
\begin{align*}
\frac{\del}{\del \mu_3} \left( \frac{1}{\phi} \right) & = \frac{-2\mu_3 (a + \tilde{\phi} (a^2-\mu_3^2)) +2\tilde{\phi}\mu_3 (a^2-\mu_3^2) - (\del_{\mu_3}\phi) (a^2-\mu_3^2)^2 }{(a + \tilde{\phi} (a^2-\mu_3^2))^2} \\
& = \frac{-2a\mu_3 - (\del_{\mu_3}\tilde{\phi}) (a^2-\mu_3^2)^2}{(a + \tilde{\phi} (a^2-\mu_3^2))^2}
\end{align*}
and 
\begin{align*}
\frac{\del^2}{\del \mu_3^2} \left( \frac{1}{\phi} \right) & = \frac{(-2a- \del_{\mu_3} ( (\del_{\mu_3}\tilde{\phi} ) (a^2-\mu_3^2)^2) )(a + \tilde{\phi} (a^2-\mu_3^2))^2  }{(a + \tilde{\phi} (a^2-\mu_3^2))^4} \\
& \ \ \ \ + \frac{ (4 a \mu_3 + 2 (\del_{\mu_3}\phi) (a^2-\mu_3^2)^2 )( -2 \tilde{\phi} \mu_3 +(\del_{\mu_3}\tilde{\phi})(a^2-\mu_3^2) )(a + \tilde{\phi} (a^2-\mu_3^2)) }{(a + \tilde{\phi} (a^2-\mu_3^2))^4}\displaybreak[0]\\
& = \frac{(-2a- \del_{\mu_3} ( (\del_{\mu_3}\tilde{\phi} ) (a^2-\mu_3^2)^2) )(a + \tilde{\phi} (a^2-\mu_3^2)) }{(a + \tilde{\phi} (a^2-\mu_3^2))^3} \\
& \ \ \ \ + \frac{ (4 a \mu_3 + 2 (\del_{\mu_3}\tilde{\phi}) (a^2-\mu_3^2)^2 )( -2 \tilde{\phi} \mu_3 +(\del_{\mu_3}\tilde{\phi}) (a^2-\mu_3^2 )) }{(a + \tilde{\phi} (a^2-\mu_3^2))^3}\\
& = \frac{-2a^2  -2a \tilde{\phi} (a^2-\mu_3^2) - a \del_{\mu_3} ( (\del_{\mu_3}\tilde{\phi} ) (a^2-\mu_3^2)^2)  -  \tilde{\phi} (a^2-\mu_3^2) \del_{\mu_3} ( (\del_{\mu_3}\tilde{\phi} ) (a^2-\mu_3^2)^2) }{(a + \tilde{\phi} (a^2-\mu_3^2))^3} \\
& \ \ \ \ + \frac{- 8 a \tilde{\phi} \mu_3^2 +4a\mu_3(a^2-\mu_3^2) \del_{\mu_3} \tilde{\phi} -4\mu_3 \tilde{\phi} (\del_{\mu_3} \tilde{\phi}) (a^2-\mu_3^2)^2 + 2 (\del_{\mu_3}\tilde{\phi})^2 (a^2-\mu_3^2)^3 }{(a + \tilde{\phi} (a^2-\mu_3^2))^3} \displaybreak[0]\\
& = -\frac{2a^2 +2a \tilde{\phi} (a^2-\mu_3^2) + a   (\del_{\mu_3}^2 \tilde{\phi} ) (a^2-\mu_3^2)^2 +  \tilde{\phi}  (\del_{\mu_3}^2\tilde{\phi} ) (a^2-\mu_3^2)^3 + 8 a \tilde{\phi} \mu_3^2 }{(a + \tilde{\phi} (a^2-\mu_3^2))^3 } \\
& \ \ \ \ + \frac{ 2 (\del_{\mu_3}\tilde{\phi})^2 (a^2-\mu_3^2)^3+8a\mu_3 (\del_{\mu_3} \tilde{\phi}) (a^2-\mu_3^2)}{(a + \tilde{\phi} (a^2-\mu_3^2))^3}.\displaybreak[0]
\end{align*}
Hence, 
$$\frac{\del^2}{\del \mu_3^2} \left( \frac{1}{\phi} \right) = M+N$$
where
\begin{equation}\label{eq:Neq}
N=-\frac{2a^2 +2a \tilde{\phi} (a^2-\mu_3^2) + 8 a \tilde{\phi} \mu_3^2  + a   (\del_{\mu_3}^2 \tilde{\phi} ) (a^2-\mu_3^2)^2 +  \tilde{\phi}  (\del_{\mu_3}^2\tilde{\phi} ) (a^2-\mu_3^2)^3}{(a + \tilde{\phi} (a^2-\mu_3^2))^3 }<0
\end{equation}
by Lemma \ref{lem:phi_Slag_Convex} and the fact that $\mu_3\in(-a,a)$, and the mixed sign term $M$ is determined by $(a + \tilde{\phi} (a^2-\mu_3^2))^3M = m_1 + m_2$ where
\begin{equation}\label{eq:m1m2}
m_1= 2 (\del_{\mu_3}\tilde{\phi})^2 (a^2-\mu_3^2)^3 , \ \ \ m_2 =8a\mu_3 (\del_{\mu_3} \tilde{\phi}) (a^2-\mu_3^2).
\end{equation}
To prove Proposition \ref{prop:Positive_Curvature} we need to show that $N+M<0$ by \eqref{eq:Keq}.

Note that the assumptions in Proposition \ref{prop:Positive_Curvature} imply that $r_l>sa$.  For $m_1$, we see that
$$|\del_{\mu_3} \tilde{\phi}| \leq \sum_{l=1}^{k-2} \frac{1}{2r_l^2}  \leq \max_{l \in \lbrace 1 , \ldots , k-2 \rbrace } \frac{1}{r_l^2} \frac{k-2}{2} \leq \frac{k-2}{2s^2a^2}$$
by Lemma \ref{lem:phi_Slag_Convex}, and hence, by \eqref{eq:m1m2},
$$0\leq m_1 \leq 2a^2\left( \frac{k-2}{2s^2} \right)^2 .$$
Since we are assuming that $s\geq \sqrt{(k-2)/2}$, the term given by $m_1$ in $M$ can be absorbed in the first term in $N$ in \eqref{eq:Neq} so that the sum remains negative.

To analyze $m_2$ in \eqref{eq:m1m2}, we observe by Lemma \ref{lem:phi_Slag_Convex} that
$$|\mu_3(\del_{\mu_3} \tilde{\phi})| \leq \sum_{l=1}^{k-2} \frac{|\mu_3|}{2r_l^2}  \leq \max_{l \in \lbrace 1 , \ldots , k-2 \rbrace } \frac{|\mu_3|}{r_l} \tilde{\phi}  \leq \frac{|\mu_3|\tilde{\phi}}{sa} \leq \frac{\tilde{\phi}}{ s},$$
where we have used $r_l >s a$ and $|\mu_3|\leq a$. Then, inserting this inequality in $m_2$ in \eqref{eq:m1m2} gives:
$$|m_2| \leq 8a|\mu_3| (\del_{\mu_3} \tilde{\phi}) (a^2-\mu_3^2) \leq 8s^{-1} a \tilde{\phi} (a^2-\mu_3^2).$$
As $s\geq 4$ by assumption, we see that the term in $M$ given by $m_2$ can be absorbed into the second term appearing in $N$ in \eqref{eq:Neq} so that the sum remains negative.

We conclude that 
$\del^2_{\mu_3} \phi^{-1}=M+N < 0$ and so,  by \eqref{eq:Keq} , $K>0$
and Proposition \ref{prop:Positive_Curvature} is proved.

\section{Hessian of circle-invariant functions}\label{app:hessian}

Let $X$ be   a hyperk\"ahler 4-manifold given by the Gibbons--Hawking ansatz and let $f:X \rightarrow \mathbb{R}$ be a smooth $\U(1)$-invariant function. The Hessian of $f$ at a point $p$ is given by $\Hess(f) =\nabla df$, so if $\lbrace e_{\mu} \rbrace_{\mu=0}^3$ denotes an orthonormal frame on $X$ as in Lemma \ref{lem:Covariant_Derivatives}, we have
\begin{equation}\label{eq:hessf}
\Hess(f) (e_\mu , e_\nu) = \nabla_{e_\mu} \nabla_{e_\nu} f - (\nabla_{e_\mu} e_\nu ) \cdot f 
\end{equation}
and $e_0\cdot f=0$ by the $\U(1)$-invariance of $f$.
From the formulae for $\nabla_{e_\mu} e_\nu$ given in Lemma \ref{lem:Covariant_Derivatives}, we find (recalling the permutation symbol $\epsilon_{ijk}$ and the fact that we are using summation convention):
\begin{align}
(\nabla_{e_0} e_0) \cdot f & = \frac{1}{2\phi^{2}}   
\frac{\partial \phi}{\partial \mu_i} \frac{\partial f}{\partial \mu_i},    \label{eq:hessf0.1}\\
(\nabla_{e_i} e_0) \cdot f &  = \frac{1}{2\phi^{2}} \epsilon_{ijk} \frac{\partial \phi}{\partial \mu_j}  \frac{\partial f}{\partial \mu_k}=(\nabla_{e_0} e_i) \cdot f ,  \label{eq:hessf0.2}\\
(\nabla_{e_i} e_j) \cdot f & = \frac{1}{2\phi^{2}} \left( \frac{\partial f}{\partial \mu_i} \frac{\partial \phi}{\partial \mu_j}  - \delta_{ij} \frac{\partial \phi}{\partial \mu_k} \frac{\partial f}{\partial \mu_k} \right).\label{eq:hessf0.4}
\end{align}
We may also compute
\begin{align}\label{eq:hessf0.5}
\nabla_{e_i} \nabla_{e_j} f  = \nabla_{e_i} \frac{1}{\phi^{1/2}} \frac{\partial f}{\partial \mu_j}  =   \frac{1}{\phi } \frac{\partial^2 f}{\partial \mu_i \partial \mu_j}- \frac{1}{2\phi^{2}}  \frac{\partial \phi}{\partial \mu_i} \frac{\partial f}{\partial \mu_j}.
\end{align}
Moreover, since both $f$ and $\phi$ are $\U(1)$-invariant, we have that \begin{equation}\label{eq:hessf0.6}
\nabla_{e_0} \nabla_{e_\mu} f = 0 = \nabla_{e_\mu} \nabla_{e_0}f.
\end{equation}
Inserting \eqref{eq:hessf0.1}--\eqref{eq:hessf0.6} in \eqref{eq:hessf} yields:
\begin{align}
\Hess(f)_{00} & = -\frac{1}{2\phi^{2}} 
\frac{\partial \phi}{\partial \mu_i} \frac{\partial f}{\partial \mu_i}, \label{eq:hessf.1}\\
\Hess(f)_{i0} & = - \frac{1}{2\phi^{2}} \epsilon_{ijk} \frac{\partial \phi}{\partial \mu_j}  \frac{\partial f}{\partial \mu_k}=\Hess(f)_{0i}, \label{eq:hessf.2}\\
\Hess(f)_{ij} & =  \frac{1}{\phi } \frac{\partial^2 f}{\partial \mu_i \partial \mu_j} +\frac{ \delta_{ij}}{2\phi^{2}} \frac{\partial \phi}{\partial \mu_k} \frac{\partial f}{\partial \mu_k}  - \frac{1}{2\phi^{2}} \left( \frac{\partial f}{\partial \mu_i} \frac{\partial \phi}{\partial \mu_j} +  \frac{\partial \phi}{\partial \mu_i} \frac{\partial f}{\partial \mu_j} \right) \label{eq:hessf.3}.
\end{align}

In particular, if we let $f=r^2=\mu_1^2+\mu_2^2+\mu_3^2$ then \eqref{eq:hessf.1}--\eqref{eq:hessf.3} give:
\begin{align}
\Hess(r^2)_{00}&=-\frac{\mu_i}{\phi^2}\frac{\partial\phi}{\partial\mu_i},\label{eq:hessr2.1}\\
\Hess(r^2)_{i0}&=\epsilon_{ijk}\frac{\mu_j}{\phi^2}\frac{\partial\phi}{\partial \mu_k}=\Hess(r^2)_{0i},\label{eq:hessr2.2}\\
\Hess(r^2)_{ij}&=\delta_{ij}\left(\frac{2}{\phi}+\frac{\mu_k}{\phi^2}\frac{\partial\phi}{\partial\mu_k}\right)-\frac{1}{\phi^2}\left(\mu_i\frac{\partial\phi}{\partial\mu_j}+\mu_j\frac{\partial\phi}{\partial\mu_i}\right).\label{eq:hessr2.3}
\end{align}

As an application of this computation we give the following result.

\begin{example}[No compact minimal submanifolds in Euclidean or Taub--NUT $\mathbb{R}^4$]\label{ex:no.min.TN}
	Let $\phi$ be as in \eqref{eq:phi.flat} or \eqref{eq:phi.TN}, which gives Euclidean or Taub--NUT $\mathbb{R}^4$.  Then 
	$$\frac{\partial\phi}{\partial\mu_i}=-\frac{\mu_i}{2 r^3}.$$
	Inserting this in the equations \eqref{eq:hessr2.1}--\eqref{eq:hessr2.3} for the Hessian of $r^2$ gives:
	\begin{align*}
	\Hess(r^2)_{00}&=\frac{1}{2r\phi^2},\\
	\Hess(r^2)_{i0}&=-\epsilon_{ijk}\frac{\mu_j\mu_k}{2r^3\phi^2}=0=\Hess(r^2)_{0i},\\
	\Hess(r^2)_{ij}&=\delta_{ij}\left(\frac{2}{\phi}-\frac{1}{2r\phi^2}\right)+\frac{\mu_i\mu_j}{r^3\phi^2}=\frac{\delta_{ij}}{2r\phi^2}(4r\phi-1)+\frac{\mu_i\mu_j}{r^3\phi^2}
	\end{align*}
	One may now explicitly compute that $\Hess(r^2)$ is a positive definite matrix for $r>0$, and so $r^2$ is a strictly convex function for $r>0$.  Hence, there are no compact minimal submanifolds (including geodesics) in Euclidean or Taub--NUT $\mathbb{R}^4$.  
\end{example}


\end{document}